\documentclass[letterpaper,12pt,oneside,reqno]{amsart}

\usepackage[sorting=nyt,style=alphabetic,backend=bibtex,hyperref=true,doi=false,maxbibnames=9,maxcitenames=3]{biblatex}
\makeatletter
\def\blx@maxline{77}
\makeatother
\usepackage{amsmath,amssymb,amsthm,amsfonts}
\usepackage{enumerate}
\usepackage{array}
\usepackage{graphicx}
\usepackage[mathscr]{euscript}
\usepackage{color,tikz}
\usetikzlibrary{shapes,arrows,positioning,decorations.markings,calc}
\usepackage[DIV12]{typearea}
\usepackage{adjustbox}
\allowdisplaybreaks
\numberwithin{equation}{section}
\usepackage{upgreek}
\usepackage{hyperref}
\usepackage{cleveref}

\newcommand{\Z}{\mathbb{Z}}
\newcommand{\C}{\mathbb{C}}
\newcommand{\R}{\mathbb{R}}
\newcommand{\HH}{\mathbb{H}}
\renewcommand{\i}{\mathbf{i}}
\DeclareMathOperator{\PP}{\mathbb{P}}

\newcommand{\al}{\alpha}

\newcommand{\la}{\lambda}

\newcommand{\be}{\beta}
\newcommand{\G}{\Gamma}

\newcommand{\D}{\boldsymbol\Delta}

\newcommand{\z}{\mathsf{z}}
\newcommand{\w}{\mathsf{w}}

\renewcommand{\Re}{\mathop{\mathrm{Re}}}
\renewcommand{\Im}{\mathop{\mathrm{Im}}}

\newcommand{\A}{\mathsf{A}}

\newcommand{\mm}{\sqrt M}

\newcommand{\dd}{\mathsf{D}}
\newcommand{\upd}{\uprho^\bullet}
\newcommand{\lod}{\uprho_\bullet}

\newcommand{\QQ}{\mathsf{Q}}

\newcommand{\W}[1]{\mathbb{W}^{#1}}
\newcommand{\aind}{{\upalpha}}
\newcommand{\bind}{{\upbeta}}
\newcommand{\ch}{\boldsymbol{\upchi}}
\newcommand{\fl}[1]{\lfloor{#1}\rfloor}
\newcommand{\V}{\mathsf{V}}

\newcommand{\Sfin}{\mathsf{S}}
\newcommand{\zfin}{\z_{\textnormal{c}}}
\newcommand{\zlim}{\z_{*}}
\newcommand{\ulim}{\mathsf{u}_{*}}
\newcommand{\uvar}{\mathsf{u}}
\newcommand{\zbfin}{\bar\z_{\textnormal{c}}}
\newcommand{\pv}{\mathrm{p.v.}}

\newcommand{\m}{\mathcal{M}}
\newcommand{\CC}{\mathcal{C}}
\newcommand{\RRR}{\mathsf{R}}

\newcommand{\CCC}{\mathsf{d}}

\newcommand{\NRandom}{\mathbf{N}}

\newcommand{\eps}{\varepsilon}
\newcommand{\deps}{\varepsilon}

\newcommand{\TNPower}{\upeta}

\newtheorem{proposition}{Proposition}[section]
\newtheorem{lemma}[proposition]{Lemma}

\newtheorem{theorem}[proposition]{Theorem}

\theoremstyle{definition}
\newtheorem{definition}[proposition]{Definition}
\newtheorem{remark}[proposition]{Remark}
\newtheorem{assumption}{Assumption}

\begin{document}

\title[Universality of local statistics for noncolliding random walks]
{Universality of local statistics\\ for noncolliding random walks}

\author{Vadim Gorin}
\address{V. Gorin,
Massachusetts Institute of Technology, Department of Mathematics,
77 Massachu\-setts Avenue, Cambridge, MA 02139-4307, USA,
and Institute for Information Transmission Problems,
Bolshoy Karetny per. 19, Moscow, 127994, Russia}
\email{vadicgor@gmail.com}

\author{Leonid Petrov}
\address{L. Petrov, University of Virginia, Department of Mathematics,
141 Cabell Drive, Kerchof Hall,
P.O. Box 400137,
Charlottesville, VA 22904, USA,
and Institute for Information Transmission Problems,
Bolshoy Karetny per. 19, Moscow, 127994, Russia}
\email{lenia.petrov@gmail.com}

\begin{abstract}
    We consider the $N$-particle noncolliding Bernoulli random walk --- a discrete time Markov process in $\mathbb{Z}^{N}$ obtained from a collection of $N$ independent simple random walks with steps $\in\{0,1\}$ by conditioning that they never collide. We study the asymptotic behavior of local statistics of this process started from an arbitrary initial configuration on short times $T\ll N$ as $N\to+\infty$. We show that if the particle density of the initial configuration is bounded away from $0$ and $1$ down to scales $\mathsf{D}\ll T$ in a neighborhood of size $\mathsf{Q}\gg T$ of some location $x$ (i.e., $x$ is in the ``bulk''), and the initial configuration is balanced in a certain sense, then the space-time local statistics at $x$ are asymptotically governed by the extended discrete sine process (which can be identified with a translation invariant ergodic Gibbs measure on lozenge tilings of the plane). We also establish similar results for certain types of random initial data. Our proofs are based on a detailed analysis of the determinantal correlation kernel for the noncolliding Bernoulli random walk.

    The noncolliding Bernoulli random walk is a discrete analogue of the $\boldsymbol{\beta}=2$ Dyson Brownian Motion whose local statistics are universality governed by the continuous sine process. Our results parallel the ones in the continuous case. In addition, we naturally include situations with inhomogeneous local particle density on scale $T$, which nontrivially affects parameters of the limiting extended sine process, and in a particular case leads to a new behavior.
\end{abstract}

\date{}

\maketitle

\setcounter{tocdepth}{3}

\section{Introduction} 
\label{sec:introduction}

Our main object is a discrete time Markov chain in the $N$--dimensional lattice $\Z^N$ which is called the \emph{noncolliding Bernoulli random walk}. At $N=1$ by a single-particle chain we mean the simple random walk on $\Z$ which at each time step jumps by $1$ in the positive direction with probability $\be\in(0,1)$, or stays put with the complementary probability $1-\be$. For $N>1$, we consider $N$ independent identical particles on $\Z$ evolving according to the single-particle chain, and condition them to never collide (i.e., never occupy the same location of $\Z$ at the same time). Note that the condition has probability zero, and therefore, needs to be defined through a limit procedure which is performed in, e.g., \cite{konig2002non} (based on a classical result of \cite{KMG59-Coincidence}). The result is a time and space homogeneous Markov chain $\vec X(t)$ living in the Weyl chamber
\begin{equation}\label{Weyl_chamber}
    \W N=\{(x_1,\ldots,x_N)\in\Z^{N}\colon x_1< x_2<\ldots<x_N\},
\end{equation}
with transition probabilities
\begin{equation}
    \PP\big(\vec X(t+1)=\vec x'\mid \vec X(t)=\vec x\big)
    =\begin{cases}
        \dfrac{\V(\vec x')}{\V(\vec x)}
        \be^{|\vec x'|-|\vec x|}(1-\be)^{N-|\vec x'|+|\vec x|},
        &
        \textnormal{if $x'_i-x_i\in\{0,1\}$ for all $i$};\\
        \rule{0pt}{8pt}0,&\textnormal{otherwise},
    \end{cases}
    \label{noncolliding_Bernoulli_walks_transitions}
\end{equation}
where for $\vec x=(x_1,\dots,x_N)$ we denote $|\vec x|=x_1+\ldots+x_N$ and
\begin{equation}\label{Vandermonde}
    \V(\vec x)=\prod_{1\le j<i\le N}(x_i-x_j).
\end{equation}
We refer to \Cref{fig:Bernoulli} for an illustration.

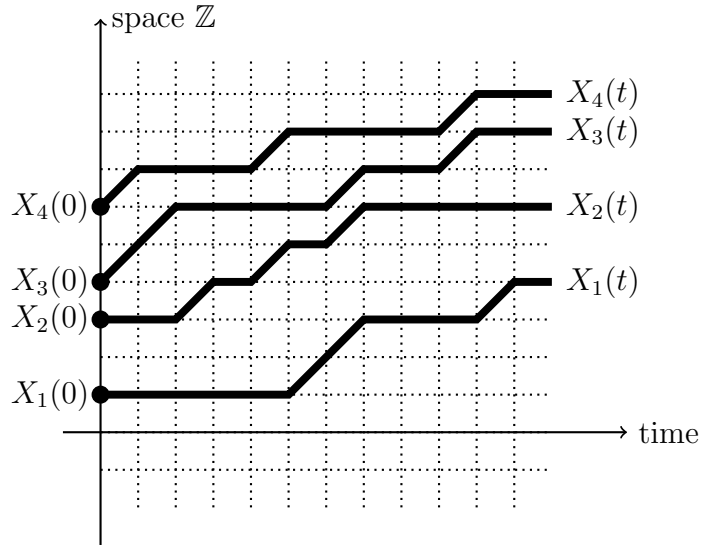
\begin{figure}[htbp]
    \begin{adjustbox}{max height=.5\textwidth}
    \begin{tikzpicture}[scale = 1, thick]
        \draw[->] (-.5,0) --++ (7.5,0) node[right] {time};
        \draw[->] (0,-1.5) --++ (0,7) node[right] {space $\Z$};
        \foreach \h in {-1,...,9}
        {
            \draw[dotted] (0,\h/2)--++(6,0);
        }
        \foreach \h in {1,...,11}
        {
            \draw[dotted] (\h/2,-1)--++(0,6);
        }
        \draw[fill] (0,3) circle(3pt)
        node[left] {$X_{4}(0)$};
        \draw[fill] (0,2) circle(3pt)
        node[left] {$X_3(0)$};
        \draw[fill] (0,1.5) circle(3pt)
        node[left] {$X_2(0)$};
        \draw[fill] (0,.5) circle(3pt)
        node[left] {$X_1(0)$};
        \draw[line width=3]
        (0,.5)--++(2.5,0)--++(.5,.5)--++(.5,.5)
        --++(1.5,0)--++(.5,.5)
        --++(.5,0) node [right] {$X_1(t)$};
        \draw[line width=3] (0,3)--++(.5,.5)
        --++(.5,0)--++(.5,0)--++(.5,0)
        --++(.5,.5)--++(.5,0)
        --++(.5,0)--++(.5,0)--++(.5,0)--++(.5,.5)--++(.5,0)--++(.5,0)
        node [right] {$X_4(t)$};
        \draw[line width=3] (0,2)--++(.5,.5)
        --++(.5,.5)--++(.5,0)--++(.5,0)
        --++(.5,0)--++(.5,0)
        --++(.5,.5)--++(.5,0)--++(.5,0)--++(.5,.5)
        --++(.5,0)--++(.5,0)
        node [right] {$X_3(t)$};
        \draw[line width=3] (0,1.5)--++(.5,0)
        --++(.5,0)--++(.5,.5)--++(.5,0)
        --++(.5,.5)--++(.5,0)--++(.5,.5)
        --++(.5,0)--++(.5,0)--++(.5,0)--++(.5,0)--++(.5,0)
        node [right] {$X_2(t)$};
    \end{tikzpicture}
    \end{adjustbox}
    \caption{Noncolliding Bernoulli random walk of $N=4$ particles started from the configuration $\vec X(0)=(1,3,4,6)\in\W4$.}
    \label{fig:Bernoulli}
\end{figure}

\medskip

If instead of the simple random walk we start from the Brownian Motion, then the same conditioning would lead to the celebrated ($\boldsymbol{\beta}=2$) \emph{Dyson Brownian Motion}, which plays a prominent role in the Random Matrix Theory \cite{dyson1962brownian}, \cite{AndersonGuionnetZeitouniBook}, \cite{LY_RMT_Bull2011}.\footnote{The $\boldsymbol{\beta}=2$ Dyson Brownian Motion is naturally associated to the Gaussian Unitary Ensemble (GUE) --- one of the most classical objects in the study of random matrices. There is an extension of the ensemble and of the Dyson Brownian Motion depending on a continuous parameter $\boldsymbol{\beta}>0$ (sometimes referred to as the inverse temperature). The noncolliding Brownian motions correspond to $\boldsymbol{\beta}=2$. The random--matrix parameter $\boldsymbol{\beta}$ has no connection to our $\beta$ of Bernoulli random walk. It is inevitable for us to use both betas, as the latter $\beta$ is rooted in the traditional notation in the asymptotic representation theory.} Therefore, the noncolliding Bernoulli random walk can be viewed as a discrete version of the Dyson Brownian Motion. There exists also an intermediate \emph{semi-discrete} version related to the Poisson process, see \Cref{sec:determinantal_kernels_for_other_noncolliding_dynamics}.

On the other hand, the noncolliding Bernoulli random walk can be coupled with a $(2+1)$--dimensional interacting particle system with local push/block interactions \cite{BorFerr2008DF}. The latter is linked to the Totally Asymmetric Simple Exclusion Process and its relatives and to random lozenge and domino tilings. We refer to \cite{BorFerr2008DF}, \cite{Nordenstam_Aztec_2009}, \cite{BorodinGorinSPB12}, \cite{BorodinPetrov2013Lect} for details.

From yet another side, fixed time distributions of $\vec X(t)$ can be identified with coefficients in decompositions of tensor products of certain representations of unitary groups, which are of interest in the asymptotic representation theory, see \Cref{sub:representation_theoretic_interpretation} for details.

\medskip

We concentrate on the local (``bulk'') limits of $\vec X(T)$ as both $N$ and $T$ tend to infinity. More precisely, we assume that $T\ll N$, which implies that the \emph{global} profile $\vec X(T)$ is almost indistinguishable from the initial condition $\vec X(0)$. On the other hand, our main results, \Cref{Theorem_main,Theorem_main_convergence}, show that under mild conditions on $\vec X(0)$ (see Assumptions \ref{ass:density}, \ref{ass:intermediate_behavior} in Section \ref{sub:main_result_bulk_limit_theorems}), the \emph{local} characteristics of $\vec X(T)$ (such as, e.g., the asymptotic distribution of the distance between two adjacent particles) become universal: they depend on two real parameters which are computed by explicit formulas involving $\vec X(0)$.

In more detail, we prove that the one--dimensional point process describing the particles of $\vec X(T)$ as $T,N\to\infty$ converges to the \emph{discrete sine process} of \cite{Borodin2000b} (we recall its definition in \Cref{sub:discrete_sine_and_incomplete_beta_kernels}). The two--dimensional point process describing the behavior of $\{\vec X(T+t)\}_{t}$ (where $t$ is kept finite as $T\to\infty$) asymptotically becomes the \emph{extended discrete sine process} \cite{okounkov2003correlation}, which can also be identified with a translation invariant ergodic Gibbs measure on lozenge tilings of the plane \cite{Sheffield2008}, \cite{KOS2006}.

As far as we know, in the \emph{discrete setting} general results on the universal appearance of the discrete sine process were not available previously, and we are aware only of \cite{Gorin2016} where a related theorem is proven for random lozenge tilings. However, for \emph{specific} examples (in our context this would mean considering very special initial conditions $\vec X(0)$ rather than general ones; note that the existing literature was mostly dealing with \emph{other}, yet related discrete random systems) the appearance of the extended discrete sine process was observed by many authors, cf.
\cite{Borodin2000b},
\cite{okounkov2003correlation},
\cite{Johansson2005arctic},
\cite{BorodinKuan2007U},
\cite{BKMM2003},
\cite{Gorin2007Hexagon},
\cite{BreuerDuits2011StaircasePaths},
\cite{Petrov2012}.
We expect that our results on local behavior of the noncolliding Bernoulli random walk can serve as a step towards establishing more general bulk universality results in discrete random systems.

\subsection*{Comparison with Dyson Brownian Motion}

In the continuous setting, the \emph{universal} appearance of the \emph{continuous} sine kernel process in bulk local limits of the Dyson Brownian Motion is relatively well understood. It was first conjectured by Dyson \cite{dyson1962brownian} in the early 1960s that the universal statistics should already appear after \emph{short} times. The first mathematical results in this direction were developed much later in \cite{Johansson2001Universality}, where the universal local behavior on \emph{large} times was proven (using the contour integral formulas of \cite{brezin1997extension} as an important ingredient). For the strongest results in this direction see \cite{Shcherbina2008universality} and references therein.

The rigorous treatment of the short times is even more recent, with the best results appearing in \cite{landon2015convergence}, \cite{erdos2015universality}. It should be noted that these results include cases other than the GUE ($\boldsymbol{\beta}=2$) one, and do not rely on explicit formulas specific to $\boldsymbol{\beta}=2$.

A detailed understanding of the bulk local behavior of the Dyson Brownian Motion became a crucial step towards establishing bulk universality of generalized Wigner matrices and other random matrix ensembles, see \cite{landon2015convergence}, \cite{erdos2015universality}, \cite{bourgade2015fixed}, references therein, and the review \cite{LY_RMT_Bull2011}. See also \cite{TaoVu2012Survey} for an alternative approach to bulk universality of random matrices.

From this point of view, our results are parallel to the $\boldsymbol \beta=2$ Dyson Brownian Motion developments as we prove a discrete analogue of the Dyson's conjecture. We also provide a generalization in a different direction and study the case when the local density of particles is not restricted to be the Lebesgue measure (as was usually assumed in the study of the Dyson Brownian Motion), but can be quite general instead. This leads to new phenomena, see the end of Section \ref{sub:applications} for one example.

\subsection*{Method}
On the technical side, our approach starts from the double contour integral
representation for the correlation kernel for the determinantal point process of
uniformly random Gelfand--Tsetlin patterns of
\cite{Petrov2012} (see also
\cite{Metcalfe2011GT}, \cite{duse2015asymptotic}). We find a limit transition which turns these random Gelfand--Tsetlin
patterns into $\vec X(t)$, and leads to formulas for the correlation functions of the latter process. These formulas are then analyzed using the steepest descent method. For this we develop arguments working for general initial conditions $\vec X(0)$ rather than specific ones, and this requires a significant technical effort.

\subsection*{Outline}

In \Cref{sec:main_results} we formulate our main results and discuss their applications. In \Cref{sec:from_lozenge_tilings_to_noncolliding_random_walks} we show how the noncolliding Bernoulli random walk can be obtained via a limit transition from the ensemble of uniformly random lozenge tilings of certain polygons. This leads to a double contour integral expression for the correlation kernel of the noncolliding Bernoulli random walk. \Cref{sec:setup_of_the_asymptotic_analysis,sec:proof_of_proposition_roots_positioning,sec:asymptotics_of_the_kernel} form the main technical part of the work and are devoted to the asymptotic analysis of the correlation kernel and of the noncolliding Bernoulli random walk. In \Cref{sec:applications} we prove the remaining statements from \Cref{sec:main_results} which deal with various applications of our main bulk limit theorems.

We discuss degenerations of our kernel to the kernels for noncolliding Poisson processes and for the Dyson Brownian Motion with arbitrary initial configurations in \Cref{sec:determinantal_kernels_for_other_noncolliding_dynamics}. In \Cref{sub:representation_theoretic_interpretation} we explain a representation-theoretic interpretation of discrete-space noncolliding random walks, and formulate a more general conjecture.

\subsection*{Acknowledgments}

We are grateful to Alexei Borodin for valuable comments on an earlier draft of the paper, and to Paul Bourgade for helpful remarks.
V.~G.\ was partially supported by the NSF grants DMS-1407562 
and DMS-1664619
and by the Sloan Research Fellowship.
L.~P.\ was partially supported by the 
NSF grant 
DMS-1664617,
and by the 
NSF grant PHY11-25915
through the participation in the
KITP program ``New approaches to non-equilibrium and random systems: KPZ
integrability, universality, applications and experiments''.
We are grateful to the anonymous reviewer for valuable comments 
which helped improve our technical arguments.


\section{Main results} 
\label{sec:main_results}

\subsection{Determinantal structure}
\label{sub:determinantal_structure_of_noncolliding_walks}

Our first result is a formula for the determinantal correlation kernel
of the
noncolliding Bernoulli random walk. Recall that a particle dynamics $\vec X(t)$ is said to be a
(dynamically) determinantal point process if its space-time correlations are given
by determinants of a certain kernel $K(t,x;s,y)$:
\begin{multline}
    \PP\left(
    \textnormal{the particle configuration $\vec X(t_i)$ on $\Z$
    contains the point $y_i$ for all $i=1,\ldots,k$}
    \right)\\=
    \det\big[K(t_\aind,y_\aind;t_\bind,y_\bind)\big]
    _{\aind,\bind=1}^{k},
    \label{determinantal_kernel}
\end{multline}
for any collection of pairwise distinct space-time points
$(t_i,y_i)\in\Z_{\ge0}\times\Z$, $i=1,\ldots,k$. In particular, when all $t_i$ are
the same (and are equal to $t$), we get a determinantal point process on $\Z$ with
the kernel $K_t(x;y)=K(t,x;t,y)$. General details
on determinantal point processes can be
found in, e.g., the surveys \cite{Soshnikov2000},
\cite{peres2006determinantal},
\cite{Borodin2009}.

\begin{theorem}\label{thm:intro_Bernoulli}
    The noncolliding Bernoulli random walk with parameter $\be\in(0,1)$ started from any initial configuration $\vec X(0)=\vec a=(a_1<\ldots<a_N) \in\W N$ is determinantal in the sense of \eqref{determinantal_kernel}, and its correlation kernel has the following form for $x_{1,2}\in\Z$, $t_{1,2}\in\Z_{\ge1}$:\footnote{Throughout the text $\mathbf{1}_{A}$ stands for the indicator function of an event $A$, $\i=\sqrt{-1}$, and we will employ the Pochhammer symbols $(z)_{k}=\frac{\G(z+k)}{\G(z)}=z(z+1)\ldots(z+k-1)$ and the binomial coefficients $\binom{k}{a}=(-1)^a\frac{(-k)_a}{a!}$, where $k,a\in\Z_{\ge0}$.}
    \begin{multline}
        K^{\textnormal{Bernoulli}}_{\vec a; \be}(t_1,x_1;t_2,x_2)
        =
        \mathbf{1}_{x_1\ge x_2}\mathbf{1}_{t_1>t_2}
        (-1)^{x_1-x_2+1}
        \binom{t_1-t_2}{x_1-x_2}
        \\+
        \frac{t_1!}{(t_2-1)!}\frac1{(2\pi\i)^{2}}
        \int\limits_{x_2-t_2+\frac12-\i\infty}
        ^{x_2-t_2+\frac12+\i\infty}dz
        \oint\limits_{\textnormal{all $w$ poles}}dw\,
        \frac{(z-x_2+1)_{t_2-1}}{(w-x_1)_{t_1+1}}
        \\\times\frac{1}{w-z}
        \frac{\sin(\pi w)}{\sin(\pi z)}
        \left(\frac{1-\be}{\be}\right)^{w-z}
        \prod_{r=1}^{N}\frac{z-a_r}{w-a_r}.
        \label{K_Bernoulli}
    \end{multline}
    The $z$ integration contour is the straight vertical line $\Re z=x_2-t_2+\frac12$ traversed upwards, and the $w$ contour is a positively (counter-clockwise) oriented circle or a union of two circles (this depends on the ordering of $t_1,x_1,t_2$, and $x_2$) encircling all the $w$ poles $\{x_1-t_1,x_1-t_1+1,\ldots,x_1-1,x_1\}\cap \{a_1,\ldots,a_N\}$ of the integrand (except $w=z$), see \Cref{fig:cont_intro}.
\end{theorem}

\begin{figure}[htbp]
    \scalebox{.85}{\begin{tikzpicture}[scale=1,ultra thick]
        \draw[->, very thick] (-7,0)--(3,0);
        \begin{scope}[shift={(-4,0)},scale=.8]
            \draw[line width=2.5] (-.2,-.2)--++(.4,.4);
            \draw[line width=2.5] (-.2,.2)--++(.4,-.4);
        \end{scope}
        \begin{scope}[shift={(-4+1/3,0)},scale=.8]
            \draw[line width=2.5] (-.2,-.2)--++(.4,.4);
            \draw[line width=2.5] (-.2,.2)--++(.4,-.4);
        \end{scope}
        \begin{scope}[shift={(-3+1/3,0)},scale=.8]
            \draw[line width=2.5] (-.2,-.2)--++(.4,.4);
            \draw[line width=2.5] (-.2,.2)--++(.4,-.4);
        \end{scope}
        \begin{scope}[shift={(-2+2/3,0)},scale=.8]
            \draw[line width=2.5] (-.2,-.2)--++(.4,.4);
            \draw[line width=2.5] (-.2,.2)--++(.4,-.4);
        \end{scope}
        \begin{scope}[shift={(-1/3,0)},scale=.8]
            \draw[line width=2.5] (-.2,-.2)--++(.4,.4);
            \draw[line width=2.5] (-.2,.2)--++(.4,-.4);
        \end{scope}
        \draw[line width=2,->] (-3.25,-2.5)--++(0,1.5);
        \draw[line width=2,->] (-3.25,-1)--++(0,2);
        \draw[line width=2] (-3.25,1)--++(0,1.5);
        \draw[densely dotted, line width=2,->]
        (-1.5,1.6) to [out=180,in=90]
        (-3.12,0) to [out=-90,in=180]
        (-1.5,-1.6) to [out=0,in=-90]
        (-0.03,0) to [out=90,in=0]
        (-1.5,1.6);
        \draw[densely dotted, line width=2,->]
        (-4,1.2) to [out=180,in=90]
        (-4.86,0) to [out=-90,in=180]
        (-4,-1.2) to [out=0,in=-90]
        (-3.4,0)
        to [out=90,in=0]
        (-4,1.2);
        \draw (-1/3,-.25)--++(0,.5) node at (-1/3-.15,.2) (x1) {};
        \node (x1l) at (0.5,1) [draw=black, line width=1, minimum width=1.7em, minimum height=1.2em,
        rectangle, rounded corners, text centered] {$x_1$};
        \draw (-4-1/2,-.25)--++(0,.5) node at (-4-1/2+.15,.2) (x1t1) {};
        \node (x1t1l) at (-6,1.5) [draw=black, line width=1, minimum width=1.7em, minimum height=1.2em,
        rectangle, rounded corners, text centered] {$x_1-t_1$};
        \draw (-3,-.25)--++(0,.5) node at (-3-.05,-.12) (x2t2) {};
        \node (x2t2l) at (-1.5,-2.2) [draw=black, line width=1, minimum width=1.7em, minimum height=1.2em,
        rectangle, rounded corners, text centered] {$x_2-t_2+1$};
        \draw[->, line width=1] (x1l) -- (x1);
        \draw[->, line width=1] (x1t1l) -- (x1t1);
        \draw[->, line width=1] (x2t2l) -- (x2t2);
        \node at (-2.9,2.3) {\Large $z$};
        \node at (-1,1.9) {\Large $w$};
        \node at (-4.3,1.5) {\Large $w$};
    \end{tikzpicture}}
    \caption{Integration contours in \eqref{K_Bernoulli}. The $w$ poles
    are highlighted by crosses.}
    \label{fig:cont_intro}
\end{figure}
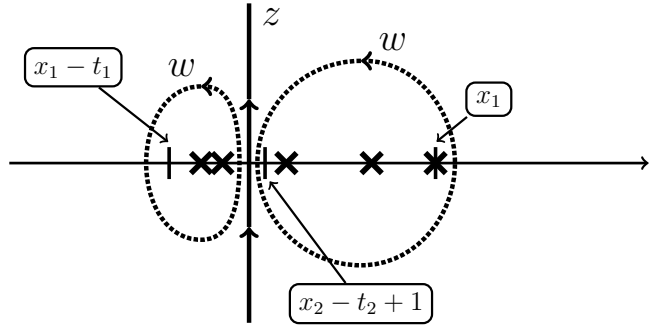


\begin{remark}\label{rmk:Krawtchouk}
When the $N$-point noncolliding Bernoulli random walk starts from the densely packed configuration $\vec a=(0,1,2,\ldots,N-1)$, the distribution of the $N$-point configuration $\vec X(t)\subset\Z$ at any time $t\in\Z_{\ge0}$ is the Krawtchouk orthogonal polynomial ensemble \cite{konig2002non}. Orthogonal polynomial ensembles are determinantal, and their correlation kernels are expressed through the corresponding univariate orthogonal polynomials --- in our case, the Krawtchouk polynomials.\footnote{See Section 1.10 in \cite{Koekoek1996} for definitions and basic properties of the Krawtchouk orthogonal polynomials, and \cite{Konig2005} for a survey of orthogonal polynomial ensembles.} This correlation kernel is explicit enough to be suitable for asymptotic analysis, see, e.g., \cite{johansson2000shape}, \cite{johansson1999}. The corresponding time-dependent kernel as in \eqref{determinantal_kernel} is also explicitly known, it is the extended Krawtchouk kernel \cite{Johansson2005arctic}. \Cref{thm:intro_Bernoulli} generalizes these results to an arbitrary initial configuration $\vec X(0)=\vec a\in\W N$.
\end{remark}

We prove \Cref{thm:intro_Bernoulli} in
\Cref{sec:from_lozenge_tilings_to_noncolliding_random_walks} below. In
\Cref{sec:determinantal_kernels_for_other_noncolliding_dynamics} we also discuss two
limits of the noncolliding Bernoulli random walk and the kernel \eqref{K_Bernoulli}:
\begin{enumerate}[$\bullet$]
    \item Noncolliding Poisson random walk ---
    independent Poisson processes conditioned to never collide.
    This limit is obtained by rescaling time from discrete to continuous,
    and sending $\be\to0$.

    \item
    Dyson Brownian Motion
    --- independent Brownian motions conditioned to never collide.
    This process (introduced
    in \cite{dyson1962brownian})
    is a diffusion limit of the noncolliding Bernoulli random walk.
    The correlation kernel for the Dyson Brownian Motion started from
    an arbitrary initial configuration was first obtained in
    \cite{Johansson2001Universality} (see
    also \cite{Shcherbina2008universality}).
    When the Dyson Brownian Motion starts from a special initial condition
    $(0,0,\ldots,0)$, its determinantal correlation kernel
    can be expressed through the Hermite orthogonal polynomials
    \cite{mehta2004random}, \cite{nagao1998multilevel}.
\end{enumerate}


\subsection{Extended discrete sine kernel} 
\label{sub:discrete_sine_and_incomplete_beta_kernels}

Let us now discuss the point process describing the local asymptotic behavior the
noncolliding Bernoulli random walk.

\begin{definition}
    By the \emph{extended discrete sine process}\footnote{Since in this paper we only discuss convergence to the discrete sine process and do not deal with its continuous counterpart, we often drop the word ``discrete''.} of slope $u\in\mathbb C$, $\Im(u)>0$, we mean the determinantal point process on $\mathbb Z\times \mathbb Z$ with the correlation kernel:
    \begin{equation}\label{incomplete_beta}
        K_{u}(t,x;s,y)
        =-\frac{1}{2\pi\i}\int_{\bar u}^{u}
        (1-z)^{t-s}z^{-(x-y)-1}dz,
    \end{equation}
    where the integration path crosses $(0,1)$ for $t> s$ and $(-\infty,0)$ for $t\le s$.
\end{definition}
The extended sine process was first introduced in \cite{okounkov2003correlation}
in relation to
the bulk limit of random lozenge tilings (equivalently, 3D Young diagrams).
In that paper
the
kernel \eqref{incomplete_beta}
was called the \emph{incomplete beta kernel}.
When $t=s$, the kernel \eqref{incomplete_beta} simplifies,
after conjugation by the
function $x\mapsto (-1)^{x}|u|^{-x}$,\footnote{Transformations of a correlation kernel
of the form
$K(x,y)\mapsto \frac{f(x)}{f(y)}K(x,y)$
(where $f$ is nowhere zero)
not changing the correlation functions
are sometimes referred to as the gauge
transformations.\label{gaugefootnote}}
to the discrete sine kernel of
\cite{Borodin2000b}:
\begin{equation}\label{discrete_sine}
    K^{\textnormal{sine}}_{\phi}(x;y)
    =\frac{\sin\big(\phi(x-y)\big)}{\pi(x-y)},\qquad x,y\in\Z,
    \quad \phi=\pi-\arg(u),
\end{equation}
with the agreement that
$K^{\textnormal{sine}}_{\phi}(x,x)=\phi/\pi$. The quantity
$\phi/\pi\in(0,1)$ is the density along the
$x$ direction of particles in the random configuration from the (extended) discrete sine process.

There exist other extensions of the discrete sine kernel \eqref{discrete_sine}, see \cite{borodin2006stochastic}, \cite{borodin2007periodic}, \cite{borodin2010gibbs}. In \Cref{sub:poisson_case} we briefly discuss the noncolliding Poisson random walk whose local statistics should be universally governed by an extension of \eqref{discrete_sine} other than \eqref{incomplete_beta}.

\begin{remark}\label{rmk:complement}
    The extended sine kernel was introduced in \cite{okounkov2003correlation} in terms of \emph{complementary} (to $\mathbb Z \times \mathbb Z$) configurations. The relation between kernels describing a configuration and its complement is \cite[Appendix A.3]{Borodin2000b}
    \begin{equation*}
        K^{\textnormal{complement}}(t,x;s,y)
        =\mathbf{1}_{x=y}\mathbf{1}_{t=s}-
        K(t,x;s,y).
    \end{equation*}
    In the case of the extended sine kernel, the above delta function can be incorporated inside the contour integral by dragging for $t=s$ the contour through zero and picking the residue of $z^{-(x-y)-1}$, which is exactly $\mathbf{1}_{x=y}$.
\end{remark}

The extended sine process is translation invariant in both directions
($t$ and $x$), and it
describes asymptotic bulk distribution of discrete
two-dimensional
determinantal point
processes when both dimensions stay discrete in the limit.
A characterization of the measure
determined by $K_{z}$ as a unique
translation invariant Gibbs measure of a given complex
slope was obtained in \cite{Sheffield2008}, \cite{KOS2006}.


\subsection{Bulk limit theorems} 
\label{sub:main_result_bulk_limit_theorems}

Here we formulate our main asymptotic
results --- an approximation of the correlation kernel of
the noncolliding Bernoulli random walk
by the extended sine kernel \eqref{incomplete_beta},
and a corresponding bulk local
limit theorem.

Assume that $N$ (the number of noncolliding particles) is our main parameter going to infinity, and that the time scale $T(N)\ll N$, $T(N)\to+\infty$ is fixed.\footnote{Throughout the paper by $A(N)\ll B(N)$ we mean that $\lim\limits_{N\to+\infty}A(N)/B(N)=0$, and similarly for $\gg$.} For each $N=1,2,\dots,$ we also fix an initial condition $\vec X(0)=\mathfrak{A}(N)=(a_1(N)<a_2(N)<\dots<a_N(N))$. We will often omit the dependence on $N$ and simply write $T$ (meaning $T(N)$) and $a_i$ (meaning $a_i(N)$, $1\le i \le N$), etc., when it leads to no confusion.

In what follows, we are describing
the behavior of $\vec X(T+t)$ near the point $x=0$.
Since the definition of the non-colliding
Bernoulli random walk is translation
invariant, one can readily extract similar results on the behavior
near an arbitrary point
$x=x(N)$ by shifting $\mathfrak{A}(N)$ appropriately.

The following two assumptions will be imposed on $\mathfrak{A}(N)$ throughout the
text:

\begin{assumption}[Local density]\label{ass:density}
    There exist scales $\dd=\dd(N)$ satisfying $\dd(N)\ll T(N)$ and
    $\QQ=\QQ(N)$ satisfying $T(N)\ll \QQ(N)\ll N$,
    and absolute constants\footnote{Here and below by an absolute constant we
    mean a certain constant which does not depend on $N$, $T(N)$, or the initial configuration $\mathfrak{A}(N)$.}
    $0<\lod\le\upd<1$, such that in every
    segment of length $\dd(N)$ inside $[-\QQ(N),\QQ(N)]\subset\R$
    there are at least $\lod \dd(N)$ and at most $\upd\dd(N)$ points of the initial configuration
    $\mathfrak{A}(N)$.
\end{assumption}

\begin{assumption}[Intermediate scales]\label{ass:intermediate_behavior}
    For all
    $\updelta>0$, $\RRR>0$ and $N$ large enough
    one has
    \begin{equation}\label{intermediate_behavior}
        \Bigg|\sum_{i\colon \RRR T(N)\le|a_i(N)|\le\updelta N}\frac{1}{a_i(N)}\Bigg|\le \A_{\RRR,\updelta},
    \end{equation}
    where $\A_{\RRR,\updelta}>0$ are absolute constants.
\end{assumption}
\begin{remark}\label{rmk:intermediate}
    Note that if \eqref{intermediate_behavior} holds for some
    $\updelta_0>0$, $\RRR_0>0$, then
    it holds for all other
    $\updelta>0$, $\RRR>0$ because
    the difference of the sums in the left-hand side of
    \eqref{intermediate_behavior} can be bounded by a
    part of the harmonic series (between $\RRR_0T$ and $\RRR T$, etc.),
    which is bounded by a
    constant
    independent of $N$.
\end{remark}

Both Assumptions \ref{ass:density} and \ref{ass:intermediate_behavior} serve the same goal: we want to guarantee that the average density of particles in $\vec X(T)$ near $x=0$ is bounded away from $0$ and from $1$, as otherwise the universal local behavior might fail. Assumption \ref{ass:density} simply bounds the density of the initial configuration $\mathfrak{A}(N)$, while Assumption \ref{ass:intermediate_behavior} requires that the ``densities'' of the configuration $\mathfrak{A}(N)$ far (at scales from $\RRR T(N)$ to $\updelta N$) to the right and to the left of $0$ (our point of observation) are ``comparable''.

We do not impose any other requirements on particles far away from $x=0$. We remark that one can easily cook up (e.g., using an intuition coming from the study of frozen boundaries in random tiling models) a situation in which Assumption \ref{ass:density} fails, and yet the average density of particles in $\vec X (T)$ close to $0$ is still bounded away from $0$ and $1$, and the universal local behavior holds. In other words, Assumption \ref{ass:density} is not necessary. On the other hand, Assumption \ref{ass:intermediate_behavior} is close to being necessary, see the discussion after \Cref{Theorem_main_convergence_Lebesgue} below. Overall, our Assumptions \ref{ass:density} and \ref{ass:intermediate_behavior} are simple to state and straightforward to check in applications, while a full necessary and sufficient condition would be much more technical and involved. We refer to detailed analysis in models of random tilings performed in \cite{Petrov2012}, \cite{duse2015asymptotic}, \cite{Duse2015_partII}, \cite{Duse_Thesis}, \cite{Gorin2016}.

\medskip

The following function of $\z\in\C$ will play a prominent role in our
asymptotic analysis:
\begin{equation}\label{S_T_prime}
    \Sfin'(\z)=
    \sum_{r=1}^{N}\frac{1}{T\z-a_r}
    -\pv\sum_{j\in  \mathfrak{L}_T }\frac{1}{T\z-j}
    -\log(\be^{-1}-1),
\end{equation}
where
\begin{equation}
\label{eq_L_def}
    \mathfrak{L}_T=\mathfrak{L}_{T(N)}=\{\ldots,-T-2,-T-1,-T\}\cup\{0,1,2,\ldots\},
\end{equation}
and the infinite sum should be understood as its principal value, i.e.,
\begin{equation}
\label{eq_pv_def}
 \pv\sum_{j\in  \mathfrak{L}_T}\frac{1}{T\z-j}=\lim_{M\to\infty}
 \sum_{\begin{smallmatrix} j\in  \mathfrak{L}_T \\
 |j|<M\end{smallmatrix}}\frac{1}{T\z-j}.
\end{equation}
We can alternatively write
\begin{equation}\label{Sfin_prime_initial_formula}
    \Sfin'(\z)=
    \sum_{r=1}^{N}\frac{1}{T\z-a_r}
    +\sum_{i=1}^{T-1}\frac{1}{T\z+i}
    -\pi\cot(\pi T\z)-\log(\be^{-1}-1)
\end{equation}
using the fact that
\begin{equation}\label{Euler_cot}
    \pi\cot(\pi z)=\pv\sum_{k\in\Z}\frac{1}{z-k},
\end{equation}
which follows from the Euler's product formula for the sine function
\begin{equation}\label{Euler_sine}
    \sin (\pi z)=\pi z\prod_{k=1}^{\infty}\left(1-\frac{z^{2}}{k^{2}}\right).
\end{equation}

\begin{proposition}\label{Proposition_roots_positioning}
Under Assumptions \ref{ass:density} and
\ref{ass:intermediate_behavior} there exists $N_0$ such that for all $N>N_0$ the
equation $\Sfin'(\z)=0$ has a unique root $\zfin=\zfin(N)$ in the upper half plane (i.e.\
satisfying $\Im(\zfin)>0$). Moreover, there exists a compact set
$\mathcal Z \subset
\{\z\in \mathbb C\colon \Im (\z)>0\}$ such that
$\zfin\in\mathcal Z$ for all $N>N_0$.
The set
$\mathcal Z$ and the constant $N_0$
may depend
on constants in Assumptions \ref{ass:density} and
\ref{ass:intermediate_behavior} but not on the choice of
$\mathfrak{A}(N)$.
\end{proposition}

\begin{theorem}\label{Theorem_main}
Under Assumptions \ref{ass:density} and \ref{ass:intermediate_behavior}, for any fixed
$t_1,x_1,t_2,x_2\in\mathbb Z$ we have
\begin{equation*}
    K^{\textnormal{Bernoulli}}_{\mathfrak{A}(N); \be}
    (t_1+T(N),x_1;t_2+T(N),x_2)= K_{\zfin(N)/(\zfin(N)+1)}(
    t_1,x_1;t_2,x_2)+o(1)
\end{equation*}
as $N\to\infty$, where $\zfin(N)$ is the unique root provided by \Cref{Proposition_roots_positioning}. The remainder $o(1)$ admits a tending to $0$ bound which may depend on constants in Assumptions \ref{ass:density} and \ref{ass:intermediate_behavior} but not on the choice of $\mathfrak{A}(N)$.
\end{theorem}
\begin{remark}
Since all probabilities describing the local behavior of
$\vec X(T+t)$
near
$x=0$ (with $t$ kept finite) are
expressed through
$K^{\textnormal{Bernoulli}}_{\mathfrak{A}(N); \be}(t_1+T,x_1;t_2+T,x_2)$
via
\eqref{determinantal_kernel},
\Cref{Theorem_main} means that as $N\to+\infty$,
locally near $x=0$ the distribution
of $\{\vec X(T+t)\}_{t}$ becomes
indistinguishable from the one corresponding to
the extended sine process.
\end{remark}

If $\mathfrak{A}(N)$ depends on $N$ in a regular way,
then \Cref{Theorem_main}
leads to a convergence statement.

\begin{definition}\label{def:vague}
We say that a sequence $\mu_k$,
$k=1,2,\dots$ of $\sigma$--finite measures on
$\mathbb R$ vaguely
converges to $\mu$, if for any continuous function $f$ with
compact support, we have
\begin{equation*}
    \lim_{k\to\infty} \int_{-\infty}^{+\infty} f(x)\mu_k(dx)=\int_{-\infty}^{+\infty} f(x) \mu(dx).
\end{equation*}
\end{definition}

Let us also denote
\begin{equation}\label{drift_N}
    \CCC_N(\RRR)=
    \sum_{i\colon |a_i(N)|\ge\RRR T(N)}\frac{1}{a_i(N)}.
\end{equation}
Assumption \ref{ass:intermediate_behavior} is equivalent to the boundedness of
the $\CCC_N(\RRR)$'s for fixed $\RRR$,
uniformly in $N$.

\begin{theorem}\label{Theorem_main_convergence}
    Suppose that a sequence $\mathfrak{A}(N)$, $N=1,2,\dots$ is such that
    Assumptions \ref{ass:density} and
    \ref{ass:intermediate_behavior} hold, there
    exists a $\sigma$--finite measure $\mu_{loc}$
    for which vaguely
    \begin{equation*}
        \lim_{N\to+\infty}
        \frac{1}{T(N)}\sum_{i=1}^N \delta_{a_i(N)/T(N)}
        =\mu_{loc},
    \end{equation*}
    and there exists a limit
    $\CCC(\RRR)=\lim_{N\to+\infty}\CCC_N(\RRR)$.

    Then the point process
    describing $\{\vec X(T(N)+t)\}_{t}$ near $x=0$ converges in
    distribution to the extended sine process of complex slope
    $\ulim$, in the sense that
    for each $t_1,x_1,t_2,x_2\in\mathbb Z$ we have
    \begin{equation*}
    \lim_{N\to\infty}
    K^{\textnormal{Bernoulli}}_{\mathfrak{A}(N); \be}(t_1+T(N),x_1;t_2+T(N),x_2)
    = K_{\ulim}(
    t_1,x_1;t_2,x_2),
    \end{equation*}
    where $\uvar=\ulim\in\C$ is a unique root in the upper half plane of the equation\footnote{The equation and the root $\ulim$ do not depend on $\RRR>0$.}
    \begin{equation}\label{Theorem_main_convergence_equation}
        \int_{-\infty}^{+\infty}
        \left(\frac{1-\uvar}{\uvar-(1-\uvar)v}
        +\frac{\mathbf{1}_{|v|>\RRR}}{v}
        \right)\mu_{loc}(dv)
        -\log \uvar
        =\CCC(\RRR)+
        \log(\be^{-1}-1)
        -\i\pi.
    \end{equation}
\end{theorem}
In fact, the additional hypotheses in \Cref{Theorem_main_convergence} as compared to \Cref{Theorem_main} are not too restrictive, see \Cref{rmk:passing_to_subsequences} below.

The condition that the quantities \eqref{drift_N} converge can sometimes be not easy to verify, and the determination of the limit $\CCC(\RRR)$ could be even harder. Let us present a sufficient condition and a way to compute $\CCC(\RRR)$ which involves the global profile $\vec X(0)$:

\begin{theorem}\label{Theorem_main_convergence_glob}
    Suppose that a sequence $\mathfrak{A}(N)$, $N=1,2,\dots$ is such that
    Assumptions \ref{ass:density} and
    \ref{ass:intermediate_behavior} hold,
    and, moreover, $\lim_{\updelta\to0}\A_{\RRR,\updelta}=0$ in Assumption \ref{ass:intermediate_behavior}.
    Next, let there
    exist a $\sigma$--finite measure $\mu_{loc}$
    and a probability measure $\mu_{glob}$
    for which vaguely
    \begin{equation}\label{mu_loc_mu_glob_convergence}
        \lim_{N\to+\infty}
        \frac{1}{T(N)}\sum_{i=1}^N \delta_{a_i(N)/T(N)}
        =\mu_{loc},
        \qquad
        \lim_{N\to+\infty}
        \frac{1}{N}\sum_{i=1}^N \delta_{a_i(N)/N}
        =\mu_{glob},
    \end{equation}
    and the principal value integral
    $\pv\int_{-\infty}^{\infty}v^{-1}\mu_{glob}(dv)$
    exists.
    Then the quantities $\CCC_N(\RRR)$ converge
    to this integral (so $\CCC(\RRR)=\CCC$ is independent of $\RRR$), and
    the conclusion of \Cref{Theorem_main_convergence} holds
    with
    \begin{equation}\label{main_convergence_glob_drift_term}
        \CCC=\pv\int_{-\infty}^{\infty}
        \frac{\mu_{glob}(dv)}{v}.
    \end{equation}
\end{theorem}

In many applications $\mu_{loc}$ is a multiple of the
Lebesgue measure on $\R$,
in which case $\ulim$ is more
explicit:
\begin{theorem}\label{Theorem_main_convergence_Lebesgue}
    Assume that \Cref{Theorem_main_convergence} holds
    with
    $\mu_{loc}$ being
    $q\in(0,1)$ times the Lebesgue measure on $\R$.
    Then $\CCC(\RRR)=\CCC$ does not depend on $\RRR$,
    and the complex slope of the limiting
    extended sine process is given by
    \begin{equation}\label{ulim_Lebesgue}
        \ulim=
        \frac{\be e^{-\CCC}}{1-\be}\,e^{\i \pi  (1-q)}.
    \end{equation}
\end{theorem}

Let us make some remarks about the elegant formula \eqref{ulim_Lebesgue}. First, the (same-time) density of particles under the limiting extended sine process is $1-\arg(\ulim)/\pi=q$, as it should be. In particular, this density does not depend on the ``speed'' $\be$ of the noncolliding random walk, or on the parameter $\CCC$ capturing the effect of the global profile.

To isolate the effect of the global profile, observe that the second parameter $|\ulim|$ of the extended sine kernel can be rewritten as
\begin{equation}\label{beta_eff}
    |\ulim|=
    \frac{\be_{\textnormal{eff}}}{1-\be_{\textnormal{eff}}},
    \qquad
    \be_{\textnormal{eff}}=\frac{1}{1+e^{\CCC}(\be^{-1}-1)}.
\end{equation}
That is, for fixed $q$ the bulk local distribution is the same as if the parameter $\be$ was replaced by $\be_{\textnormal{eff}}$, while the contribution from the global profile (encoded by $\CCC$) was not present.

The quantity $\be_{\textnormal{eff}}$ increases in $\be$ and decreases in $\CCC$. The dependence on $\CCC$ can be interpreted as an effect of repulsion. For example, having much more particles of the initial configuration to the right of $0$ than to the left corresponds to larger values of $\CCC$, which leads to a decrease in $\be_{\textnormal{eff}}$.

Moreover, if $\CCC$ is very large or very small, then  $\be_{\textnormal{eff}}$ is close to $0$ or $1$, respectively, leading to an almost deterministic behavior of the noncolliding paths in the bulk local limit. This suggests that our Assumption \ref{ass:intermediate_behavior} is close to being necessary for the universal local bulk behavior. Namely, if it is violated, then $\CCC=\pm\infty$ along a subsequence $\{N_k\}$, and so the local bulk distribution is not described by the universal extended sine kernel. However, we will not pursue this analysis further.


\subsection{Applications} 
\label{sub:applications}

Let us give several examples which demonstrate that Assumptions \ref{ass:density} and \ref{ass:intermediate_behavior} are checkable in applications. The first example deals with an arbitrary \emph{smooth} deterministic initial configuration $\vec X(0)$.

\begin{theorem}\label{Theorem_f_IC}
    Take a twice continuously differentiable function $f$ on $[-\frac12,\frac12]$ such that \mbox{$f'(x)>1$} for all $x\in[-\frac12,\frac12]$, and $f(-\frac12)<0<f(\frac12)$. Let $\chi\in(-\frac12,\frac12)$ be the unique point where $f(\chi)=0$. Assume for simplicity that $N$ is odd, and let the initial configuration of the noncolliding Bernoulli random walk be\footnote{Throughout the text $\fl{\cdots}$ denotes the integer part.}
    \begin{equation}\label{f_IC}
    a_i(N)=\fl{N f(i/N)},
    \qquad
    i=-\tfrac{N-1}2,
    -\tfrac{N-1}2+1,\ldots,\tfrac{N-1}2-1,\tfrac{N-1}2.
    \end{equation}
    Fix any $0<\TNPower<1$, and let $T(N)=\fl{N^{\TNPower}}$. Then \Cref{Theorem_main_convergence_Lebesgue} is applicable, where $\mu_{loc}$ is the Lebesgue measure on $\R$ times $q=1/f'(\chi)$, and $\CCC$ has the form
    \begin{equation}\label{f_IC_drift_term}
        \CCC=
        \pv\int_{-\frac12}^{\frac12}\frac{dx}{f(x)}.
    \end{equation}
\end{theorem}
\begin{remark}
    In the situation of \Cref{Theorem_f_IC} the global probability measure exists and has the form $\mu_{glob}(dv)={dv}/{f'(f^{-1}(v))}$. The expression \eqref{main_convergence_glob_drift_term} for $\CCC$ is equivalent to \eqref{f_IC_drift_term} via a change of variables.
\end{remark}

The next two examples deal with \emph{random} initial configurations $\vec X(0)$.

\begin{theorem} \label{Theorem_Bernoulli_IC}
Fix $0<p<1$ and $0<\al<1$. For $M=1,2,\dots$, consider a particle configuration on $\{-\fl{M(1-\al)},-\fl{M(1-\al)}+1,\dots, \fl{M\al}-1,\fl{M\al}\}$ obtained by putting a particle at each location with probability $p$ independently of all others. Let $\NRandom$ be the (random) number of particles in this configuration, and $\boldsymbol{\mathfrak{A}}(\NRandom)$ denote the configuration itself. By $\vec{\mathbf{X}}(t)$ denote the noncolliding Bernoulli random walk started from $\boldsymbol{\mathfrak{A}}(\NRandom)$. Choose $0<\TNPower<1$ and set $T(M)=\fl{M^{\TNPower}}$.

Then the point process $\{\vec{\mathbf{X}}(T(M)+t)\}_{t}$ converges near $x=0$ to the extended sine process as in \Cref{Theorem_main_convergence}, where $\mu_{loc}$ is $p$ times the Lebesgue measure on $\R$, and $\CCC(\RRR)=\CCC=p\log(\frac{\al}{1-\al})$. That is, the complex slope of the limiting extended sine process is
\begin{equation*}
    \ulim=
    \frac{\be}{1-\be}\left(\frac{1-\al}{\al}\right)^{p}
    e^{\i\pi(1-p)}.
\end{equation*}
\end{theorem}

\begin{proposition}\label{Proposition_sine_IC}
Fix $0<\phi<\pi$ and $0<\al<1$. For $M=1,2,\dots$, let the initial particle configuration of the noncolliding Bernoulli random walk be obtained by restricting the configuration of the discrete sine process of density $\phi/\pi$ (i.e., with the correlation kernel $K^{\textnormal{sine}}_{\phi}$ given by \eqref{discrete_sine}) to $\{-\fl{M(1-\al)},-\fl{M(1-\al)}+1,\dots, \fl{M\al}-1,\fl{M\al}\}$.

With other notation the same as in \Cref{Theorem_Bernoulli_IC}, the point process $\{\vec{\mathbf{X}}(T(M)+t)\}_{t}$ describing the configuration of the noncolliding walk started from the sine process initial configuration converges near $x=0$ to the extended sine process as in \Cref{Theorem_main_convergence}, where $\mu_{loc}$ is $\frac\phi\pi$ times the Lebesgue measure on $\R$, and $\CCC(\RRR)=\CCC=\frac\phi\pi\log(\frac{\al}{1-\al})$, so
\begin{equation*}
    \ulim=
    \frac{\be}{1-\be}\left(\frac{1-\al}{\al}\right)^{\frac\phi\pi}
    e^{\i(\pi-\phi)}.
\end{equation*}
is the complex slope of the limiting extended sine process.
\end{proposition}
\begin{remark}
    \Cref{Proposition_sine_IC} is formulated for time $T(M)=\fl{M^{\TNPower}}$ going to infinity. However, it is probably true even for $T(M)=0$ because the initial configuration is already close the same-time configuration of the extended sine process. Here we do not pursue in this direction.
\end{remark}

For the last example let us consider an initial configuration $\vec X(0)$ for which $\mu_{loc}$ differs from the Lebesgue measure.
\begin{proposition}\label{prop:nonLeb_IC}
    Fix two parameters $0<\eta<1$ and $h>0$. Set $T(N)=\lfloor N^\eta\rfloor$. For $N=1,2,\dots$, let $\mathfrak{A}(N)$ be the $N$--particle configuration defined by the following three conditions, cf. \Cref{fig:nonLeb_IC}:
    \begin{enumerate}[$\bullet$]
        \item There are $\fl{N/2}$ particles to the left from the origin and they occupy every second lattice site;
        \item $\lfloor h T(N) \rfloor$ particles adjacent to the origin from the right occupy every third lattice site;
        \item The remaining particles are to the right from $3\lfloor h T(N) \rfloor$ and they occupy every second lattice site.
    \end{enumerate}
    Then the point process $\{\vec{\mathbf{X}}(T(M)+t)\}_{t}$ converges near $x=0$ to the extended sine process as in \Cref{Theorem_main_convergence}. The complex slope $\ulim$ of the limiting extended sine process is a unique point in the upper half--plane satisfying
    \begin{equation}
    \label{eq_non_leb_example}
        \sqrt[6]{1-3h\frac{1-\ulim}{\ulim}}=
        \i\ulim (1-\be^{-1}),
    \end{equation}
    where the $6$th degree root is understood in the sense of the principal branch.
\end{proposition}
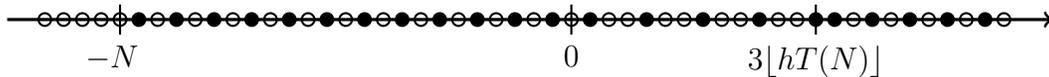
\begin{figure}[htbp]
    \begin{tikzpicture}
        [scale=1,thick]
        \draw[->, very thick] (-7.5,0)--(6.4,0);
        \draw (0,.2)--++(0,-.4) node[below] {$0$};
        \draw (-6,.2)--++(0,-.4) node[below, xshift=-3]
        {$-N$};
        \draw (3.25,.2)--++(0,-.4) node[below] {$3\lfloor h T(N) \rfloor$};
        \foreach \ppp in {-5,-3,-1,-7,-9,-11,-13,
        -15,-17,-19,-21,-23,
        1,4,7,10,13,14,16,18,20,22}
        {
            \draw[fill,thick] (\ppp/4,0) circle (2.4pt);
        }
        \foreach \ppp in {0,-2,-4,-6,-8,-10,-12,-14,-16,-18,
        -20,-22,-24,-25,-26,-27,-28,
        2,3,5,6,8,9,11,12,15,17,19,21,23}
        {
            \draw[thick] (\ppp/4,0) circle (2.4pt);
        }
    \end{tikzpicture}
    \caption{Initial condition in \Cref{prop:nonLeb_IC}. Black dots mean particles.}
    \label{fig:nonLeb_IC}
\end{figure}

The behavior of $\ulim$ given by \eqref{eq_non_leb_example} as a function of $h$ is quite interesting. When $h$ is small, $\ulim \approx \i \frac{\be}{1-\be}$, matching Theorem \ref{Theorem_main_convergence_Lebesgue}. On the other hand, as $h\to+\infty$, $\ulim$ goes to infinity in the direction $+\i\infty$. This leads to a behavior which seems new. Namely, as $h\to+\infty$, the same-time local configuration around zero is still governed by the discrete sine kernel \eqref{discrete_sine} (with $\phi=\pi/2$), while the time-dependent extension of this process is \emph{deterministic}: at each discrete time step each particle always goes to the right by $1$ and does not stay put. Heuristically, for very large $h$ the repelling force coming from the higher density region to the left of the origin is so large that this creates a deterministic flow of particles.

It is likely that with a proper time rescaling one can find a more delicate $h\to+\infty$ limit in which the paths make rare jumps, as in the classical transition from the random walk to the Poisson process. We will not pursue this direction here.

\begin{remark}
    In this observation we first took the large $N$ limit, and then a degeneration in the parameter $h$. This is the reason why the same-time distribution stays universal as $h\to+\infty$. We believe that if instead $h=h(N)$ goes to infinity in a certain way, then the limiting local configuration around zero would become \emph{completely deterministic}: particles would occupy every other site, and at each time step go to the right by $1$. This combined limit does not follow from \Cref{prop:nonLeb_IC}, and we will not consider it further.
\end{remark}


\subsection*{Notation}

Throughout the paper $C$ stands for positive constants 
whose values may change from line to line.
These constants might depend on the parameters of the model
(and our assumptions about them),
but not on variables going to zero or infinity.


\section{From lozenge tilings to noncolliding walks: proof of \Cref{thm:intro_Bernoulli}} 
\label{sec:from_lozenge_tilings_to_noncolliding_random_walks}

\subsection{Random lozenge tilings} 
\label{sub:uniformly_random_lozenge_tilings}

Consider uniformly random tilings
(by lozenges of three types) of
polygons drawn on the triangular lattice,
see \Cref{fig:tiling_paths},
left. The asymptotic behavior of such tilings in various regimes has
been studied in \cite{CohnKenyonPropp2000},
\cite{OkounkovKenyon2007Limit}, \cite{Kenyon2004Height},
\cite{Gorin2007Hexagon}, \cite{borodin-gr2009q}, \cite{Petrov2012},
\cite{Petrov2012GFF}, \cite{GorinPanova2012_full},
\cite{Novak2015Lozenge}.\footnote{Note that we are using an affine
transform of the regular triangular lattice, thus
our picture differs from some
of the references cited. This is done for a better coordinate
notation in our situation.}

We will employ a result of \cite{Petrov2012}
(see also \cite{duse2015asymptotic})
on the determinantal
structure of uniformly random lozenge tilings of polygons such as in
\Cref{fig:tiling_paths}, left. That is, consider a
trapezoid $\mathcal{T}_{N,L}$ of height $L\in\Z_{\ge1}$ with vertices
$(\frac12,0)$, $(\frac12,L)$, $(\frac12-N,L)$, and $(\frac12-N-L,0)$.
Fix $(y_1<\ldots<y_L)\in\{0,-1,\ldots,-N-L+1\}$
and put lozenges of type
\scalebox{.38}{\begin{tikzpicture}
    \draw [line width=3] (-.25,0)--++(0,-.5)
    --++(.5,.5)--++(0,.5)--cycle;
\end{tikzpicture}}
at each of the $y_i$'s cutting $L$ small triangles at the base of the trapezoid. We will consider tilings of the resulting figure by lozenges of three types. The assumption that the $y_i$'s belong to $\{0,-1,\ldots,-N-L+1\}$ (and hence are all negative) is not essential since the whole situation is translation invariant. However, this will be convenient in \Cref{sub:limit_ltoinfty_paths} when discussing the limit as $L\to+\infty$.

\begin{remark}
    Putting the lozenges
    \scalebox{.38}{\begin{tikzpicture}
        \draw [line width=3] (-.25,0)--++(0,-.5)
        --++(.5,.5)--++(0,.5)--cycle;
    \end{tikzpicture}}
    at the $y_i$'s fixes locations of some other of the lozenges of the same type (the darker ones in \Cref{fig:tiling_paths}, left). In this way the tiling we are describing can be alternatively interpreted at a lozenge tiling of a certain polygon.
\end{remark}

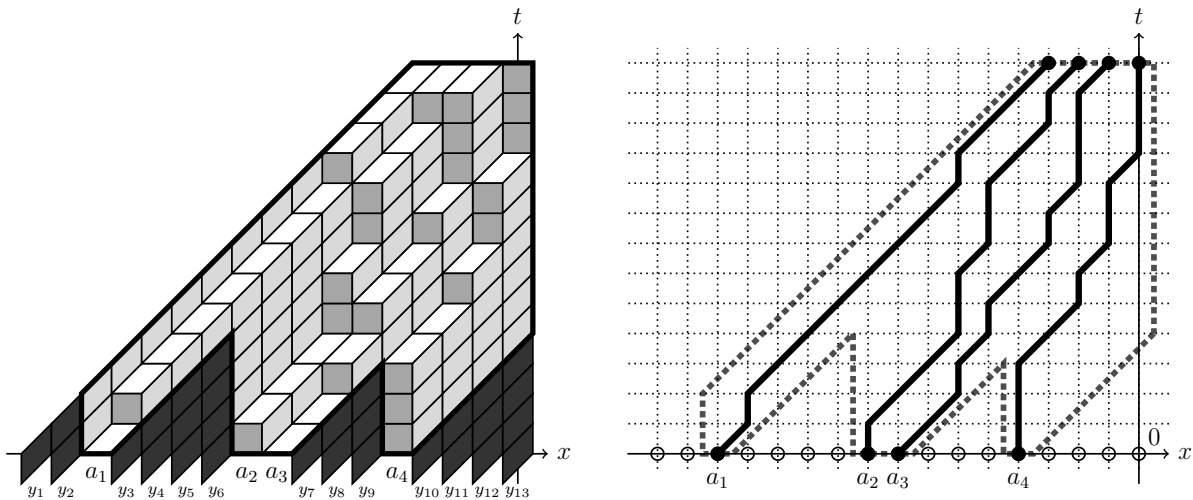
\begin{figure}[htbp]
    \scalebox{.8}{\begin{tikzpicture}[scale = 1, thick]
        \draw[->] (0,-.5) --++ (0,7.5) node[above] {$t$};
        \draw[->] (-8.5,0) --++ (9,0) node[right] {$x$};
        \node[below,yshift=-2] at (-2,0) {$a_{4}$};
        \node[below,yshift=-2] at (-4,0) {$a_{3}$};
        \node[below,yshift=-2] at (-4.5,0) {$a_{2}$};
        \node[below,yshift=-2] at (-7,0) {$a_{1}$};
        \node[below] at (-8,-.4) {$\scriptstyle y_{1}$};
        \node[below] at (-7.5,-.4) {$\scriptstyle y_{2}$};
        \node[below] at (-6.5,-.4) {$\scriptstyle y_{3}$};
        \node[below] at (-6,-.4) {$\scriptstyle y_{4}$};
        \node[below] at (-5.5,-.4) {$\scriptstyle y_{5}$};
        \node[below] at (-5,-.4) {$\scriptstyle y_{6}$};
        \node[below] at (-3.5,-.4) {$\scriptstyle y_{7}$};
        \node[below] at (-3,-.4) {$\scriptstyle y_{8}$};
        \node[below] at (-2.5,-.4) {$\scriptstyle y_{9}$};
        \node[below] at (0,-.4) {$\scriptstyle y_{13}$};
        \node[below] at (-.5,-.4) {$\scriptstyle y_{12}$};
        \node[below] at (-1,-.4) {$\scriptstyle y_{11}$};
        \node[below] at (-1.5,-.4) {$\scriptstyle y_{10}$};

        \foreach \ll in {
        (-3.5,.5),(-6.5,.5),
        (-4,1),(-3,1),
        (-3.5,1.5),(-6,1.5),
        (-5.5,2),(-3,2),(-2.5,2),(-1.5,2),
        (-5,2.5),(-1,2.5),
        (-4.5,3),(-2,3),
        (-4,3.5),(-2.5,3.5),(-1.5,3.5),(-.5,3.5),
        (-3.5,4),
        (-3,4.5),(-1,4.5),
        (-2,5),(0,5),
        (-2.5,5.5),(-1.5,5.5),
        (-2,6),
        (-1.5,6.5),(-1,6.5),(-.5,6.5)
        }
        {
        \begin{scope}[shift=\ll]
            \draw [thick,fill=white] (-.25,0)--++(.5,0)--++(-.5,-.5)
            --++(-.5,0)--cycle;
        \end{scope}}

        \foreach \sl in {
        (-2,.5),(-4.5,.5),
        (-2,1),(-6.5,1),
        (-2,1.5),(-3,1.5),
        (-2.5,2.5),(-3,2.5),
        (-1,3),(-3,3),
        (-.5,4),(-1.5,4),(-2.5,4),
        (-.5,4.5),(-2.5,4.5),
        (-1,5),(-3,5),
        (-1,5.5),(0,5.5),
        (-1,6),(0,6),(-1.5,6),(0,6.5)
        }
        {
        \begin{scope}[shift=\sl]
            \draw [thick,fill=gray!70!white] (-.25,0)--++(.5,0)
            --++(0,-.5)--++(-.5,0)--cycle;
        \end{scope}}

        \foreach \vl in {
        (-7,.5),(-4,.5),(-1.5,.5),
        (-1,1),(-1.5,1),(-1.5,1.5),(-1,1.5),(-.5,1.5),(-1,2),
        (-.5,2),(-0,2),
        (0,2.5),(-.5,2.5),(0,3),(-.5,3),(0,3.5),(0,4),(0,4.5),
        (-7,1),(-6.5,1.5),(-6,2),(-5.5,2.5),(-5,3),(-4.5,3.5),
        (-4,4),(-3.5,4.5),
        (-6,1),(-5.5,1.5),(-5,2),(-4.5,1),(-4.5,1.5),(-4.5,2),
        (-4.5,2.5),
        (-4,1.5),(-4,2),(-4,2.5),(-4,3),
        (-3.5,2),(-3.5,2.5),(-3.5,3),(-3.5,3.5),
        (-3,3.5),(-3,4),
        (-2.5,5),(-2,5.5),
        (-3.5,1),
        (-2.5,1.5),(-2.5,3),
        (-2,2),(-2,2.5),(-2,3.5),(-2,4),(-2,4.5),
        (-1.5,5),(-1.5,4.5),(-1.5,3),(-1.5,2.5),
        (-1,3.5),(-1,4),
        (-.5,5),(-.5,5.5),(-.5,6)
        }
        {
        \begin{scope}[shift=\vl]
            \draw [thick,fill=gray!30!white] (-.25,0)--++(0,-.5)
            --++(.5,.5)--++(0,.5)--cycle;
        \end{scope}}

        \foreach \vl in {
        (0,0),(-.5,0),(-1,0),(-1.5,0),
        (0,.5),(-.5,.5),(-1,.5),
        (0,1),(-.5,1),
        (0,1.5),
        (-2.5,0),(-3,0),(-3.5,0),
        (-2.5,.5),(-3,.5),(-2.5,1),
        (-5,0),(-5.5,0),(-6,0),(-6.5,0),
        (-5,.5),(-5.5,.5),(-6,.5),
        (-5,1),(-5.5,1),
        (-5,1.5),
        (-7.5,0),(-8,0),
        (-7.5,.5)
        }
        {
        \begin{scope}[shift=\vl]
            \draw [thick,fill=black!80!white] (-.25,0)--++(0,-.5)
            --++(.5,.5)--++(0,.5)--cycle;
        \end{scope}}

        \draw[line width=2.5] (.25,2)--++(-2,-2)--++(-.5,0)
        --++(0,1.5)--++(-1.5,-1.5)--++(-1,0)--++(0,2)--++(-2,-2)
        --++(-.5,0)--++(0,1)--++(5.5,5.5)--++(2,0)--++(0,-4.5)--cycle;
    \end{tikzpicture}}
    \hspace{10pt}
    \scalebox{.8}{\begin{tikzpicture}[scale = 1, thick]
        \draw[->] (-8.5,0) --++ (9,0) node[right] {$x$};
        \draw[->] (0,-.5) --++ (0,7.5) node[above] {$t$};

        \node[above right] at (0,0) {0};
        \foreach \h in {1,...,13}
        {
            \draw[dotted] (.5,\h/2)--++(-9,0);
        }
        \node[below] at (0,-.4) {\phantom{$\scriptstyle y_{13}$}};
        \foreach \y in {-16,...,0}
        {
            \draw (\y/2,0) circle(3pt);
            \draw[dotted] (\y/2,-.25)--++(0,7);
        }
        \draw[fill] (-2,0) circle(3pt) node[below,yshift=-5]
        {$a_{4}$};
        \draw[fill] (-4,0) circle(3pt) node[below,yshift=-5]
        {$a_{3}$};
        \draw[fill] (-4.5,0) circle(3pt) node[below,yshift=-5]
        {$a_{2}$};
        \draw[fill] (-7,0) circle(3pt) node[below,yshift=-5]
        {$a_{1}$};
        \draw[fill] (0,13/2) circle(3pt);
        \draw[fill] (-.5,13/2) circle(3pt);
        \draw[fill] (-1,13/2) circle(3pt);
        \draw[fill] (-1.5,13/2) circle(3pt);
        \draw[line width=3] (-2,0)--++(0,.5)--++(0,.5)--++(0,.5)
        --++(.5,.5)--++(.5,.5)--++(0,.5)--++(.5,.5)
        --++(0,.5)--++(0,.5)--++(.5,.5)--++(0,.5)--++(0,.5)
        --++(0,.5);
        \draw[line width=3] (-4,0)--++(.5,.5)--++(.5,.5)--++(0,.5)
        --++(.5,.5)--++(0,.5)--++(.5,.5)--++(.5,.5)
        --++(0,.5)--++(.5,.5)--++(0,.5)--++(0,.5)--++(0,.5)
        --++(.5,.5);
        \draw[line width=3] (-4.5,0)--++(0,.5)--++(0.5,.5)--++(0.5,.5)
        --++(.5,.5)--++(0,.5)--++(0,.5)--++(.5,.5)
        --++(0,.5)--++(0,.5)--++(.5,.5)--++(.5,.5)--++(0,.5)
        --++(.5,.5);
        \draw[line width=3] (-7,0)--++(.5,.5)--++(0,.5)--++(.5,.5)
        --++(.5,.5)--++(.5,.5)--++(.5,.5)--++(.5,.5)
        --++(.5,.5)--++(.5,.5)--++(0,.5)--++(.5,.5)--++(.5,.5)
        --++(.5,.5);

        \draw[densely dashed, line width = 2.7, opacity=.7] (.25,2)
        --++(-2,-2)
        --++(-.5,0)--++(0,1.5)--++(-1.5,-1.5)--++(-1,0)--++(0,2)
        --++(-2,-2)
        --++(-.5,0)--++(0,1)--++(5.5,5.5)
        --++(2,0)--++(0,-4.5)--cycle;
    \end{tikzpicture}}
    \caption{
    A lozenge tiling of (left) and its bijective encoding as
    a collection of noncolliding paths (right).
    The trapezoid $\mathcal{T}_{4,13}$
    has height
    $L=13$, and the
    number of paths ($N=4$) is
    the same as the length of the top side of $\mathcal{T}_{4,13}$.}
    \label{fig:tiling_paths}
\end{figure}

The total number of such tilings is equal to
(e.g., see Section 2 in \cite{BorodinPetrov2013Lect})
\begin{equation}\label{ZNL_partition_function}
    Z^{N,L}_{\{y_1,\ldots,y_L\}}
    =\prod_{1\le i<j\le L}\frac{y_j-y_i}{j-i}.
\end{equation}
We interpret the uniformly random tiling as a random
particle configuration by looking at centers of the lozenges
of type
\scalebox{.38}{\begin{tikzpicture}
    \draw [line width=3] (-.25,0)--++(0,-.5)
    --++(.5,.5)--++(0,.5)--cycle;
\end{tikzpicture}}
(there are $L(L+1)/2$ lozenges in total).
The centers have integer coordinates
$(x,t)\in\Z_{\le 0}\times\Z_{\ge0}$.

\begin{theorem}\label{thm:tilings_kernel}
    The uniformly random tiling described
    above gives rise to a determinantal point process, that is,
    \begin{multline*}
        \PP\left(\textnormal{there are lozenges
        \scalebox{.38}{\begin{tikzpicture}
        \draw [line width=3] (-.25,0)--++(0,-.5)
        --++(.5,.5)--++(0,.5)--cycle;
        \end{tikzpicture}}
        at locations $(x_i,t_i)$, $i=1,\ldots,k$}\right)
        \\=\det\big[K^{\textnormal{tilings}}_{N;\{y_1,\ldots,y_L\}}
        (t_\aind,x_\aind;t_\bind,x_\bind)\big]_{\aind,\bind=1}^{k},
    \end{multline*}
    with the correlation kernel
    \begin{multline}\label{tilings_kernel}
        K^{\textnormal{tilings}}_{N;\{y_1,\ldots,y_L\}}
        (t_1,x_1;t_2,x_2)=
        -\mathbf{1}_{t_1<t_2}\mathbf{1}_{x_2\le x_1}
        \frac{(x_1-x_2+1)_{t_2-t_1-1}}{(t_2-t_1-1)!}
        \\+\frac{t_1!}{(t_2-1)!}
        \frac1{(2\pi\i)^{2}}
        \oint\limits_{c(x_2-t_2+1)}dz\oint\limits_{C(\infty)}dw
        \frac{(z-x_2+1)_{t_2-1}}{(w-x_1)_{t_1+1}}
        \frac{1}{w-z}
        \prod_{r=1}^{L}\frac{w-y_r}{z-y_r}.
    \end{multline}
    Here $x_1,x_2\in\Z$, $0\le t_1\le L-1$, $1\le t_2\le L-1$,
    and the integration contours look as
    in \Cref{fig:cont_one}:
    \begin{enumerate}[$\bullet$]
        \item The $z$ contour $c(x_2-t_2+1)$ is positively
        oriented and encircles the points
        $x_2-t_2+1,x_2-t_2+2,\ldots,y_L$,
        and does not encircle $x_2-t_2,x_2-t_2-1,\ldots$;
        \item The $w$ contour $C(\infty)$
        is positively oriented and
        encircles $c(x_2-t_2+1)$ and is
        sufficiently large so that to include all $w$ poles
        $x_1,x_1-1,\ldots,x_1-t_1$.
    \end{enumerate}
\end{theorem}
This is the same kernel as in \cite[Thm. 5.1]{Petrov2012},
up to the change of coordinates $t_{1,2}=L-n_{1,2}$.

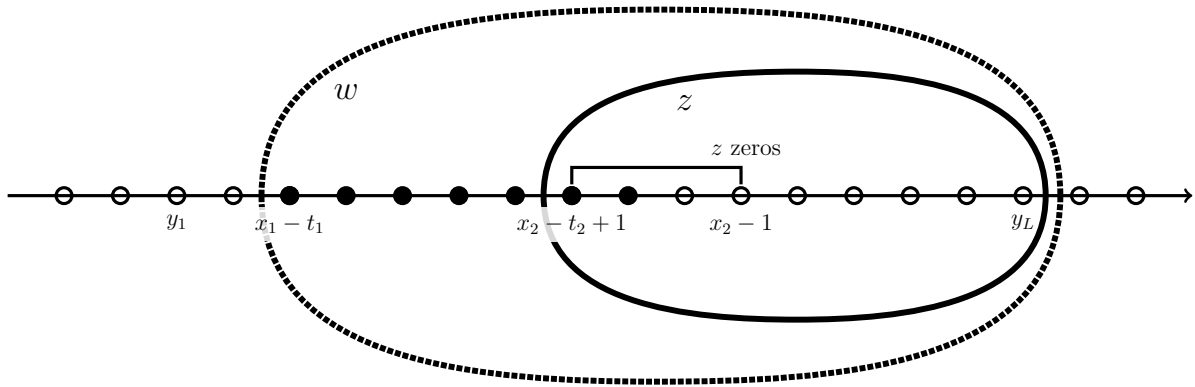
\begin{figure}[htbp]
    \scalebox{.75}{\begin{tikzpicture}[scale=1,ultra thick]
        \draw[->] (-10,0)--(11,0);
        \draw (3,0) circle (4pt) node[below,yshift=-5] {$x_2-1$};
        \draw (8,0) circle (4pt) node[below,yshift=-5] {$y_L$};
        \draw (-7,0) circle (4pt) node[below,yshift=-5] {$y_1$};
        \foreach \ii in {-9,-8,-6,2,4,5,6,7,9,10}
        {
            \draw (\ii,0) circle (4pt);
        }
        \foreach \ii in {-4,-3,-2,-1,1}
        {
            \draw[fill] (\ii,0) circle (4pt);
        }
        \draw[line width=3] (-.5,0) to [out=90,in=180]
        (4,2.2) to [out=0,in=90] (8.4,0)
        to [out=-90,in=0] (4,-2.2) to [out=180,in=-90] (-.5,0);
        \draw (0,0) circle (4pt)
        node[below,yshift=-5,fill=white,opacity=.85]
        {\phantom{$x_2-t_2+1$}};
        \draw[fill] (0,0) circle (4pt)
        node[below,yshift=-5] {$x_2-t_2+1$};
        \node at (2,1.65) {\Large $z$};
        \draw (0,0.2) --++(0,.3)--++(3,0)--++(0,-.3)
        node [above right, yshift=10,xshift=-20] {$z$ zeros};
        \draw[densely dotted, line width=3] (-5.5,0)
        to [out=90,in=180] (1.5,3.3) to [out=0,in=90] (8.66,0)
        to [out=-90,in=0] (1.5,-3.3) to [out=180,in=-90] (-5.5,0);
        \node at (-4,1.85) {\Large $w$};
        \draw (-5,0) circle (4pt)
        node[below,yshift=-5,fill=white,opacity=.85]
        {\phantom{$x_1-t_1$}};
        \draw[fill] (-5,0)
        circle (4pt) node[below,yshift=-5] {$x_1-t_1$};
    \end{tikzpicture}}
    \caption{Integration contours for the kernel
    $K^{\textnormal{tilings}}_{N;\{y_1,\ldots,y_L\}}$
    \eqref{tilings_kernel}.
    The poles at $w=x_1-t_1,x_1-t_1+1,\ldots,x_1-1,x_1$
    are highlighted in black.}
    \label{fig:cont_one}
\end{figure}


\subsection{Noncolliding paths of length $L$} 
\label{sub:noncolliding_paths}

Let us keep $L$ and $y_1,\ldots,y_L$ fixed and consider another
interpretation of a lozenge tiling in terms of noncolliding paths as
in \Cref{fig:tiling_paths}, right. There are $N$ such paths; they
trace lozenges of two types other than
\scalebox{.38}{\begin{tikzpicture}
    \draw [line width=3] (-.25,0)--++(0,-.5)
    --++(.5,.5)--++(0,.5)--cycle;
\end{tikzpicture}},
start at points
\begin{equation}\label{a_through_y_configurations}
    \vec a=(a_1<\ldots<a_N)=\{0,-1,\ldots,-N-L+1\}
    \setminus\{y_1,\ldots,y_L\}
\end{equation}
on the $t=0$ horizontal, and end at points $-N+1,\ldots,-1,0$,
respectively, on the $t=L$ horizontal. Denote the path configuration
at height $t$ by $X^{(L)}_{1}(t)<\ldots<X^{(L)}_{N}(t)$.
Thus, we obtain a random path ensemble
$\{\vec X^{(L)}(t)\}_{t=0}^{L}$ corresponding to our random tiling.

The correlation kernel for the paths $\vec X^{(L)}$
(whose configuration is complementary
to the configuration of the \scalebox{.38}{\begin{tikzpicture}
    \draw [line width=3] (-.25,0)--++(0,-.5)
    --++(.5,.5)--++(0,.5)--cycle;
\end{tikzpicture}} lozenges) can be obtained
from the one for the tilings \eqref{tilings_kernel} by a particle-hole involution (e.g., see \cite[Appendix A.3]{Borodin2000b}):
\begin{equation*}
    \PP\big(\textnormal{paths $\vec X^{(L)}$
    pass through
    points $(x_i,t_i)$, $i=1,\ldots,k$}\big)
    =\det\big[K^{\textnormal{paths}}_{L;\vec a}
    (t_\aind,x_\aind;t_\bind,x_\bind)\big]_{\aind,\bind=1}^{k}
\end{equation*}
with
\begin{equation}\label{paths_kernel}
    K^{\textnormal{paths}}_{L;\vec a}(t_1,x_1;t_2,x_2)=
    \mathbf{1}_{x_1=x_2}\mathbf{1}_{t_1=t_2}
    -K^{\textnormal{tilings}}_{N;\{y_1,\ldots,y_L\}}(t_1,x_1;t_2,x_2),
\end{equation}
as long as the pairwise distinct points $(x_i,t_i)$, $i=1,\ldots,k$,
are inside the trapezoid $\mathcal{T}_{N,L}$ defined in
\Cref{sub:uniformly_random_lozenge_tilings} above.
\begin{remark}\label{rmk:parasite_paths}
    The reason for the last restriction (that the observation points $(x_i,t_i)$ belong to $\mathcal{T}_{N,L}$) is because the particle-hole involution of the configuration of lozenges gives rise not only to the $N$ noncolliding paths $\vec X^{(L)}$ but also to infinitely many trivial paths connecting $(j,0)$ to $(j,L)$, $j\ge1$, and $(j,0)$ to $(j+L,L)$, $j\le -N-L$. These paths correspond to a unique way of extending the lozenge tiling of our polygon to the infinite horizontal strip of height $L$ with $L$ small triangles added at the bottom. Therefore, to capture the correlation structure of the nontrivial paths $\vec X^{(L)}$, the points $(x_i,t_i)$ should be inside $\mathcal{T}_{N,L}$.
\end{remark}
\begin{remark}\label{rmk:extended_Hahn}
    When the starting configuration $\vec a$ is the densely packed one $(0,1,\ldots,N-1)$, the polygon which is tiled reduces to the hexagon. In this case the distribution of the Markov process $\{\vec X^{(L)}(t)\}_{t=0}^{L}$ is described by the extended Hahn kernel expressed through the Hahn orthogonal polynomials, see \cite{johansson2005non}. Asymptotic analysis of the noncolliding paths in the hexagon utilizing this representation of the kernel was performed in \cite{BKMM2003}, \cite{Gorin2007Hexagon}.
\end{remark}


\subsection{Limit $L\to\infty$ of noncolliding paths} 
\label{sub:limit_ltoinfty_paths}

We will now perform a limit transition of our uniformly random tilings
to the noncolliding Bernoulli random walks. Fix $\be\in(0,1)$,
$N\in\Z_{\ge1}$, and $\vec a=(a_1<\ldots<a_N)\in\W N$.
Start the path ensemble
$\{\vec X^{(L)}(t)\}_{t=0}^{L}$ from the points
\begin{equation}\label{y_tilings_scaled}
    \vec a^{(L)}=(a_1-\fl{\be L},\ldots,a_N-\fl{\be L}).
\end{equation}
Clearly, for $L$ large enough we have
\begin{equation*}
    -N-L+1< a_i-\fl{\be L}<0
\end{equation*}
for all $i$, cf. \eqref{a_through_y_configurations}.
Thus, for any
fixed $\vec a\in\W N$ the uniformly random tiling
and the path ensemble $\{\vec X^{(L)}(t)\}_{t=0}^{L}$
are well-defined for large $L$.

The above shifting of the initial configuration of $\vec X^{(L)}$ forces the $N$ noncolliding paths to have asymptotic speed $\be$. This leads to the noncolliding Bernoulli random walk with parameter $\be$:
\begin{proposition}\label{prop:limit_to_Bernoulli}
    As $L\to\infty$, all finite-dimensional
    distributions of
    the path ensemble
    \begin{equation*}
        \{\vec X^{(L)}(t)+\fl{\be L}\}_{t=0}^{L}
    \end{equation*}
    converge
    to those of the noncolliding Bernoulli random walk
    $\vec X(t)$ (defined in
    the Introduction)
    started from the configuration $\vec a$.
\end{proposition}
\begin{proof}
    Because the random lozenge tiling used to construct the path ensemble $\vec X^{(L)}$ is picked uniformly, $\{\vec X^{(L)}(t)\}_{t=0}^{L}$ can be viewed as a Markov process (with time bounded by~$L$). Indeed, the uniformity ensures that the past and the future are conditionally independent given the present configuration \cite[Ch. I.12]{ShiryaevProb}.

    Observe that each conditional probability
    \begin{equation}\label{X_L_conditional_probabilities}
        \PP\big(\vec X^{(L)}(t+1)=\vec b'\mid
        \vec X^{(L)}(t)=\vec b\big),\qquad
        \vec b,\vec b'\in\W N
    \end{equation}
    is simply the ratio of the number of lozenge tilings of a polygon of height $L-(t+1)$ with the bottom boundary determined by $\vec b'$ similarly to \eqref{a_through_y_configurations}, and the number of tilings of a polygon of height $L-t$ with the bottom boundary corresponding to $\vec b$. It suffices to show that the transition probabilities \eqref{X_L_conditional_probabilities} converge to the transition probabilities \eqref{noncolliding_Bernoulli_walks_transitions} of the noncolliding Bernoulli random walk.

    Let us fix $\vec b,\vec b'\in \W N$ such that $b_i'-b_i\in\{0,1\}$
    for all $i$. (Clearly, if these conditions do not hold, then the
    transition probabilities from $\vec b$ to $\vec b'$ in both $\vec
    X^{(L)}(t)+\fl{\be L}$ and $\vec X(t)$ vanish.) We have
    \begin{multline*}
        \PP\big(\vec X^{(L)}(t+1)=\vec b'-\fl{\be L}\mid
        \vec X^{(L)}(t)=\vec b-\fl{\be L}\big)=
        \frac{Z^{N,L-t-1}_{\{m_1',\ldots,m_{L-t-1}'\}}}
        {Z^{N,L-t}_{\{m_1,\ldots,m_{L-t}\}}}\\=
        \prod_{1\le i<j\le L-t-1}\frac{m_j'-m_i'}{j-i}
        \prod_{1\le i<j\le L-t}\frac{j-i}{m_j-m_i}=
        (L-t-1)!\,\frac{\prod\limits_{1\le i<j\le L-t-1}(m_j'-m_i')}{
        \prod\limits_{1\le i<j\le L-t}(m_j-m_i)
        },
    \end{multline*}
    where we used \eqref{ZNL_partition_function}, and
    \begin{equation*}
        \begin{split}
            \{m_1,\ldots,m_{L-t}\}&=\{0,-1,\ldots,-N-L+t+1\}
            \setminus\{b_1-\fl{\be L},\ldots,b_N-\fl{\be L}\},\\
            \{m_1',\ldots,m_{L-t-1}'\}&=\{0,-1,\ldots,-N-L+t+2\}
            \setminus\{b_1'-\fl{\be L},\ldots,b_N'-\fl{\be L}\}.
        \end{split}
    \end{equation*}
    Above we have assumed that $L$ is large enough so that all
    polygons are well-defined.
    We can write (using the notation \eqref{Vandermonde})
    \begin{align*}
				&
        \frac{\prod\limits_{1\le i<j\le L-t-1}(m_j'-m_i')}{
        \prod\limits_{1\le i<j\le L-t}(m_j-m_i)}=
				\frac{\V(-N-L+t+2,\ldots,-1,0)}
				{\V(-N-L+t+1,-N-L+t+2,\ldots,-1,0)}
				\frac{\V(\vec b)}{\V(\vec b')}
        \\&\hspace{85pt}\times
				\prod_{i=1}^{N}
				\Biggl(
					\prod_{\substack{-N-L+t+2\le j\le 0\\j\notin \vec b'-\lfloor \beta L \rfloor }}
					|b_i'-\fl{\be L}-j|^{-1}
					\prod_{\substack{-N-L+t+1\le j\le 0\\j\notin \vec b-\lfloor \beta L \rfloor }}
					|b_i-\fl{\be L}-j|
				\Biggr).
    \end{align*}
		We have 
		$
				\frac{\V(-N-L+t+2,\ldots,-1,0)}
				{\V(-N-L+t+1,-N-L+t+2,\ldots,-1,0)}
				=\frac{1}{(N+L-t-1)!}.
		$
		Next, let us insert the 
		products over $i\ne j$ of $|b_i-b_j|/|b_i'-b_j'|$ into the 
		big product in the previous formula. We obtain
		\begin{multline*}
				\frac{\V(\vec b)}{\V(\vec b')}
				\prod_{i=1}^{N}
				\Biggl(
					\prod_{\substack{-N-L+t+2\le j\le 0\\j\notin \vec b'-\lfloor \beta L \rfloor }}
					|b_i'-\fl{\be L}-j|^{-1}
					\prod_{\substack{-N-L+t+1\le j\le 0\\j\notin \vec b-\lfloor \beta L \rfloor }}
					|b_i-\fl{\be L}-j|
				\Biggr)
				\\=
				\frac{\V(\vec b')}{\V(\vec b)}
				\prod_{i=1}^{N}
				\prod_{j=-N-L+t+2}^{0}
				\frac
				{|b_i-\fl{\be L}-j|^{\mathbf{1}_{b_i-\lfloor \beta L \rfloor \ne j}}}
				{|b_i'-\fl{\be L}-j|^{\mathbf{1}_{b_i'-\lfloor \beta L \rfloor \ne j}}}
				\prod_{i=1}^{N}
				|b_i-\fl{\be L}+N+L-t-1|.
		\end{multline*}
    Using the well-known asymptotics
    for the Gamma function \cite[1.18.(5)]{Erdelyi1953}
    \begin{equation}\label{gamma_asymptotics}
        \frac{\G(L+\alpha)}{\G(L)}\sim L^{\alpha},\qquad L\to+\infty
    \end{equation}
    (where $\alpha$ is fixed), let us collect the last product and the factors involving 
		factorials and write
		\begin{equation*}
			\lim_{L\to+\infty}
			\frac{(L-t-1)!}{(N+L-t-1)!}
			\prod_{i=1}^{N}
			|b_i-\fl{\be L}+N+L-t-1|=(1-\beta)^{N}.
		\end{equation*}
		
		Let us turn to the remaining factors.
		We have to following equivalence as $L\to+\infty$:
		\begin{equation*}
				\prod_{i=1}^{N}
				\prod_{j=-N-L+t+2}^{0}
				\frac
				{|b_i-\fl{\be L}-j|^{\mathbf{1}_{b_i-\lfloor \beta L \rfloor \ne j}}}
				{|b_i'-\fl{\be L}-j|^{\mathbf{1}_{b_i'-\lfloor \beta L \rfloor \ne j}}}
				\sim
				\prod_{i=1}^{N}
				\frac
				{\Gamma(\beta L-b_i)\Gamma(b_i+(1-\beta)L)}
				{\Gamma(\beta L-b_i')\Gamma(b_i'+(1-\beta)L)}.
		\end{equation*}
		If $b_i'=b_i$, the corresponding term is simply $1$, 
		and otherwise it converges to $\beta/(1-\beta)$
		due to \eqref{gamma_asymptotics}.

		Collecting all the terms we see that
		the transition probabilities of $\vec X^{(L)}$
		converge to those of the noncolliding random walk.
		This completes the proof.
\end{proof}


\subsection{Limit $L\to\infty$ in the kernel} 
\label{sub:limit_ltoinfty_kernel}

Let us now take the $L\to\infty$ limit in the kernel
for the process $\vec X^{(L)}(t)$
coming from the uniformly random tilings. The latter kernel is given
by \eqref{tilings_kernel} and \eqref{paths_kernel}.

\begin{proposition}\label{prop:proof_of_intro_Bernoulli}
    Fix $\be\in(0,1)$
    and $\vec a\in\W N$,
    and consider the correlation kernel
    of the
    process $\vec X^{(L)}(t)$
    started from the shifted
    initial configuration
    $\vec a^{(L)}$ as in \eqref{y_tilings_scaled}.
    Then for any fixed\footnote{Clearly, under
    our scaling the restrictions on
    the variables in the kernel
    $K^{\textnormal{paths}}_{L;\vec a^{(L)}}$ imposed in
    \Cref{sub:noncolliding_paths}
    (cf.~\Cref{rmk:parasite_paths}) will eventually
    disappear.}
    $t_{1}\in\Z_{\ge0}$, $t_2\in\Z_{\ge1}$,
    and
    $x_{1,2}\in\Z$
    we have the convergence
    \begin{equation*}
        \lim_{L\to+\infty}
        K^{\textnormal{paths}}_{L;\vec a^{(L)}}
        (t_1,x_1-\fl{\be L};t_2,x_2-\fl{\be L})
        =
        K^{\textnormal{Bernoulli}}_{\vec a; \be}(t_1,x_1;t_2,x_2),
    \end{equation*}
    where
    $K^{\textnormal{Bernoulli}}_{\vec a; \be}$ is
    given by \eqref{K_Bernoulli}.
\end{proposition}
\Cref{prop:proof_of_intro_Bernoulli}
together with
\Cref{prop:limit_to_Bernoulli}
will imply \Cref{thm:intro_Bernoulli}.
\begin{proof}[Proof of \Cref{prop:proof_of_intro_Bernoulli}]
    We first focus on the part of $K^{\textnormal{paths}}_{L;\vec a^{(L)}}(t_1,x_1-\fl{\be L};t_2,x_2-\fl{\be L})$ containing the double contour integral. By substituting the shifted parameters $\vec a^{(L)}$ and $x_{1,2}-\fl{\be L}$ into the kernel given by \eqref{tilings_kernel}, \eqref{paths_kernel} and at the same time shifting both integration variables $z,w$ by $\fl{\be L}$ turns the double contour integral into
    \begin{multline}\label{I_paths_scaled}
        I^{\textnormal{paths}}
        =\frac{t_1!}{(t_2-1)!}
        \frac1{(2\pi\i)^{2}}
        \oint\limits_{c(x_2-t_2+1)}dz\oint\limits_{C(\infty)}dw
        \frac{(z-x_2+1)_{t_2-1}}{(w-x_1)_{t_1+1}}
        \\\times\frac{1}{w-z}
        \frac{(w-\fl{\be L})_{N+L}}{(z-\fl{\be L})_{N+L}}
        \prod_{r=1}^{N}\frac{z-a_r}{w-a_r}.
    \end{multline}
    (Note that this integral enters $K^{\textnormal{paths}}_{L;\vec a^{(L)}}$ with a negative sign which we ignore for now.) Here the $z$ contour $c(x_2-t_2+1)$ encircles the points $x_2-t_2+1,x_2-t_2+2,\ldots,\fl{\be L}$ and not $x_2-t_2-1,x_2-t_2-2,\ldots$, while the $w$ contour $C(\infty)$ encircles $c(x_2-t_2+1)$ and all the $w$ poles of the integrand. For large enough $L$ these $w$ poles are contained inside the intersection $\{x_1-t_1,x_1-t_1+1,\ldots,x_1-1,x_1\}\cap \{a_1,\ldots,a_N\}$. See \Cref{fig:cont_two}.

    Let us split the $w$ integration over $C(\infty)$ into integration over two contours: one encircling all the $w$ poles outside the $z$ contour $c(x_2-t_2+1)$ (denote it by $c'(x_2-t_2)$), and the other one encircling just the $z$ contour $c(x_2-t_2+1)$ (denote it by $c_{\textnormal{out}}(x_2-t_2+1)$). In this second integral we will drag the $w$ contour inside the $z$ contour at the cost of picking the residue at $w=z$. Denote the resulting $w$ contour by $c_{\textnormal{in}}(x_2-t_2+1)$. See \Cref{fig:cont_three}. Thus, \eqref{I_paths_scaled} can be rewritten as follows:
    \begin{multline}
        I^{\textnormal{paths}}
        =
        \frac{t_1!}{(t_2-1)!}
        \frac1{(2\pi\i)^{2}}
        \oint\limits_{c(x_2-t_2+1)}dz\oint\limits_{c'(x_2-t_2)\cup
        c_{\textnormal{in}}(x_2-t_2+1)}dw\,
        \frac{1}{w-z}
        \frac{(z-x_2+1)_{t_2-1}}{(w-x_1)_{t_1+1}}
        \\\times\frac{(w-\fl{\be L})_{N+L}}{(z-\fl{\be L})_{N+L}}
        \prod_{r=1}^{N}\frac{z-a_r}{w-a_r}
        +
        \frac{t_1!}{(t_2-1)!}
        \frac1{2\pi\i}
        \oint\limits_{c(x_2-t_2+1)}
        \frac{(z-x_2+1)_{t_2-1}}{(z-x_1)_{t_1+1}}dz.
        \label{I_paths_scaled_two}
    \end{multline}

    \begin{figure}[htbp]
        \scalebox{.85}{\begin{tikzpicture}[scale=1,ultra thick]
            \draw[->] (-8,0)--(11,0);
            \foreach \ii in {-22,...,30}
            {
                \draw[fill] (\ii/3,0) circle (1.5pt);
            }
            \begin{scope}[shift={(-4,0)},scale=.8]
                \draw[line width=2.5] (-.2,-.2)--++(.4,.4);
                \draw[line width=2.5] (-.2,.2)--++(.4,-.4);
            \end{scope}
            \begin{scope}[shift={(-4+1/3,0)},scale=.8]
                \draw[line width=2.5] (-.2,-.2)--++(.4,.4);
                \draw[line width=2.5] (-.2,.2)--++(.4,-.4);
            \end{scope}
            \begin{scope}[shift={(-3+1/3,0)},scale=.8]
                \draw[line width=2.5] (-.2,-.2)--++(.4,.4);
                \draw[line width=2.5] (-.2,.2)--++(.4,-.4);
            \end{scope}
            \begin{scope}[shift={(-2+2/3,0)},scale=.8]
                \draw[line width=2.5] (-.2,-.2)--++(.4,.4);
                \draw[line width=2.5] (-.2,.2)--++(.4,-.4);
            \end{scope}
            \begin{scope}[shift={(-1/3,0)},scale=.8]
                \draw[line width=2.5] (-.2,-.2)--++(.4,.4);
                \draw[line width=2.5] (-.2,.2)--++(.4,-.4);
            \end{scope}
            \draw[line width=2.4] (-3.16,0) to [out=90,in=180]
            (3,3) to [out=0,in=90] (9.12,0) to [out=-90,in=0]
            (3,-3) to [out=180,in=-90] (-3.16,0);
            \draw[densely dotted, line width=2.4] (-4.82,0)
            to [out=90,in=180]
            (2.2,3.7) to [out=0,in=90] (9.22,0) to [out=-90,in=0]
            (2.2,-3.7) to [out=180,in=-90] (-4.82,0);
            \node[below, yshift=-5, fill=white, opacity=.85]
            at (-3,0) {\phantom{$x_2-t_2+1$}};
            \draw (-3,.25)--++(0,-.5) node[below, yshift=-5]
            at (-3,0) {$x_2-t_2+1$};
            \node[below, yshift=-5, fill=white, opacity=.85]
            at (9,0) {\phantom{$\fl{\be L}$}};
            \draw (9,.25)--++(0,-.5) node[below, yshift=-5]
            at (9,0) {$\fl{\be L}$};
            \node[above, yshift=5, fill=white, opacity=.85]
            at (-4-2/3,0) {\phantom{$x_1-t_1$}};
            \draw (-4-2/3,.25)--++(0,-.5) node[above, yshift=5]
            at (-4-2/3,0) {$x_1-t_1$};
            \draw (-1/3,.25)--++(0,-.5) node[above, yshift=5]
            at (-1/3,0) {$x_1$};
            \node at (1,2.55) {\Large $z$};
            \node at (-2.5,3.35) {\Large $w$};
        \end{tikzpicture}}
        \caption{Integration contours for
        $I^{\textnormal{paths}}$ \eqref{I_paths_scaled}.
        Note that the $z$ contour grows with $L$ but the
        $w$ poles (highlighted by crosses) do not depend on $L$.}
        \label{fig:cont_two}
    \end{figure}
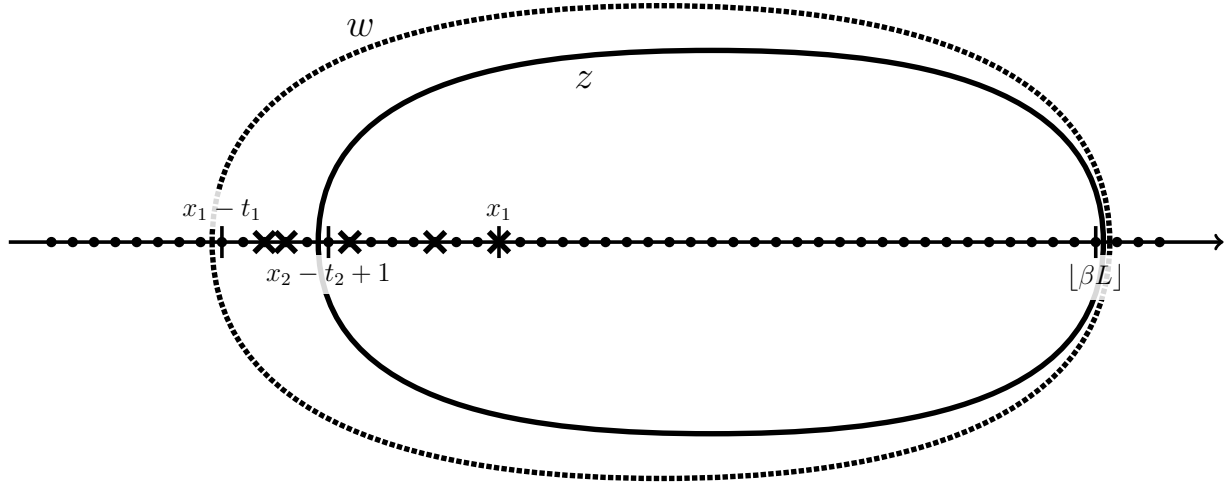

    \begin{figure}[htbp]
        \scalebox{.85}{\begin{tikzpicture}[scale=1,ultra thick]
            \draw[->] (-8,0)--(11,0);
            \foreach \ii in {-22,...,30}
            {
                \draw[fill] (\ii/3,0) circle (1.5pt);
            }
            \begin{scope}[shift={(-4,0)},scale=.8]
                \draw[line width=2.5] (-.2,-.2)--++(.4,.4);
                \draw[line width=2.5] (-.2,.2)--++(.4,-.4);
            \end{scope}
            \begin{scope}[shift={(-4+1/3,0)},scale=.8]
                \draw[line width=2.5] (-.2,-.2)--++(.4,.4);
                \draw[line width=2.5] (-.2,.2)--++(.4,-.4);
            \end{scope}
            \begin{scope}[shift={(-3+1/3,0)},scale=.8]
                \draw[line width=2.5] (-.2,-.2)--++(.4,.4);
                \draw[line width=2.5] (-.2,.2)--++(.4,-.4);
            \end{scope}
            \begin{scope}[shift={(-2+2/3,0)},scale=.8]
                \draw[line width=2.5] (-.2,-.2)--++(.4,.4);
                \draw[line width=2.5] (-.2,.2)--++(.4,-.4);
            \end{scope}
            \begin{scope}[shift={(-1/3,0)},scale=.8]
                \draw[line width=2.5] (-.2,-.2)--++(.4,.4);
                \draw[line width=2.5] (-.2,.2)--++(.4,-.4);
            \end{scope}
            \draw[line width=2] (-3.15,0) to [out=90,in=180]
            (3,3) to [out=0,in=90] (9.2,0) to [out=-90,in=0]
            (3,-3) to [out=180,in=-90] (-3.15,0);
            \draw[densely dotted, line width=2] (-3.08,0)
            to [out=90,in=180]
            (-1.5,1.6) to [out=0,in=90] (-.11,0) to [out=-90,in=0]
            (-1.5,-1.6) to [out=180,in=-90] (-3.08,0);
            \draw[densely dotted, line width=2] (-3.22,0)
            to [out=90,in=0]
            (-4,1.2) to [out=180,in=90] (-4.86,0) to [out=-90,in=180]
            (-4,-1.2) to [out=0,in=-90] (-3.22,0);
            \node[below, yshift=-5, fill=white, opacity=.85]
            at (-3,0) {\phantom{$x_2-t_2+1$}};
            \draw (-3,.25)--++(0,-.5) node[below, yshift=-5]
            at (-3,0) {$x_2-t_2+1$};
            \node[below, yshift=-5, fill=white, opacity=.85]
            at (9,0) {\phantom{$\fl{\be L}$}};
            \draw (9,.25)--++(0,-.5) node[below, yshift=-5]
            at (9,0) {$\fl{\be L}$};
            \node[above, yshift=5, fill=white, opacity=.85]
            at (-4-2/3,0) {\phantom{$x_1-t_1$}};
            \draw (-4-2/3,.25)--++(0,-.5) node[above, yshift=5]
            at (-4-2/3,0) {$x_1-t_1$};
            \node[above, yshift=5, fill=white, opacity=.85]
            at (-1/3,0) {\phantom{$x_1$}};
            \draw (-1/3,.25)--++(0,-.5) node[above, yshift=5]
            at (-1/3,0) {$x_1$};
            \node at (1,3.25) {\Large $z$};
            \node at (.95,1.65)
            {\Large $c_{\textnormal{in}}(x_2-t_2+1)$};
            \node at (-5.2,1.53) {\Large $c'(x_2-t_2)$};
        \end{tikzpicture}}
        \caption{Integration contours in the double
        integral in \eqref{I_paths_scaled_two}.}
        \label{fig:cont_three}
    \end{figure}
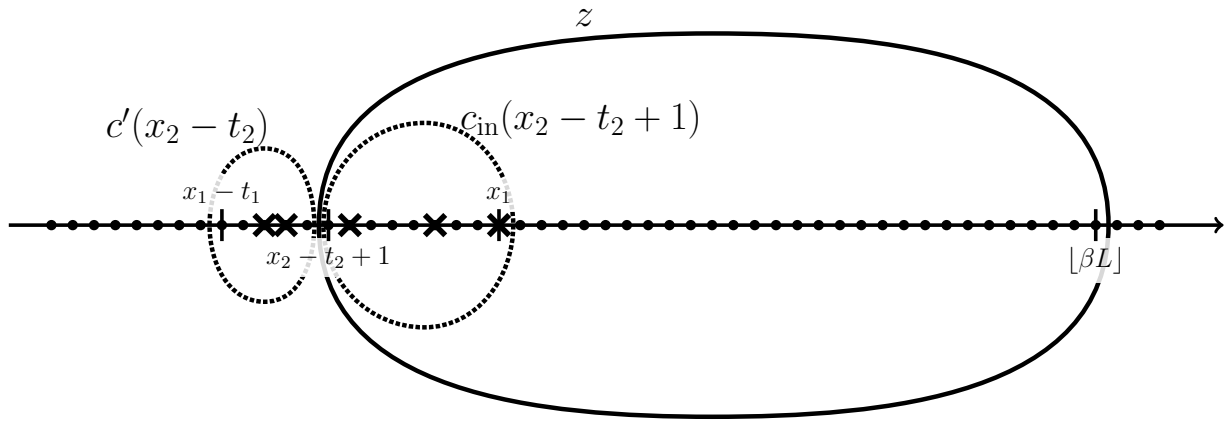

    The single integral in \eqref{I_paths_scaled_two}
    can be evaluated using the results of
    Section 6.2 in
    \cite{Petrov2012}, it is equal to
    \begin{equation}\label{new_single_integral}
        \mathbf{1}_{x_1\ge x_2}\frac{(t_2-t_1)_{x_1-x_2}}{(x_1-x_2)!}.
    \end{equation}

    In the double contour integral in \eqref{I_paths_scaled_two} we first
    note that for $L$ fixed but large enough, due to the presence of the
    polynomial $(z-\fl{\be L})_{N+L}$ in the denominator, the integrand
    decays rapidly as $z\to\infty$. Thus, the $z$ integration contour can
    be replaced by the vertical line from $x_2-t_2+\frac12-\i\infty$ to
    $x_2-t_2+\frac12+\i\infty$ traversed from bottom to top, yielding
    a new
    minus sign in front of the double contour
    integral (cf. \Cref{fig:cont_intro}).

    The $z$
    integral over the vertical line converges uniformly in $L$.
    Indeed, observe that
    \begin{equation*}
        \left|
        \frac{w+k}{x+\i y+k}
        \right|<\frac{1}{1+C|y|/|k|},\qquad k\in\Z,
    \end{equation*}
    for some $C>0$, where $z=x+\i y$, and $C$ is uniform in $w$ belonging to a bounded contour.
    For large $|y|$ and $L>L_0$ the product of the above quantities over $k$ as in
    \eqref{I_paths_scaled_two}
    can be bounded by a constant independent of $L$
    times a fixed (but arbitrarily large) negative power of $|y|$.
    Here we also used the fact that for fixed $y$ the
    infinite product over $k$ diverges to infinity.

    Thus, the integration contours do not depend on $L$, and we can pass
    to a pointwise limit as $L\to\infty$ in the integrand.
    Since $w,z\notin\Z$ on our contours, we can write
    \begin{multline}\label{from_tilings_to_Bernoulli_estimation}
        \frac{(w-\fl{\be L})_{N+L}}{(z-\fl{\be L})_{N+L}}=
        \frac{\G(w+L-\fl{\be L}+N)}{\G(z+L-\fl{\be L}+N)}
        \frac{\G(z-\fl{\be L})}{\G(w-\fl{\be L})}
        \\=\frac{\G(w+L-\fl{\be L}+N)}{\G(z+L-\fl{\be L}+N)}
        \frac{\G(-w+1+\fl{\be L})}{\G(-z+1+\fl{\be L})}
        \frac{\sin (\pi w)}{\sin (\pi z)},
    \end{multline}
    where in the second equality we used
    \begin{equation}\label{Gamma_function_flip}
        \G(u)=\frac{\pi}{\sin(\pi u)\G(1-u)}.
    \end{equation}
    Let us employ the Stirling asymptotics
    for the
    Gamma function \cite[1.18.(2)--(3)]{Erdelyi1953}
    which can be formulated as
    \begin{equation}\label{Stirling_Gamma}
        \G(L+\al)=\big(1+O(L^{-1})\big)\sqrt{2\pi}\exp\Big(
        \big(L+\al-\tfrac12\big)\log L-L\Big),
        \qquad L\to+\infty,
    \end{equation}
    where $\al\in\C$ if fixed and the remainder $O(L^{-1})$ is uniform
    in $\al$ belonging to compact subsets of $\C$.
    Thus, continuing \eqref{from_tilings_to_Bernoulli_estimation},
    \begin{equation*}
        \frac{(w-\fl{\be L})_{N+L}}{(z-\fl{\be L})_{N+L}}=
        \frac{\sin (\pi w)}{\sin (\pi z)}
        \left(\frac{1-\be}{\be}\right)^{w-z}\big(1+O(L^{-1})\big).
    \end{equation*}

    Finally, the summands not involving contour integrals
    coming from \eqref{tilings_kernel}, \eqref{paths_kernel},
    and \eqref{new_single_integral}
    can be simplified as
    \begin{multline*}
        \mathbf{1}_{x_1=x_2}\mathbf{1}_{t_1=t_2}+
        \mathbf{1}_{t_1<t_2}\mathbf{1}_{x_2\le x_1}
        \frac{(x_1-x_2+1)_{t_2-t_1-1}}{(t_2-t_1-1)!}
        -
        \mathbf{1}_{x_1\ge x_2}
        \frac{(t_2-t_1)_{x_1-x_2}}{(x_1-x_2)!}
        \\=
        \mathbf{1}_{x_1\ge x_2}
        \mathbf{1}_{t_1>t_2}
        (-1)^{x_1-x_2+1}\binom{t_1-t_2}{x_1-x_2}
    \end{multline*}
    (note that all of them involve only the difference $x_1-x_2$
    which is not affected by the shift by $\fl{\be L}$).
    This coincides with the summand not containing
    integrals in \eqref{K_Bernoulli}.
    This completes the proof of
    \Cref{prop:proof_of_intro_Bernoulli} and hence of
    \Cref{thm:intro_Bernoulli}.
\end{proof}

\begin{remark}\label{rmk:moving_int_contours}
    The argument in the above proof implies in particular that the integration in
    $K^{\textnormal{Bernoulli}}_{\vec a; \be}$ \eqref{K_Bernoulli} can
    be alternatively performed over a shifted contour $z$. This
    contour can be shifted as far as to the vertical  line traversed
    from $x_2-\frac12-\i\infty$ to $x_2-\frac12+\i\infty$. Indeed, the
    difference between the two expressions is equal to the residue at
    $z=w$ integrated over a certain part of the $w$ contour; it is the same
    as the single integral in \eqref{I_paths_scaled_two} but over a
    contour which does not contain any poles inside, and thus
    vanishes.
\end{remark}



\section{Setup of the asymptotic analysis} 
\label{sec:setup_of_the_asymptotic_analysis}

Here we explain the relevance of the function $\Sfin'(\z)$ defined
in \eqref{S_T_prime} for the asymptotics of the correlation
kernel of the noncolliding Bernoulli random walks.

\subsection{A change of variables} 
\label{sub:a_change_of_variables}

Changing the variables as $z=t_2\z+x_2$, $w=t_2\w+x_2$, and employing the shorthand notation
\begin{equation}\label{delta_notation}
    \D t=t_1-t_2,\qquad \D x=x_1-x_2
\end{equation}
turns the kernel \eqref{K_Bernoulli}  of the noncolliding Bernoulli random walk into
\begin{multline}
    K^{\textnormal{Bernoulli}}_{\vec a; \be}(t_1,x_1;t_2,x_2)
    =
    \mathbf{1}_{\D x\ge0}\mathbf{1}_{\D t>0}
    (-1)^{\D x+1}
    \binom{\D t}{\D x}
    \\+
    \frac1{(2\pi\i)^{2}}
    \int\limits_{-1+\frac12t_2^{-1}-\i\infty}
    ^{-1+\frac12t_2^{-1}+\i\infty}d\z
    \oint\limits_{\textnormal{all $\w$ poles}}d\w\,
    \frac{(t_2+\D t)!\cdot t_2}{(t_2-1)!}
    \frac{(t_2\z+1)_{t_2-1}}{(t_2\w-\D x)_{t_2+\D t+1}}
    \\\times\frac{1}{\w-\z}
    \frac{\sin(\pi t_2\w)}{\sin(\pi t_2\z)}
    \left(\frac{1-\be}{\be}\right)^{t_2(\w-\z)}
    \prod_{r=1}^{N}\frac{t_2\z+x_2-a_r}{t_2\w+x_2-a_r}.
    \label{K_Bernoulli_general}
\end{multline}
Here $\z$ is integrated over a vertical line (which crosses the real line to the right of $-1$), and the $\w$ integration contour (a circle or a union of two circles, cf. \Cref{fig:cont_intro}) must encircle all the $\w$ poles of the integrand except $\w=\z$. Note that now these poles all belong to $\{-1+t_2^{-1}(\D x-\D t),\ldots,t_2^{-1}(\D x-1),t_2^{-1}\D x\}$.

From \eqref{K_Bernoulli_general} we see that by shifting the initial data $\vec a\in\W N$ it is possible to take $x_2=0$. Since the initial data is arbitrary and its finite shifts do not change our Assumptions \ref{ass:density} and \ref{ass:intermediate_behavior}, throughout the sequel without loss of generality we may and will assume that $x_2=0$, and so $x_1=\D x\in\Z$ is fixed throughout the analysis. Moreover, since we aim to study the asymptotic behavior of $K^{\textnormal{Bernoulli}}_{\mathfrak{A}(N); \be} (t_1+T(N),x_1;t_2+T(N),0)$ (cf. \Cref{Theorem_main}) and finite shifts in the $t$ parameters can be incorporated into $T=T(N)$, we may also assume that $t_2=T$ and $t_1=\D t+T$, where $\D t\in\Z$ is fixed.


\subsection{Definition of the function $\Sfin(\z)$} 
\label{sub:function_sfin_z_}

With the notation explained in \Cref{sub:a_change_of_variables},
rewrite the integrand in \eqref{K_Bernoulli_general}
(without $1/(\w-\z)$)
as follows:
\begin{multline}
    \frac{(T+\D t)!\cdot T}{(T-1)!}
    \frac{(T\z+1)_{T-1}}{(T\w-\D x)_{T+\D t+1}}
    \frac{\sin(\pi T\w)}{\sin(\pi T\z)}
    \left(\frac{1-\be}{\be}\right)^{T(\w-\z)}
    \prod_{r=1}^{N}\frac{T\z-a_r}{T\w-a_r}
    \\=
    \exp\Big\{T\big(\Sfin(\z)-\Sfin(\w)\big)\Big\}
    \frac{(T+\D t)!\cdot T}{(T-1)!}
    \frac{(T\w+1)_{T-1}}{(T\w-\D x)_{T+\D t+1}}
    ,\label{integrand_asymptotics}
\end{multline}
where
\begin{equation}\label{S_T_action_prelim}
    \Sfin(\z)=
    \frac 1T\sum_{r=1}^{N}\log\Big(\z-\frac{a_r}T\Big)
    +\frac 1T\sum_{i=1}^{T-1}\log\Big(\z+\frac iT\Big)
    -\frac 1T\log(\sin(\pi T\z))-\z\log(\be^{-1}-1),
\end{equation}
Let us discuss the choice of branches of the logarithms. Because $\Sfin(\z)$ is exponentiated in \eqref{integrand_asymptotics}, different choices of branches lead to the same integrand. However, a certain particular choice makes $\Sfin(\z)$ holomorphic in the upper half plane $\HH=\{z\in\C\colon \Im z>0\}$, which will be convenient in \Cref{sec:asymptotics_of_the_kernel}. Let us restrict our attention to $\HH$, the situation in the lower half plane is analogous (however, one clearly cannot choose a branch making $\Sfin(\z)$ holomorphic in the whole complex plane).

The standard branch of the logarithm, denoted by $\log z$, has the cut along the negative real axis, and takes positive real values for real $z>1$. Let $\log_{\HH} z$ denote a branch in the upper half plane which extends holomorphically to $\R\setminus\{0\}$ and has the cut along the negative imaginary axis:
\begin{equation*}
    \log_{\HH} z=\log (ze^{-\i\pi/2})+\i\pi/2.
\end{equation*}
For $z\in\HH$ the branches $\log z$ and $\log_\HH z$ coincide. We will use $\log_{\HH}$ for the logarithms of $\z-a_r/T$ and $\z+i/T$ in \eqref{S_T_action_prelim}. Next, simply plugging $\sin(\pi T\z)$ into any of these logarithms does not produce a continuous function in $\HH$. Let us use \eqref{Euler_sine} instead, and define
\begin{equation}\label{log_branch}
    \log(\sin\pi z)_{\HH}=\log(\pi z)+\sum_{k=1}^{\infty}\bigl(\log(1+z/k)+\log(1-z/k)\bigr).
\end{equation}
In the right-hand side the logarithms are standard, and we mean direct substitution. The series in $k$ converges for any fixed $z\in\HH$ because it is bounded by the sum of $C/k^2$. One can check that alternatively \eqref{log_branch} can be written as
\begin{equation}\label{log_branch2}
    \log(\sin\pi z)_{\HH}=
    \log_{\HH}(\sin\pi z)-
    2\pi\i\left\lfloor
    \tfrac12\Re(z)+\tfrac12
    \right\rfloor,
\end{equation}
where $\log_{\HH}(\sin\pi z)$ is the direct substitution.
This expression provides a holomorphic continuation
of $\log(\sin\pi z)_{\HH}$ into $\R\setminus\Z$.
From \eqref{log_branch2} it readily follows that
\begin{equation}\label{log_branch_large_im}
    \log\bigl(\sin(\pi (x+\i y))\bigr)_{\HH}
    =-\i \pi  (x+\i y)+{\i \pi }/{2}-\log 2+o(1),\qquad y\to+\infty,
\end{equation}
uniformly in $x\in\R$ (the remainder $o(1)$ is periodic in $x$).

Therefore, the function $\Sfin(\z)$ takes the form
\begin{equation}\label{S_T_action}
    \begin{split}
        &\Sfin(\z)=
        \frac 1T\sum_{r=1}^{N}\log_{\HH}\Big(\z-\frac{a_r}T\Big)
        +\frac 1T\sum_{i=1}^{T-1}\log_{\HH}\Big(\z+\frac iT\Big)
        \\&
        \hspace{80pt}-\frac 1T\log_{\HH}(\sin(\pi T \z))
        +
        \frac {2\pi\i}T
        \left\lfloor
        \tfrac12\Re(T\z)+\tfrac12
        \right\rfloor
        -\z\log(\be^{-1}-1).
    \end{split}
\end{equation}
With these choices of branches it
becomes holomorphic in $\HH$,
and extends to
$\z\in\R$ everywhere except the singularities.
Recall the notation $\mathfrak{L}_T=\mathfrak{L}_{T(N)}$ \eqref{eq_L_def} and
$\mathfrak{A}=\mathfrak{A}(N)$, and denote
\begin{equation}
    \mathfrak{l}_T=\mathfrak{l}_{T(N)}=\{\ldots,-1-\tfrac 2T,-1-\tfrac 1T,-1\}\cup\{0,\tfrac 1T,\tfrac 2T,\ldots\},
    \quad
    \mathfrak{a}=\mathfrak{a}(N)=\{\tfrac{a_1}T,\tfrac{a_2}T,\ldots,\tfrac{a_N}T\}.
    \label{L_A_small_set_notation}
\end{equation}
The set of (non-removable) singularities of $\Sfin(\z)$ is $\mathfrak{l}_T\Delta\mathfrak{a}$
(the symmetric difference)
because of the cancellations in \eqref{S_T_action}
with the help of \eqref{log_branch}.

The function $\Sfin'(\z)$ defined by \eqref{S_T_prime}
is simply the derivative of $\Sfin(\z)$.
Note that this derivative does not depend on choices of the branches.

\medskip

We will study the asymptotic behavior
of the kernel \eqref{K_Bernoulli_general}
by means of the steepest descent method.
That is, we will find
\emph{critical points} of the function $\Sfin(\z)$
(i.e., where $\Sfin'(\z)=0$) and deform the contours so that
they pass through these critical points and are steepest descent
for $\Re \Sfin(\z)$
(that is, $\Re \Sfin(\z)$ on these contours
decreases or increases the most).
As was first noted in \cite[Section 3.2]{Okounkov2002},
having a pair of nonreal complex conjugate
simple critical points
 $\zfin$ and $\zbfin$ (plus certain properties of the integration contours)
leads to the discrete sine kernel.



\section{Existence of nonreal
critical points: proof of \Cref{Proposition_roots_positioning}} 
\label{sec:proof_of_proposition_roots_positioning}

In this section we deal with properties of $\Sfin'(\z)$
\eqref{S_T_prime}, and prove \Cref{Proposition_roots_positioning}
(stating that the equation $\Sfin'(\z)=0$ has a unique root $\zfin=\zfin(N)$ in $\HH$,
and it is
uniformly bounded away from the real line and infinity)
through a
series of lemmas.

\begin{lemma}\label{lemma:S_T_critical_points_equation}
    The equation $\Sfin'(\z)=0$
    has at most one pair of nonreal complex conjugate
    roots.
\end{lemma}
\begin{proof}
    The sum over $|j|<M$ in \eqref{eq_pv_def} converges, as $M\to+\infty$,
    uniformly on compact sets in $\C$ to the corresponding principal value sum (i.e., the
    left-hand side of \eqref{eq_pv_def}).
    Therefore, by Hurwitz's theorem, for the
    purpose of counting critical points
    it is enough to prove that the following equation (approximating
    $\Sfin'(\z)=0$)
    \begin{equation}\label{S_T_critical_points_approximate_equation}
        \sum_{r=1}^{N}\frac{1}{T\z-a_r}
        -\sum_{j\in \mathfrak{L}_T\cap\{-M,\ldots,M\}}\frac{1}{T\z-j}
        =\log(\be^{-1}-1),
    \end{equation}
    has at most one pair of nonreal complex roots for all large enough $M$.

    Let $d$ be the size of $(\mathfrak{L}_T\Delta\mathfrak{A})\cap\{-M,\ldots,M\}$, this is the number of poles in the left-hand side of \eqref{S_T_critical_points_approximate_equation} after canceling out equal terms with opposite signs. Multiplying by the common denominator turns equation \eqref{S_T_critical_points_approximate_equation} into a polynomial equation of degree $d$ if $\be\ne\frac12$, and of degree $d-1$ otherwise (when the logarithm in the right-hand side vanishes).

    Let us demonstrate that \eqref{S_T_critical_points_approximate_equation} already has at least $d-3$ real roots. The left-hand side of \eqref{S_T_critical_points_approximate_equation} has $d$ poles which divide the real line onto $d-1$ segments of finite length, plus two semi-infinite rays. These $d$ poles are of two types (see \Cref{fig:roots} for an example):
    \begin{enumerate}[$\bullet$]
        \item For all $a_i\in\{-T+1,\ldots,-1\}$, the pole comes from the term $\frac{1}{T\z-a_i}$.
        \item
        For points of
        $\{-M,\ldots,-T-1,-T\}\cup\{0,1,\ldots,M\}$ which are not equal to any $a_i$
        the pole comes from the term
        $-\frac{1}{T\z-\ell}$ of the opposite sign.
    \end{enumerate}
    Clearly, on a segment
    between any two poles of the same sign
    the left-hand side of \eqref{S_T_critical_points_approximate_equation} takes all values
    between $-\infty$ and $+\infty$, and thus equation \eqref{S_T_critical_points_approximate_equation} has at least one root
    on this segment.
    Among the $d-1$ segments of finite length, at most two
    have endpoints which are singularities of
    different types, and thus the presence of a root there is not guaranteed.
    Thus, there are at least $d-3$ real roots.
    Because the coefficients of \eqref{S_T_critical_points_approximate_equation} are real,
    its nonreal roots come in complex conjugate pairs, and so this equation cannot have more than one such pair of nonreal roots.
\end{proof}

    \begin{figure}[htbp]
        \includegraphics[width=.7\textwidth]{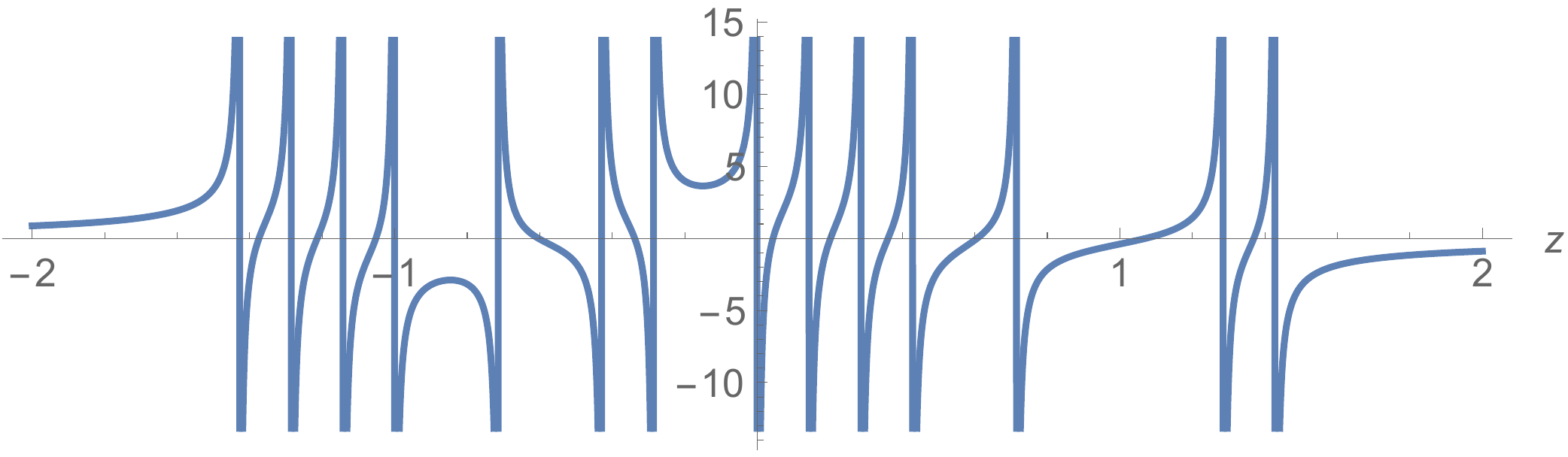}
        \caption{Plot of the left-hand side of \eqref{S_T_critical_points_approximate_equation} as a function
        of $\z\in\R$. The parameters are
        $\mathfrak{A}=(-5,-3,-2,4, 6, 7, 8)$,
        $T=7$, $M=10$, and $d=14$.
        Note that for
        $\be\ne \frac12$ equation \eqref{S_T_critical_points_approximate_equation} has
        one extra real root belonging to one of the semi-infinite rays.}
        \label{fig:roots}
    \end{figure}

\Cref{lemma:S_T_critical_points_equation} implies that there is at most one
critical point in the upper half plane. In the rest of this section we
show its existence, and
obtain a
more precise control on the position of this critical point. The complex equation
$\Sfin'(\z)=0$ \eqref{S_T_prime}
is equivalent to a pair of real equations in $\z=x+\i y$, $x\in\R$, $y\in\R_{>0}$:
\begin{align}
    \label{eq_imaginary_eq}
    0=\Im \Sfin'(x+\i y)&=
    -\frac{1}{T}\sum_{r=1}^{N}\frac{y}{y^2+(x-a_r/T)^{2}}
    +\frac{1}{T}\sum_{j\in \mathfrak{L}_T}\frac{y}{y^2+(x-j/T)^{2}},\\
    \label{eq_real_eq}
    0=\Re\Sfin'(x+\i y)&=
    \sum_{r=1}^{N}\frac{1}{T}\frac{x-a_r/T}{y^2+(x-a_r/T)^{2}}
    -\pv\sum_{j\in \mathfrak{L}_T}\frac{1}{T}\frac{x-j/T}{y^2+(x-j/T)^{2}}
    -\log(\be^{-1}-1).
\end{align}
Note that the infinite sum in \eqref{eq_imaginary_eq} is absolutely convergent,
while in \eqref{eq_real_eq} we need to use the principal value summation.

We start from \eqref{eq_imaginary_eq}, and rewrite it in a more compact form.
For a discrete subset $U\subset\R$, define the atomic measure
\begin{equation}\label{atomic_measure}
    \m_{T}[U]=\frac{1}{T}\sum_{u\in U}\delta_{u}
\end{equation}
(note that it is not necessarily a probability or even a finite measure), and denote
by
\begin{equation}\label{Cauchy_measure}
    \CC_{y}(u)=\frac{y}{\pi(y^2+u^2)}
\end{equation}
the Cauchy probability density on $\R$ rescaled by $y>0$. Using this notation,
rewrite \eqref{eq_imaginary_eq} as
\begin{equation}\label{im_Sprime_measures}
    0=\frac{1}{\pi}\Im \Sfin'(x+\i y)=
    -
    \bigl(\m_{T}[\mathfrak{a}]*\CC_y\bigr)(x)+
    \bigl(\m_{T}[\mathfrak{l}_{T}]*\CC_y\bigr)(x),
\end{equation}
where ``$*$'' means the usual convolution of measures.

\begin{lemma} \label{lemma_imaginary_part_1}
 Under Assumptions \ref{ass:density} and \ref{ass:intermediate_behavior},
 for each $0<\delta<1$ there exists
 $\eps_0>0$ (which may depend
 on constants in our assumptions but not on the choice of
 $\mathfrak{A}(N)$), such that for any $0<\eps<\eps_0$ there is
 $N_0\in\Z_{\ge1}$, and for all $N>N_0$ (see \Cref{fig:compact_in_C_with_signs_of_Im_S_prime}):
 \begin{enumerate}[$\bullet$]
    \item $\Im \Sfin'(x+\i y)>0$ for all $(x,y)$ such that
    $\sqrt{x^2+y^2}=\eps^{-1}$ and $y\ge\eps$,
    and for all
    $(x,y)=(x,\eps)$, where
    $x\in[-\eps^{-1},-1-\eps^{\delta})\cup(\eps^{\delta},\eps^{-1}]$;
    \item $\Im \Sfin'(x+\i y)<0$ for all $(x,y)=(x,\eps)$, where
    $x\in(-1+\eps^{\delta},-\eps^{\delta})$.
 \end{enumerate}
\end{lemma}

\begin{figure}[htpb]
    \begin{tikzpicture}[scale=1,thick]
        \def\ee{.22}
        \def\eedd{.6}
        \draw[->] (-5,0)--(5,0);
        \draw[->] (0,-.5)--(0,5);
        \draw[fill] (0,0) circle (2pt) node [below left] {$\phantom{\eps^{\delta}}0$};
        \draw[fill] (-2,0) circle (2pt) node [xshift=5,below] {$-1\phantom{\eps^{\delta}}$};
        \draw[fill] (\eedd,0) circle (1pt) node [below] {$\eps^{\delta}$};
        \draw[fill] (1/\ee,0) circle (1pt) node [below] {$\eps^{-1}$};
        \draw[fill] (-2-\eedd,0) circle (1pt) node [xshift=5,below left] {$-1-\eps^{\delta}$};
        \def\centerarc[#1](#2)(#3:#4:#5)
        {\draw[#1] ($(#2)+({#5*cos(#3)},{#5*sin(#3)})$) arc (#3:#4:#5)}
        \centerarc[line width = 3.4](0,0)(3.5:176.5:1/\ee);
        \draw[line width = 3.4] (\eedd,\ee)--(1/\ee+.05,\ee);
        \draw[line width = 3.4] (-1/\ee-.05,\ee)--(-2-\eedd,\ee);
        \draw[line width = 0.7, dotted] (-1/\ee,\ee)--(1/\ee,\ee);
        \draw[line width = 3.4] (-2+\eedd,\ee)--(-\eedd,\ee);
        \node at (1,4.1) {$+$};
        \node[above] at (2.5,\ee) {$+$};
        \node[above] at (-3.5,\ee) {$+$};
        \node[above] at (-1,\ee) {$-$};
        \draw[blue, densely dashed, line width = 2] (-\eedd/2,\ee)
        to [out=90,in=180] (-\eedd/3,1)
        to [out=0,in=90] (-\ee/2,\ee);
        \draw[blue, densely dashed, line width = 2] (\eedd/2,\ee)
        to [out=90,in=0] (-1,2)
        to [out=180,in=90] (-2+\eedd/2,\ee);
        \draw[blue, densely dashed, line width = 2] (-2-\eedd/2,\ee)
        to [out=90,in=180] (-2-\eedd/3,.7)
        to [out=0,in=90] (-2-\ee/2,\ee);
    \end{tikzpicture}
    \caption{Signs of $\Im \Sfin'(\z)$ along the curves described in \Cref{lemma_imaginary_part_1}. Blue dashed curves represent a possible part of the boundary of the set $\mathcal U$ inside $\mathcal D$, see \Cref{lemma_imaginary_part_2} below.}
    \label{fig:compact_in_C_with_signs_of_Im_S_prime}
\end{figure}
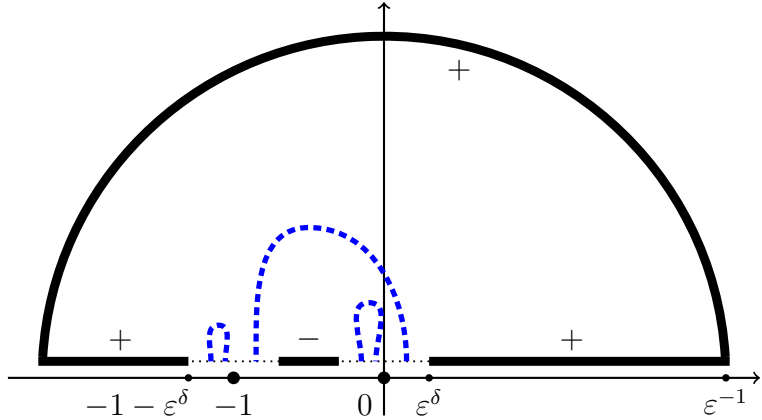

\begin{remark}
    The presence of $\delta$ in this lemma and \Cref{lemma_imaginary_part_2} below
    is not essential. However, $\delta$ is put here to better
    link these statements with
    \Cref{lemma_real_part} below where $\delta$ plays an important role.
\end{remark}
\begin{proof}[Proof of \Cref{lemma_imaginary_part_1}]
    Fix $\eps>0$. As $N$ (and thus $T$) grows, the
    absolutely convergent sum
    $\bigl(\m_{T}[\mathfrak{l}_{T}]*\CC_y\bigr)(x)$
    is a Riemann sum for the corresponding integral,
    and it approximates the integral
    uniformly on compact subsets
    of the upper half plane (and in particular, for $(x,y)$
    in each of the sets described in the hypotheses of the lemma).
    Thus,
    for any $c>0$ there exists $N_0$
    such that for all $N>N_0$,
    \begin{equation*}
        \left|\bigl(\m_{T}[\mathfrak{l}_{T}]*\CC_y\bigr)(x)-
        \frac{1}{\pi}
        \bigg(\int_{-\infty}^{-1}+
        \int_{0}^{\infty}\bigg)
        \frac{y\,du}{y^2+(x-u)^2}\right|<c.
    \end{equation*}
    The integral above can be explicitly evaluated,
    it is equal to
    \begin{equation*}
        1+\frac{1}{\pi}\tan ^{-1}\left(\frac{x}{y}\right)
        -\frac{1}{\pi}\tan ^{-1}\left(\frac{x+1}{y}\right).
    \end{equation*}
    For small $y$, this expression is close to $1$ if $x\in(-\infty, -1)\cup(0,+\infty)$,
    and close to $0$ if $x\in(-1,0)$. Moreover, for
    $\sqrt{x^2+y^2}=\eps^{-1}$ and $y\ge\eps$ this expression is close to $1$, too.

    Let us now deal with
    $\bigl(\m_{T}[\mathfrak{a}]*\CC_y\bigr)(x)$,
    which enters \eqref{eq_imaginary_eq} with a negative sign.
    We aim to show that this sum is bounded away from $0$ and $1$,
    which will imply the claim.
    Use Assumption \ref{ass:density}
    and take $N$ so large that $\QQ>2T\eps^{-1}$.
    Throw away summands for which $|a_i|>2T\eps^{-1}$,
    and then split the segment $(-2T\eps^{-1},2T\eps^{-1})$
    into $4T\eps^{-1}/\dd$ segments
    of the form
    $(-2T\eps^{-1}+j\dd,-2T\eps^{-1}+(j+1)\dd)$,
    each of which contains
    at least $\lod\dd$ points from the configuration $\mathfrak{A}$.
    On each of these segments, replace the
    summands $\frac{1}{\pi T}\frac{y}{y^2+(x-a_r/T)^{2}}$
    by $\lod \dd$ times
    the minimum of $\frac{1}{\pi T}\frac{y}{y^2+(x-a/T)^{2}}$
    over $a$ belonging to the corresponding segment.
    This allows to estimate
    $\bigl(\m_{T}[\mathfrak{a}]*\CC_y\bigr)(x)$ from below
    by a Riemann
    sum of the integral
    \begin{equation*}
        \lod\int_{-2\eps^{-1}}^{2\eps^{-1}}\frac{1}{\pi}\frac{y\,du}{y^2+(x-u)^2}=
        \frac{\lod}{\pi}
        \left[\tan ^{-1}\left(\frac{2\eps^{-1}-x}{y}\right)+
        \tan ^{-1}\left(\frac{2\eps^{-1}+x}{y}\right)\right]
    \end{equation*}
    within error $O(T^{-1}\eps^{-1})$ which goes to zero.
    For $y=\eps$, the expression
    in the square brackets is close to $\pi$, and for
    $\sqrt{x^2+y^2}=\eps^{-1}$ and $y\ge\eps$ it is $\ge\frac \pi2$.
    Therefore, $\bigl(\m_{T}[\mathfrak{a}]*\CC_y\bigr)(x)\ge\frac{\lod}2$.

    The other estimate is obtained in a similar manner but now
    we assume that all locations outside $(-2T\eps^{-1},2T\eps^{-1})$
    are occupied by particles from the configuration $\mathfrak{A}$. This
    allows to write
    \begin{multline*}
        \bigl(\m_{T}[\mathfrak{a}]*\CC_y\bigr)(x)\le
        \frac{\upd}{\pi}
        \int_{-2\eps^{-1}}^{2\eps^{-1}}\frac{y\,du}{y^2+(x-u)^2}
        +\int_{\R\setminus(-2\eps^{-1},2\eps^{-1})}
        \frac{1}{\pi}\frac{y\,du}{y^2+(x-u)^2}+O\Big(\frac{1}{T\eps}\Big)
        \\=
        \frac{\upd-1}{\pi}
        \left[\tan ^{-1}\left(\frac{2\eps^{-1}-x}{y}\right)+
        \tan ^{-1}\left(\frac{2\eps^{-1}+x}{y}\right)\right]+
        1+O\Big(\frac{1}{T\eps}\Big)\le
        \frac{1+\upd}{2}
    \end{multline*}
    for large enough $N$.
    This completes the proof.
\end{proof}

\begin{lemma}\label{lemma_imaginary_part_2}
 Under Assumptions \ref{ass:density} and \ref{ass:intermediate_behavior}, for each $0<\delta<1$ there exists
 $\eps_0>0$ (which may depend
 on constants in our assumptions but not on the choice of
 $\mathfrak{A}(N)$),
 such that for any $0<\eps<\eps_0$ there is
 $N_0\in\Z_{\ge1}$, and for all $N>N_0$
 there exists a curve $\gamma=\gamma(N)$
 in the upper half plane
 with the following properties:
 \begin{enumerate}[$\bullet$]
  \item For all $\z\in\gamma$ we have
  $\Im \Sfin'(\z)=0$, $\Im(\z)\ge\eps$, and
  $|\z|<\eps^{-1}$;
  \item The curve $\gamma$ starts in the set $\{x+\i \eps \colon
  {-1}-\eps^\delta<x<{-1}+\eps^\delta\}$, and ends in the set
  $\{x+\i \eps \colon{-\eps}^\delta<x<\eps^\delta\}$.
 \end{enumerate}
\end{lemma}
\begin{proof}
    Let
    $\mathcal D=\{x+\i y \in \mathbb C \colon y>\eps,\, \sqrt{x^2+y^2}<\eps^{-1}\}$,
    and denote
    \begin{equation*}
    \mathcal U =
    \mathcal D \cap
    \{x+\i y\in\C\colon \Im \Sfin'(x+\i y)<0\}.
    \end{equation*}
    By \Cref{lemma_imaginary_part_1}, the part of the boundary of $\mathcal U$ which
    lies inside the interior of $\mathcal D$
    is a union of several curves
    whose start and end points belong to
    \begin{equation*}
        \{x+\i\eps\in\C\colon -1-\eps^{\delta}<x<-1+\eps^{\delta}
        \textnormal{ or }
        -\eps^{\delta}<x<\eps^{\delta}\},
    \end{equation*}
    cf. \Cref{fig:compact_in_C_with_signs_of_Im_S_prime}.\footnote{One can show
    that these curves do not intersect, i.e., that
    $\Sfin''(\z)$ cannot vanish where
    $\Im\Sfin'(\z)=0$, but we do not need this fact.}

    By continuity and \Cref{lemma_imaginary_part_1},
    on any path from
    the segment
    $\{x+\i \eps\colon -1+\eps^\delta<x<-\eps^\delta\}$
    (where $\Im \Sfin'(\z)<0$) to the curved boundary of $\mathcal D$
    (where $\Im \Sfin'(\z)>0$)
    there exists a point where
    $\Im\Sfin'(\z)=0$. Thus,
    as $\gamma$ we can take any of the curves forming the boundary
    of $\mathcal U$ inside $\mathcal D$
    which starts to the left of $-1+\eps^{\delta}$,
    ends to the right of $-\eps^{\delta}$,
    and does not intersect the set
    $\{x+\i\eps\}$ except at its endpoints.
    This implies the claim.
\end{proof}

\begin{lemma}
\label{lemma_real_part}
    Under Assumptions \ref{ass:density} and \ref{ass:intermediate_behavior}, there exist
    $0<\delta<1$ and $\eps_0>0$,
    (which may depend
    on constants in our assumptions but not on the choice of
    $\mathfrak{A}(N)$), such that for each $0<\eps<\eps_0$ there
    exists $N_0\in\Z_{\ge1}$, and for all $N>N_0$ we have
    \begin{enumerate}[$\bullet$]
    \item $\Re \Sfin'(x+\i \eps)<-1$ for all $-1-\eps^\delta<x<-1+\eps^\delta$;
    \item $\Re \Sfin'(x+\i \eps)>1$ for all $-\eps^\delta<x<\eps^\delta$.
    \end{enumerate}
\end{lemma}
\begin{proof}
    We will prove only the second claim, as the first one is analogous.
    We will specify the exact value of $\eps_0$ at the end of the proof,
    and for now let us just fix arbitrary
    $\eps<\eps_0<1$ and $\delta\in(0,1)$.
    In addition, take a large positive real $\RRR$. If
    $\RRR$ and $N$ are large enough, then we can restrict the summation in the infinite
    principal value sum in
    \eqref{eq_real_eq} to $j\in \mathfrak L_T \cap [-\RRR T,\RRR T]$, so that
    \begin{equation*}
     \left| \pv\sum_{j\in \mathfrak{L}_T}\frac{1}{T}\frac{x-j/T}{y^2+(x-j/T)^{2}}- \sum_{j\in \mathfrak L_T \cap
    [-\RRR T,\RRR T]}\frac{1}{T}\frac{x-j/T}{y^2+(x-j/T)^{2}}\right|<1,\qquad
    y=\eps.
    \end{equation*}
    In turn, the sum over $j\in \mathfrak L_T \cap [-\RRR T,\RRR T]$ is the Riemann sum for
    the corresponding integral, so for large $N$ we have
    \begin{equation}
        \left| \pv\sum_{j\in \mathfrak{L}_T}\frac{1}{T}\frac{x-j/T}{y^2+(x-j/T)^{2}}-
        \int\limits_{u\in [-\RRR,-1]\cup [0,\RRR]} \frac{(x-u)du}{y^2+(x-u)^{2}}\right|<2,
        \qquad
        y=\eps.
    \end{equation}

    Let us now
    bound the sum over the configuration
    $\mathfrak{A}(N)$ in \eqref{eq_real_eq}.
    For that
    we split this sum into three parts:
    \begin{multline}\label{eq_split_into_3}
        \sum_{i\in\mathfrak{A}(N)}\frac{1}{T}\frac{x-i/T}{y^2+(x-i/T)^{2}}=
        \sum_{i\in\mathfrak{A}(N)\cap[-\RRR T,\RRR T]}\frac{1}{T}\frac{x-i/T}{y^2+(x-i/T)^{2}}
        \\+
        \sum_{i\in\mathfrak{A}(N)\setminus [-\RRR T,\RRR T]}\frac{1}{T}\left(\frac{x-i/T}{y^2+(x-i/T)^{2}}+\frac{T}{i}\right)
        - \sum_{i\in\mathfrak{A}(N)\setminus [-\RRR T,\RRR T]} \frac{1}{i}.
    \end{multline}
    The third sum in \eqref{eq_split_into_3} is bounded due to
    Assumption \ref{ass:intermediate_behavior}.
    For the second sum, observe that
    \begin{equation*}
        \frac{1}{T}\left(\frac{x-i/T}{y^2+(x-i/T)^{2}}+\frac{T}{i}\right)=
        \frac{1}{T}
        \frac{y^2+x^2  -(i/T)x}{
        (i/T)(y^2+(x-i/T)^{2})},
    \end{equation*}
    that is, the second sum over $i$ converges absolutely.
    Moreover, we can estimate as $N\to\infty$:
    \begin{equation*}
        \sum_{i\in\mathfrak{A}(N)\setminus [-\RRR T,\RRR T]}
        \left|\frac{1}{T}
        \left(\frac{x-i/T}{y^2+(x-i/T)^{2}}+\frac{T}{i}\right)
        \right|\le
        \sum_{i\in \Z\setminus [-\RRR T,\RRR T]}
        \left|\frac{1}{T}
        \frac{y^2+x^2  -(i/T)x}{
        (i/T)(y^2+(x-i/T)^{2})}
        \right|,
    \end{equation*}
    and the right-hand side is the Riemann sum
    for the integral
    \begin{equation*}
        \int\limits_{\R\setminus [-\RRR,\RRR]}
        \left|\frac{y^2+x^2-ux}{
        u(y^2+(x-u)^{2})}\right|du,
    \end{equation*}
    which is uniformly bounded for $(x,y)$ in our segment
    (where $x$ is around $0$). Thus,
    the second sum in \eqref{eq_split_into_3} is uniformly bounded
    by a constant independent of $\eps$.

    Finally, for the first sum in \eqref{eq_split_into_3} we use
    Assumption \ref{ass:density} and approximate sums by integrals
    similarly to the proof of \Cref{lemma_imaginary_part_1}.
    To get a lower bound, first throw away all nonnegative
    summands in this sum, and write for the remaining ones:
    \begin{equation}
    \sum_{\substack{i\in\mathfrak{A}(N)\cap[-\RRR T,\RRR T]
    \\
    i/T>x}}\frac{1}{T}\frac{x-i/T}{y^2+(x-i/T)^{2}}>
    \frac{1+\upd}{2}\int_{x}^{\RRR}
    \frac{(x-u)du}{y^2+(x-u)^2},
    \end{equation}
    where $N$ is sufficiently large.

    Combining all the estimates, we obtain the following bound. For each $\eps>0$ there
    exists $N_0$ such that for all $N>N_0$ we have
    \begin{multline}
    \Re \Sfin'(x+\i y) >
    \frac{1+\upd}{2}
    \int_{x}^{\RRR}
    \frac{(x-u)du}{y^2+(x-u)^2}-
    \int\limits_{u\in [-\RRR,-1]\cup [0,\RRR]} \frac{(x-u)du}{y^2+(x-u)^{2}}
    \\- \log(\be^{-1}-1)+\textnormal{``error''},
    \end{multline}
    where \textnormal{``error''}
    is uniform in $(x,y)$ in our segment
    and is independent of $\eps$.
    Observe that
    \begin{equation*}
        \int_{x}^{\RRR}
        \frac{(x-u)du}{y^2+(x-u)^2}=
        \int_{0}^{\RRR-x}\frac{-vdv}{v^2+y^2}>
        \int_{0}^{\RRR+1}\frac{-vdv}{v^2+y^2},
    \end{equation*}
    and
    \begin{equation*}
        -\int_{0}^{\RRR}
        \frac{(x-u)du}{y^2+(x-u)^{2}}=
        \int_{-x}^{\RRR-x}
        \frac{vdv}{v^2+y^2}
        >\int_{-x}^{\RRR-1}
        \frac{vdv}{v^2+y^2}.
    \end{equation*}
    At the same time,
    \begin{equation*}
        -\int_{-\RRR}^{-1}\frac{(x-u)du}{y^2+(x-u)^{2}}
        =-\frac{1}{2} \log \left(\frac{(\RRR+x)^2+y^2}{(1+x)^2+y^2}\right)
    \end{equation*}
    can be bounded by an absolute constant since
    both $x$ and $y$ are close to zero.
    Thus, we can write
    \begin{multline*}
        \frac{1+\upd}{2}
        \int_{x}^{\RRR}
        \frac{(x-u)du}{y^2+(x-u)^2}-
        \int\limits_{u\in [-\RRR,-1]\cup [0,\RRR]}
        \frac{(x-u)du}{y^2+(x-u)^{2}}
        \\>
        \frac{1+\upd}{2}
        \int_{0}^{\RRR+1}\frac{-vdv}{v^2+y^2}
        +
        \int_{-x}^{\RRR-1}
        \frac{vdv}{v^2+y^2}+C
        \\=
        -\frac{1+\upd}{4}
        \log \left(\frac{(\RRR+1)^2}{y^2}+1\right)
        +
        \frac{1}{2}
        \log \left(\frac{(\RRR-1)^2+y^2}{x^2+y^2}\right)+C.
    \end{multline*}
    Here and below in
    this proof $C$ stands for some real
    constant which is
    uniform in $x,y$ and does not depend on $\eps$
    but may depend on $\RRR$ (but we fixed large $\RRR$
    once and for all in the beginning of the proof).
    The value of $C$ can change from line to line.
    Since $y=\eps$ is small, we have
    \begin{equation*}
        -\frac{1+\upd}{4}
        \log \left(\frac{(\RRR+1)^2}{y^2}+1\right)
        =
        \frac{1+\upd}{2}\log y+C+O(\eps^2).
    \end{equation*}
    We also have
    \begin{equation*}
        \frac{1}{2}
        \log \left(\frac{(\RRR-1)^2+y^2}{x^2+y^2}\right)
        =\frac{1}{2}\log\left(y^2+(\RRR-1)^{2}\right)-
        \log\sqrt{x^2+y^2}>C- \delta\log y
    \end{equation*}
    because $x^2+y^2<\eps^{2 \delta}+\eps^{2}$,
    which behaves as $\eps^{2 \delta}(1+o(1))=
    y^{2 \delta}(1+o(1))$.
    When $\delta$ is close enough to $1$,
    \begin{equation*}
        \frac{1+\upd}{2}\log y- \delta\log y
    \end{equation*}
    tends to $+\infty$ as
    $y=\eps\to 0$, and we are done.
\end{proof}

\begin{proof}[Proof of \Cref{Proposition_roots_positioning}]
Fix $\eps>0$ and $N_0\in\Z_{\ge1}$ depending on $\eps$
such that
\Cref{lemma_imaginary_part_1,lemma_imaginary_part_2,lemma_real_part}
hold
(recall that $\eps<\eps_0$,
where $\eps_0$ may depend
 on constants in our assumptions but not on the choice of
 $\mathfrak{A}(N)$).
Consider the curve $\gamma$ from
\Cref{lemma_imaginary_part_2}. This is a continuous curve
on which
$\Im \Sfin'(\z)=0$. Furthermore, \Cref{lemma_real_part} guarantees
that $\Re \Sfin'(\z)$ has distinct signs
at the endpoints of $\gamma$.
Since $\Re \Sfin'(\z)$ is a continuous function on $\gamma$, we conclude that there
exists $\zfin\in \gamma$ for which $\Re \Sfin'(\zfin)=0$, and so
$\Sfin'(\zfin)=0$ (as $\Sfin'$ depends on $N$,
so does $\zfin$, and this statement is
valid for all $N>N_0$).
\Cref{lemma:S_T_critical_points_equation} then implies that there
are no other critical points in the upper half plane,
and therefore $\zfin$ is the
desired unique one.
As the compact set $\mathcal{Z}$ capturing $\zfin$
take $\{x+\i y\in\C
\colon y\ge \eps,\sqrt{x^2+y^2}\le \eps^{-1}\}$.
\end{proof}


\section{Asymptotics of the kernel:
proofs of
\Cref{Theorem_main,Theorem_main_convergence,Theorem_main_convergence_glob,Theorem_main_convergence_Lebesgue}} 
\label{sec:asymptotics_of_the_kernel}

In this section, based on the existence of nonreal critical points afforded by \Cref{Proposition_roots_positioning}, we establish the approximation of the correlation kernel \eqref{K_Bernoulli} of the noncolliding Bernoulli random walk by the extended sine kernel, and also the corresponding bulk limit theorems. That is, here we prove the remaining statements from \Cref{sub:main_result_bulk_limit_theorems}.

\subsection{Behavior of $\Im \Sfin(\z)$ and $\Re \Sfin(\z)$} 
\label{sub:behavior_of_im_sfin_z_old}



%

We aim to describe the steepest descent contours for $\Re \Sfin(\z)$. For that
we need to analyze the behavior of $\Re \Sfin(\z)$ and $\Im\Sfin(\z)$ in
various parts of the upper half plane $\HH$.  Recall that we defined
$\Sfin(\z)$ in \Cref{sub:function_sfin_z_} so that it is holomorphic in $\HH$
and extends to the real axis except the singularities at
$\mathfrak{l}_T\Delta\mathfrak{a}$ (all other logarithmic singularities
belonging to $\frac1T\Z$ are removable).

We start by considering the behavior of $\Im\Sfin(\z)$ close to the real line,
and define
\begin{equation}\label{eq:notation_b_left_right}
    \begin{split}
        b^r_0&=\min(\Z_{\ge0}\setminus\mathfrak{A}),\qquad
        b^{\ell}_{0}=\max(\mathfrak{A}\cap \{-T+1,\ldots,-1\}),
        \\
        b^{\ell}_{-1}&=\max(\Z_{\le -T}\setminus\mathfrak{A}),
        \qquad
        b^{r}_{-1}=\min(\mathfrak{A}\cap \{-T+1,\ldots,-1\}).
    \end{split}
\end{equation}
Clearly,
\begin{equation*}
    b^{\ell}_{-1}\le-T<-T+1\le b^{r}_{-1}\le b^{\ell}_{0}\le-1<0\le b^{r}_{0}.
\end{equation*}

\begin{lemma}\label{lemma:im_at_R_ladder}
    For $x\in\R$, $x\notin \mathfrak{l}_T\Delta\mathfrak{a}$, the function $\Im\Sfin(x)$ is piecewise constant, 
    making jumps at points of 
    $\mathfrak{l}_T\Delta\mathfrak{a}$. 
    It weakly increases for $x\in(-\infty,b_{-1}^{r}/T)\cup (b_0^{\ell}/T,+\infty)$, 
    and weakly decreases for $x\in(b^{\ell}_{-1}/T,b^{r}_{0}/T)$. See \Cref{fig:Im_S} for an example.
\end{lemma}
\begin{proof}
    This is straightforward from the definition of $\Sfin(\z)$ in \Cref{sub:function_sfin_z_} and the observation that $\Im(\log_{\HH}(x))=\pi\mathbf{1}_{x<0}$, where
    $x\in\R\setminus\{0\}$.
\end{proof}

\begin{figure}[htbp]
    \includegraphics[width=.4\textwidth]{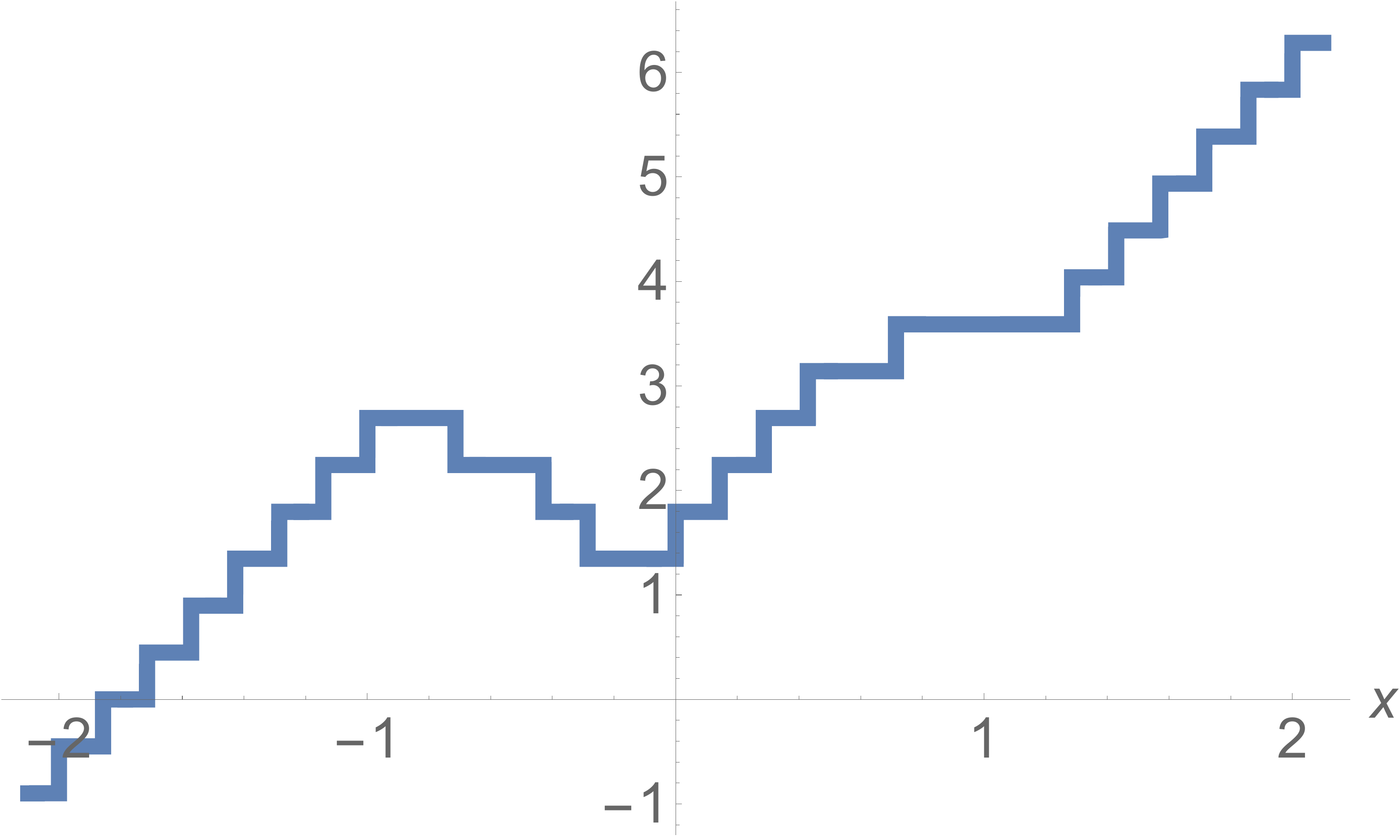}
    \caption{Staircase-type plot of $\Im\Sfin(x)$ for $x\in\R$, with parameters as in \Cref{fig:roots}.
    The singularities
    leading to the down steps are
    $\{-\frac67,\ldots,-\frac17\}\cap
    \{-\frac57,-\frac37,-\frac27,\frac47,\frac67,1,\frac87\}$.}
    \label{fig:Im_S}
\end{figure}

\begin{lemma}\label{lemma:real_neighborhood}
		Fix $0<\beta<1$ and the constants in Assumption
		\ref{ass:intermediate_behavior}. There exists $C>0$ depending only on these
		choices, and such that for each $T,N=1,2,\dots$, $x\in\mathbb R$, $y>0$ we
        have
        (note that $y^2-y\log y$ in the right-hand side is positive for $y>0$)
    \begin{equation}
    \label{eq_imaginary_bound}
     \bigl|\Im \Sfin(x+\i y)-\Im \Sfin(x+\i 0)\bigr|\le C \cdot \left( y  \log(|x|+1)-y\log y+y^2 +\frac{1}{T} \right).
    \end{equation}
\end{lemma}
\begin{remark}
	The value of $\Im \Sfin(x+\i 0)$ 
	when $x\in \mathfrak{l}_T\Delta\mathfrak{a}$
	(so that this piecewise linear function makes a jump)
	can be chosen arbitrarily (as long as \Cref{lemma:im_at_R_ladder} holds) ---
	this introduces an error of at most $1/T$ which is included
	in the right-hand side of \eqref{eq_imaginary_bound}.
\end{remark}
\begin{proof}[Proof of Lemma \ref{lemma:real_neighborhood}]
	Recall the definition 
	\eqref{S_T_action} of the function $\Sfin$.
	Our aim is to obtain a uniform bound on the increment $\Im \Sfin(x+\i y)-\Im
	\Sfin(x+\i 0)$. 
	
	We start from the second line in \eqref{S_T_action}. For
	$-\z\log(\be^{-1}-1)$ the increment is linear in $y$ and fits into the
	right-hand side of \eqref{eq_imaginary_bound}. For  $\frac {2\pi\i}T
    \left\lfloor
    \tfrac12\Re(T\z)+\tfrac12
    \right\rfloor$ the increment vanishes. For $-\frac 1T\log_{\HH}(\sin(\pi T \z))$,
		the imaginary part of $\log_{\HH}(\cdot)$ is bounded, since it is an
		argument of a complex number. Thus, the increment is bounded by $C/T$.

	 We proceed to the first line of \eqref{S_T_action}. Let us analyze the first
	 term, $ \frac 1T\sum_{r=1}^{N}\log_{\HH}\big(\z-\frac{a_r}T\big)$. Choose a
	 $\delta>0$, which will be later set to $\delta=4y$, and split the sum into
   \begin{multline}
   \label{eq_sum_split}
    \frac 1T\sum_{r=1}^{N}\log_{\HH}\Big(x+\i y-\frac{a_r}T\Big)= \frac 1T\sum_{1\le r\le N:\, |x-a_r/T|<\delta } \log_{\HH}\Big(x+\i y-\frac{a_r}T\Big)\\ +  \frac 1T\sum_{1\le r\le N:\, |x-a_r/T|\ge \delta } \log_{\HH}\Big(x+\i y-\frac{a_r}T\Big).
   \end{multline}
	 The first term in \eqref{eq_sum_split} has at most $2\delta T$ summands, the
	 increment of each one between the points $x+\i y$ and $x+\i 0$ is bounded by
	 a constant. Therefore, the increment of the first term is bounded by $C
	 \cdot \delta=4C y$.

   For the second term in \eqref{eq_sum_split}, we compute the increment directly as
   \begin{equation}
   \label{eq_x1}
    \Im \frac 1T \sum_{1\le r\le N:\, |x-a_r/T|\ge \delta } \log_{\HH}\left(1+\frac{\i y}{x-a_r/T} \right).
  \end{equation}
	By our choice of $\delta$, $\left|\frac{\i y}{x-a_r/T}\right|\le 1/4$.
	Therefore, we can Taylor expand each $\log_{\HH}(\cdot)$ and bound the
	absolute value of \eqref{eq_x1} as
  \begin{equation}
  \label{eq_x2}
   \frac y T \left|\sum_{1\le r\le N:\, |x-a_r/T|\ge \delta } \frac{1}{x-a_r/T} \right| +\frac{C \cdot y^2}{T}
   \sum_{1\le r\le N:\, |x-a_r/T|\ge \delta } \frac{1}{(x-a_r/T)^2}  .
  \end{equation}
  The second term in \eqref{eq_x2} is smaller than
	\begin{equation*}
   C y^2 \int_{|u|>\delta} \frac{1}{u^2}=C y \cdot \frac{2 y}{\delta}<Cy,
	\end{equation*}
	and therefore fits into the right-hand side of \eqref{eq_imaginary_bound}.
	For the first term in \eqref{eq_x2} we write
	$\frac{1}{u-v}=\frac{u}{(u-v)v}-\frac{1}{v}$ and bound from above as follows:
  \begin{multline}
    \label{eq_x3}
        \frac y T \left|
				\sum_{\substack{1\le r\le N\colon\\ |x-a_r/T|\ge \delta, \, |a_r|> T } }
				\left(\frac{x}{(x-a_r/T)(a_r/T)}-\frac{T}{a_r} \right) \right|+\frac y T \left|
				\sum_{\substack{1\le r\le N\colon\\ |x-a_r/T|\ge \delta, \, |a_r|\le  T } }
				\frac{1}{x-a_r/T} \right|
     \\ \le
		   y  \left|
			 \sum_{1\le r\le N\colon  |a_r|> T } 
			 \frac{1}{a_r}\right|+
		   y  \left|
			 \sum_{\substack{1\le r\le N\colon |x-a_r/T|< \delta, \, |a_r|> T } }
			 \frac{1}{a_r}\right| \\ +
				\frac y T \left|
				\sum_{\substack{1\le r\le N\colon\\ |x-a_r/T|\ge \delta, \, |a_r|> T}}
				\frac{x}{(x-a_r/T)(a_r/T)} \right| +y \left|
				\sum_{\substack{1\le r\le N\colon\\ |x-a_r/T|\ge \delta, \, |a_r|\le  T } }
				\frac{1}{Tx-a_r} \right|.
  \end{multline}
	In the right-hand side of \eqref{eq_x3} the first term is bounded by $C\cdot
	y$  due to Assumption \ref{ass:intermediate_behavior} and Remark
	\ref{rmk:intermediate}.  The second term has at most $2\delta T$ summands,
	each of which is at most $1/T$. Therefore, the second term in \eqref{eq_x3}
	is bounded from above by $2 y \delta=4y^2$. For the third term, we can
	replace the sum over $a_r$ by the sum over all integers $j$ satisfying the
	inequalities 
	$|x-j/T|\ge \delta$, $|j|> T$,
	and then upper bound the sum by the integral to get
	\begin{equation*}
    y \int_{|u-x|\ge \delta, \, |u|>1} \frac{x du }{(x-u)u}=y \lim_{M\to+\infty} \int_{|u-x|\ge \delta, \, 1<|u|<M} \left(\frac{1}{u}+\frac{1}{x-u} \right)du
	\end{equation*}
	At this point we need to consider several cases depending on the order of the
	points $x\pm \delta$ and $\pm 1$. In all the cases the integral evaluates
	into a combination of the expressions of the form $\log|x\pm \delta|$, $\log
	\delta$, and $\log|x\pm 1|$. We conclude that this term fits into the form of
	the right-hand side of \eqref{eq_imaginary_bound}.

	For the forth term in the right-hand side of \eqref{eq_x3} we again replace
	$a_r$ by all integers and then use $\sum\limits_{n= \theta_1 T}^{\theta_2 T}
	n^{-1}\approx \log(\theta_2/ \theta_1)$. As a result we get a bound of
	the form $C (\log y+1)$, which fits into the right-hand side of
	\eqref{eq_imaginary_bound}.

	\smallskip

	We have obtained a uniform bound for the increment of each term in
	\eqref{S_T_action} except for $\frac 1T\sum_{i=1}^{T-1}\log_{\HH}\bigl(x+\i
	y+\frac iT\bigr)$, and we proceed to bound this term. When $x$ is bounded away
	from $0$ and $-1$, the argument is the same as we just had. However, when $x$
	is close to $0$ or~$-1$, we need to proceed differently. Let us split the
	sum into two according to the sign of $x+ \frac{i}{T}$. Each of them is
	analyzed in the same way, so we will only deal with one. 
	This reduces the
	problem to bounding
 \begin{equation}
 \label{eq_x4}
  \frac{1}{T} \sum_{i=1}^{T'-1} \Im \left[\log_{\HH}
	\Big(x+\frac iT+\i y\Big)-\log_{\HH}\Big(x+\frac iT+\i 0\Big)\right], \qquad x\ge -\frac{1}{T}, \quad T'\le T.
 \end{equation}
 Each term in \eqref{eq_x4} has the form
	\begin{equation*}
   \Im \left[\log_{\HH}\Big(1+\frac{\i y}{x+\frac iT}\Big)\right]=\arctan\left(\frac{y}{x+\frac iT} \right).
 \end{equation*}
 For $x=-1/T$, $i=1$, the corresponding term in \eqref{eq_x4} vanishes. 
 For all other cases we note that $\arctan$ 
 is monotone and use $\arctan(u)\le \min(\pi/2, u)$, $u\ge 0$.  
 We thus bound~\eqref{eq_x4} by
 \begin{multline*}
 \frac{\pi}{2}\cdot \frac{Ty+1}{T}+ \frac{1}{T} 
 \sum_{i=\lceil Ty \rceil+1 }^{T'-1}  \left(\frac{Ty}{i-1}\right)\le  
 \frac{\pi}{2}\cdot \frac{Ty+1}{T} +y \int_{T\min(y,1)}^{T} \frac{1}{u} du 
 \\\le  \frac{\pi}{2}\cdot \frac{Ty+1}{T} -y \log\bigl(\min(y,1)\bigr).
 \end{multline*}
 Since the last expression fits into the right-hand side of \eqref{eq_imaginary_bound}, we are done.
\end{proof}

We now turn to the real part $\Re \Sfin(\z)$.

\begin{lemma}\label{lemma:estimate_real} 
    Under Assumptions \ref{ass:density} and \ref{ass:intermediate_behavior} 
    and with constants depending only on these assumptions, 
    the following estimates hold:
\begin{enumerate}
\item
    For any $k\in\Z$, $N=1,2,\dots$, and $y>0$ we have
    \begin{equation*}
        \left|\frac{\partial}{\partial y}\Re \Sfin\Big(
        \frac{k+1/2}{T}+\i y
        \Big)\right|\le \pi.
    \end{equation*}
\item There exists $Y>0$, $N_0>0$, such that for any $|x|<\frac{1}{2} \cdot \frac{\QQ(N)}{T(N)}$,  $y>Y$, and $N>N_0$ we have
	\begin{equation*}
     \frac{\partial}{\partial y}\Re \Sfin\bigl(
        x+\i y\bigr)=-\frac{\partial}{\partial x}\Im \Sfin\bigl(
        x+\i y\bigr)<-\frac{1}{y}.
 \end{equation*}

\item For each $X>0$ there exist $C,N_0>0$ such that for all $0<y<1/2$, $|x|<X$ and all $N>N_0$ we have
  \begin{equation*}
     \left|\frac{\partial}{\partial x}\Re \Sfin\bigl(
        x+\i y\bigr)\right|=  \left|\frac{\partial}{\partial y}\Im \Sfin\bigl(
        x+\i y\bigr)\right|< C\left( \log(y^{-1})+\frac{1}{y^2 T^2}\right).
  \end{equation*}

%

\end{enumerate}
\end{lemma}
\begin{proof}
    Using \eqref{log_branch} we see that, apart from the linear term $-(\Re \z)\log(\be^{-1}-1)$, the function $\Re \Sfin(\z)$ is an infinite linear combination (with coefficients $\pm1$) of shifts of $\frac{1}{T}\log|\z|$. We have for any $x,y\in \mathbb R$
	\begin{equation*}
     \frac{\partial}{\partial y} \log|x+\i y|=\frac{y}{x^2+y^2}.
 \end{equation*}
    In particular, for
    $j\in\Z$:
    \begin{equation*}
        \left|\frac{\partial}{\partial y}
        \frac{1}{T}\log\Big|\frac{k+1/2}{T}+\i y+\frac{j}{T}\Big| \right|
    =
        \frac{4 T |y|}{(2 k+2j+1)^2+4 T^2 y^2}.
    \end{equation*}
    Thus, the absolute value of the derivative of $\Re\Sfin$ in the first claim can be bounded in the absolute value by
    \begin{equation*}
        \sum_{j\in\Z}\frac{4 T |y|}{(2j+1)^2+4 T^2 y^2}=
        \pi  \tanh (\pi  T |y|)\le \pi,
    \end{equation*}
    this is summed with the help of a partial fraction expansion and \eqref{Euler_cot}, and $\tanh$ is bounded by one. This establishes the first claim.

   \smallskip

    For the second claim we need to be more careful with signs. 
    Recalling that $\mathfrak A(N)$ is the initial condition and using notation \eqref{eq_L_def}, we write
    \begin{equation}
    \label{eq_x5}
     \frac{\partial}{\partial y} \Re \Sfin(x+\i y)= -\frac{1}{T} \sum_{a\in \mathfrak L_T \setminus
     \mathfrak A(N)} \frac{y}{(x-a/T)^2+y^2}+\frac{1}{T}\sum_{a\in \mathfrak A(N)\bigcap \{1-T,\dots,-1\}}  \frac{y}{(x-a/T)^2+y^2}.
    \end{equation}
    Our aim is to show that in the last sum the first term dominates. 
    Using Assumption \ref{ass:density} and replacing sums by integrals 
    (with multiplicative error at most 2), we upper bound \eqref{eq_x5} by
    \begin{multline}
    \label{eq_x6}
      -\frac{1-\upd}{2} \left(\int_{-\QQ(N)/T}^{-1}+\int_0^{\QQ(N)/T}\right)  \frac{y\, dv }{(v-x)^2+y^2}+2 \int_{-1}^0 \frac{y\, dv}{(v-x)^2+y^2}
      \\ = -\frac{1-\upd}{2} \left(\int_{-x/y-\QQ(N)/(Ty)}^{(x-1)/y}+\int_{x/y}^{-x/y+\QQ(N)/(Ty)}\right)  \frac{du }{u^2+1}+2 \int_{-(x+1)/y}^{-x/y} \frac{ du}{u^2+1}
      \\ \le
      -\frac{1-\upd}{2} \int_{-\QQ(N)/(2Ty)}^{\QQ(N)/(2Ty)}  \frac{du }{u^2+1}+3 \int_{-(x+1)/y}^{-x/y} \frac{ du}{u^2+1}
      \\ \le
       -(1-\upd) \arctan\left( \frac{\QQ(N)}{2Ty}\right)+ \frac{3}{y}.
    \end{multline}
		Considering separately the cases of small and large $y$, using $\QQ(N)/T(N)\to\infty$ 
		as $N\to\infty$, 
		we see that the last expression is smaller than $-1/y$ for large $N$, and the second claim is proven. 
		(Notice that 
		$\frac{\partial}{\partial y}\Re \Sfin\bigl(x+\i y\bigr)
		=
		-\frac{\partial}{\partial x}\Im \Sfin\bigl(x+\i y\bigr)$ 
		is the Cauchy--Riemann equation.)

    \smallskip

 Let us turn to the third claim.    Assume that $x$ is fixed. We have
    \begin{equation*}
        \frac{\partial}{\partial x}
        \frac{1}{T}\log\Big|x+
        \i y+\frac jT\Big|
        =\frac{T x+j}{y^2 T^2+(T x+j)^2},
    \end{equation*}
    so

    \begin{multline}
        \frac{\partial}{\partial x}\Re \Sfin\big(
        x+
        \i y
        \big)=
        \sum_{r=1}^{N}\frac{Tx-a_r}{y^{2}T^2+(Tx-a_r)^{2}}
        +\sum_{i=1}^{T-1}
        \frac{T x+i}{y^2 T^2+(T x+i)^2}
        \\-\log(\be^{-1}-1)-\frac{\partial}{\partial x}
        \frac 1T\log|\sin(\pi T(x+\i y))|.
        \label{estimate_close_to_R_proof_new}
    \end{multline}

    Fix sufficiently large $\RRR>0$, which might depend on $x$, but not on $y$. For the first sum in \eqref{estimate_close_to_R_proof_new} with $|a_r|<\RRR T$, and also for the second sum in \eqref{estimate_close_to_R_proof_new} we upper bound the absolute values of the sums by
    twice of

    \begin{multline*}
        \sum_{j=0}^{2 \RRR T}\frac{j+1}{j^2+y^{2}T^2}\le
        \sum_{j=0}^{2 \RRR T}\frac{j}{j^2+y^{2}T^2}+\frac{1}{y^2 T^2}+\sum_{j=1}^\infty \frac{1}{j^2}\le
        \int_{0}^{2 \RRR}\frac{v}{v^2+y^{2}}dv +\frac{1}{y^2 T^2} +C
       \\ =
         \frac{1}{2} \log \left(1+\frac{4 \RRR^2}{y^2}\right)+\frac{1}{y^2 T^2}+C.
    \end{multline*}

    Therefore, the contribution of these terms admits the desired bound.
    Next, for $|a_r|>\RRR T$ in the first sum in \eqref{estimate_close_to_R_proof_new} we have
    \begin{equation*}
        \frac{1}{Tx-a_r}-\frac{T x-a_r}{y^{2}T^2+(Tx-a_r)^{2}}
        =\frac{T^2 y^2}{(Tx-a_r) \left(T^2 y^2+(Tx-a_r)^2\right)},
    \end{equation*}
    and summing this over $|a_r|>\RRR T$ has order $y^{2}/\RRR^{2}$, which is bounded. Thus, the contribution from $|a_r|>\RRR T$ in the first sum in \eqref{estimate_close_to_R_proof_new} is the same as if the summands were just $1/(Tx-a_r)$. Observe that
    \begin{equation*}
        \frac1{Tx-a_r}+\frac1{a_r}=
        \frac{T x}{a_r (Tx-a_r)}
    \end{equation*}
    and the sum of these quantities over $|a_r|>\RRR T$ with large $\RRR$ is bounded (recall that $|x|$ is bounded and  $\RRR$ is chosen to be much larger than it). Thus, the sum of the terms with $|a_r|>\RRR T$ in the first sum in \eqref{estimate_close_to_R_proof_new} has the same order as the sum of $1/a_r$ over $|a_r|>\RRR T$, which is bounded by Assumption \ref{ass:intermediate_behavior}.

    Finally, for the last summand in \eqref{estimate_close_to_R_proof_new} we have
    \begin{equation*}
        \left| \frac{\partial}{\partial x}
        \frac 1T\log|\sin(\pi T(x+\i y))\right|= \pi \left|\frac{\cos(\pi T(x +\i y))}{\sin(\pi T(x+\i y))} \right|\le C\left(1+\frac{1}{T^2y^2}\right),
        \end{equation*}
        where we used the bound 
        $|\sin(\alpha +\i \beta)|\ge C \cdot \min(|\beta|^2,1)$ 
        for an absolute constant $C>0$.
\end{proof}

\subsection{Steepest descent/ascent contours in a large rectangle}
\label{sub:new_contours}

Our next aim is to present a new set of contours for the double contour integral expression of Theorem \ref{Theorem_Bernoulli_IC}. 
In this section we explain their geometry in a (sufficiently large)
compact subset of the upper half plane.
In discussion of integration 
contours in the rest of this section
it
suffices to 
argue in the upper half plane: the 
contour configuration
in
the lower half plane (with a suitable choice of
branches of logarithms, cf. \Cref{sub:function_sfin_z_}) is obtained by reflection
with respect to the real line.

Recall the critical point $\zfin=\zfin(N)$ afforded by \Cref{Proposition_roots_positioning}.
We need the following statement which will be proven 
in \Cref{sub:convergent_initial_data} below:

\begin{lemma}
\label{lemma:second_derivative}
    Under Assumptions \ref{ass:density} and \ref{ass:intermediate_behavior} there exists $C>0$,
    such that $C^{-1}<|\Sfin''(z)|<C$ for all $z$ 
    in the upper half--plane satisfying $|z-\zfin|<C^{-1}$, and all $N=1,2,\dots$.
\end{lemma}

Fix three constants: small $\eps<0$ and large $\mathsf{R}_x,\mathsf{R}_y>0$,
which do not depend on $N$ or $\mathfrak A(N)$, but might depend on the constants in Assumptions
\ref{ass:density}, \ref{ass:intermediate_behavior}. There are  four contours $\{\z\colon \Im
\Sfin(\z)=\Im\Sfin(\zfin)\}$ emanating from the critical point. Let us trace these contours until
they leave a rectangle $\mathcal{R}:=\{x+\i y \in \mathbb C\colon |x|<\mathsf{R}_x, \, \eps<y<\mathsf{R}_y\}$
(for some $\mathsf{R}_y>\eps>0$, $\mathsf{R}_x>0$).
Let $z_1,z_2,z_3,z_4$
be the \emph{escape} points where the contours leave the rectangle.

\begin{proposition}
\label{Proposition_large_rectangle_contours}
There exist $\mathsf{R}_y>\eps>0$, $\mathsf{R}_x>0$, such that for all large enough $N$:
	\begin{enumerate}
		\item Three escape points, $z_1, z_2, z_3$ (ordered as $\Re z_1< \Re z_2 < \Re z_3$) are on the lower
			side $\Im z=\eps$, and $z_4$ is on the upper side $\Im z=\mathsf{R}_y$.
		\item The real part $\Re \Sfin(\z)$ grows along the contours escaping through $z_1$ and $z_3$, and
			decays along the contours escaping through $z_2$, $z_4$.
		\item The escape points on the lower sides of the 
			rectangle satisfy 
			\begin{equation}\label{eq:large_rectangle_estimates_on_z_123}
				\Re z_1 \in (-\infty, -1+\eps \log^2\eps),
				\quad
				\Re z_2\in(-1-\eps \log^2\eps, \eps\log^2\eps),
				\quad
				\Re z_3\in (-\eps \log^2\eps,+\infty).
			\end{equation}
		\item $\mathsf{R}_y>Y$, where $Y$ is from the second claim of \Cref{lemma:estimate_real}.
		\item $\eps< \frac{1}{10}\min_{i=1,2,3}|\Re \Sfin (z_i)-\Re \Sfin(\zfin)|$,
	\end{enumerate}
\end{proposition}

\begin{proof}
	First, $\Re \Sfin(\z)$ can not have local extrema on the contours 
	$\{\z\colon \Im\Sfin(\z)=\Im\Sfin(\zfin)\}$ 
	inside the rectangle (expect at point $\zfin$), as any
	such extremum would be a new critical point for $\Sfin(\z)$ contradicting
	\Cref{Proposition_roots_positioning}. 
	Therefore, $\Re \Sfin(\z)$ is monotone along
	these contours. This also means that these contours cannot intersect
	anywhere in the rectangle except at $\zfin$.
	Out of these four contours, along two the real part
	$\Re \Sfin(\z)$ grows, and along other two it
	decays.
	Since the growth/decay types interlace
	as the contours leave $\zfin$, we conclude that the growth/decay types also
	interlace along the boundary of the rectangle.

	We now fix arbitrary $\mathsf{R}_y>Y$, such that $\Im(\zfin)<\mathsf{R}_y/2$. The second claim of 
	\Cref{lemma:estimate_real} implies that (for large $N$) $\Im\Sfin(\z)$ is monotone along the top
	side $\Im\z=\mathsf{R}_y$
	of the rectangle.
	Therefore, at most one of the points $z_i$, $1\le i\le 4$,
	can be there.

	Next, \Cref{lemma:im_at_R_ladder} combined with Assumption \ref{ass:density} implies that
	there exists $R>0$ such that 
	$\Im\Sfin(x+\i 0)> R^{-1} x$ 
	for $x>R$ and 
	$\Im\Sfin(x+\i 0)<-R^{-1} x$ 
	for $x<-R$.  Thus, \Cref{lemma:real_neighborhood} implies that we can choose large enough
  $\mathsf{R}_x>R$, such that 
	\begin{equation*}
		\Im\Sfin(\mathsf{R}_x+\i y)> 2|\Im \Sfin(\zfin)|,
		\qquad 
    \Im\Sfin(-\mathsf{R}_x+\i y)< 2|\Im \Sfin(\zfin)|
	\end{equation*}
	for all $0\le y \le \mathsf{R}_y$. 
  We fix such $\mathsf{R}_x$ and notice that
	this choice implies that $z_i$, $1\le i \le 4$ do not belong to the vertical sides of the
	rectangle.

	Thus, either three or four of the points $z_i$ belong to the bottom horizontal side of the
	rectangle. It remains to specify $\eps>0$, so that there are exactly three and their positions
	satisfy \eqref{eq:large_rectangle_estimates_on_z_123}.

	By \Cref{lemma:second_derivative} we have a uniform control over the
	growth/decay of $\Re \Sfin(\z)$ in a small (but fixed size) neighborhood of
	$\zfin$. Thus, when the contours reach the boundary of the rectangles, the values of
	$\Re \Sfin(\z)$ are separated by a constant. Combining this fact with interlacing of the
	growth/decay contours and the bound of the third statement of 
	\Cref{lemma:estimate_real} we conclude that there exists $\delta>0$ such that for
	each $0<\eps<1/2$ and $T>\eps^{-1}$ we have $|z_i-z_j|> \frac{\delta}{\log(1/\eps)}$ for
	all $i\ne j$.

	Next, let $\mathcal{U}\subset \mathbb R$ denote the set of points $x$ such that
	$\Im\Sfin(x+0\i )=\Im \Sfin(\zfin)$. According to \Cref{lemma:im_at_R_ladder},
	$\mathcal{U}$ splits into three disjoint sets $\mathcal{U}_1$, $\mathcal{U}_2$,
	$\mathcal{U}_3$ (some of which might be empty): $\mathcal{U}_1\subset (-\infty, -1]$,
	$\mathcal{U}_2\subset [-1,0]$, $\mathcal{U}_3\subset [0,+\infty)$. Using Assumption
	\ref{ass:density}, we see that the diameter of each set $\mathcal{U}_i$ is at most
	$\frac{\dd(N)}{T(N)}$ which tends to $0$ as $N\to\infty$.

	We further would like to show that $z_i$ is close to $\mathcal{U}_i$, $i=1,2,3$. For that note
	that by Assumption \ref{ass:density}, the function $x\mapsto \Im\Sfin(x+\i 0)$ has
	growth bounded away from $0$ in the sense that for some $c>0$ we have
	\begin{equation*}
		|\Im(\Sfin(x+\i 0)-\Im(\Sfin(x'+\i 0)|\ge c\cdot |x-x'|, \text{ if }  |x-x'|\ge
		\frac{\dd(N)}{T(N)}\text{ and } \begin{cases} x,x'\in
		[-\mathsf{R}_x, -1], \text{ or}\\ x,x'\in [-1,0],\text{ or}\\ x,x'\in [0,\mathsf{R}_x].\end{cases}
	\end{equation*}

	Thus, using \Cref{lemma:real_neighborhood} we conclude that for fixed $\eps>0$
	and large $N$, $T$ (compared to that $\eps$) the real part of each point $z_i$ (out of
	those lying in the bottom horizontal side of the rectangle) should be in $C \eps
    \log(\eps^{-1})$--neighborhood of $\mathcal{U}_1 \cup \mathcal{U}_2 \cup \mathcal{U}_3$, 
    where $C$ does not depend on $T$ or $N$. On the other hand, the diameter of
	each $\mathcal{U}_i$ is small and $|z_i-z_j|>\frac{\delta}{\log(1/\eps)}$. Since
	$\frac{\delta}{\log(1/\eps)}\gg C \eps \log(\eps^{-1})$, we conclude that the only
	possibility is that there are precisely three points $z_i$ on the bottom horizontal
	side of the rectangle (which means that $z_4$ is on the upper horizontal side) and
	each $z_i$ is in $C \eps \log(\eps^{-1})$ neighborhood of $\mathcal{U}_i$,
	respectively, for $i=1,2,3$. For small enough $\eps$ we would have $C\eps
	\log(\eps^{-1})< \eps \log^2(\eps)$, which finishes the proof.
\end{proof}


\subsection{Completing the proof of \Cref{Theorem_main}} 
\label{sub:moving_the_contours}

Here we describe how the new contours 
in a large rectangle
constructed in \Cref{sub:new_contours}
should be continued outside the rectangle.
(Recall that by symmetry, it suffices to argue in the upper
half plane only.)
We then rewrite
the correlation kernel 
$K^{\textnormal{Bernoulli}}_{\vec a; \be}$
in terms of these new contours, and 
complete the proof of \Cref{Theorem_main}
on approximation of $K^{\textnormal{Bernoulli}}_{\vec a; \be}$
by the 
extended discrete sine kernel.

Fix $\mathsf{R}_y>\eps>0$, $\mathsf{R}_x>0$
as in \Cref{Proposition_large_rectangle_contours}. 
Define the new $z$ contour
$\CC_\z^{\deps}=\CC_\z^{\deps}(N)$ as follows.\footnote{This contour, 
as well as $\CC_\w^{\deps}$ defined below, also depends 
on the constants $\mathsf{R}_{x}$, $\mathsf{R}_y$ and 
other data in Assumptions \ref{ass:density} and \ref{ass:intermediate_behavior}.
We suppress all this dependence in the notation.}
Inside the rectangle $\mathcal{R}$
it
coincides with the union of the steepest descent (for $\Re\Sfin$)
contours
$\{\z\colon\Im\Sfin(\z)=\Im\Sfin(\zfin)\}$
escaping through the points $z_2$ and $z_4$.
After the point $z_4$ we continue 
the contour vertically so that it escapes to infinity.
Inside the $\deps$-neighborhood of the real line we 
have to modify the steepest descent contour 
so that it crosses
$\R$ strictly between $-1$ and $0$.
To achieve that, we add to this contour a horizontal segment 
of the line 
$\Im\z=\i \deps$, and then a vertical segment 
connecting it to the real line such that 
$\CC_\z^{\deps}$
crosses $\R$
at a point of the form $(k+1/2)/T$ for some $k\in\{-T,\ldots,-1\}$
which is close to $z_2$ within $C\cdot(\eps \log^2\eps+T^{-1})$.
The contour $\CC_\z^{\deps}$ is oriented upwards.

\begin{figure}[htbp]
    \includegraphics[width=.57\textwidth]{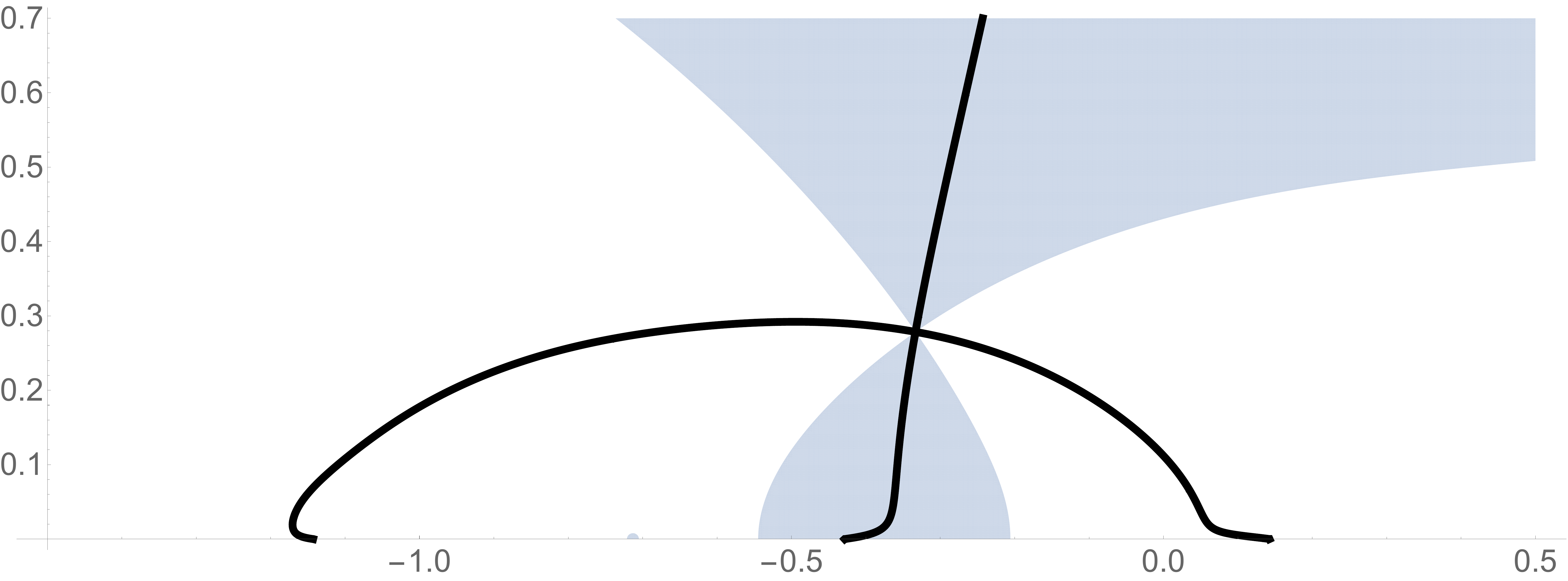}
		\qquad 
		\begin{adjustbox}{max width=.37\textwidth}
			\begin{tikzpicture}[scale=1,ultra thick]
        \draw[thick,->] (-5,0)--(5,0);
        \draw[thick,->] (2,-3.5)--(2,3.5);
        \draw[fill] (-2,0) circle (1pt) node[below] {$-1$};
        \draw[fill] (2,0) circle (1pt) node[below right] {$0$};
        \draw[fill] (2,1) circle (1pt) node[above right] {$\deps$};
        \draw[fill] (2,-1) circle (1pt) node[below left] {$-\deps$};
        \draw (1,3.3) to [out=-60,in=110] (2.8,1)--++(-1,0)--++(0,-2)--++(1,0) to
        [out=-110,in=60] (1,-3.3) node [right] {$\CC_\z^{\deps}$};
        \draw[densely dashed] (3.5,1)--(4,1) --++ (0,-2) --++ (-.5,0)
        to [out=-90,in=0] (1.5,-2.5) to [out=180,in=-50] (-2.2,-1)
        --++(-1.5,0)--++(0,2)--++(1.5,0)
        to [out=50,in=180] (1.5,2.5) to [out=0,in=90] (3.5,1)
        node [above right] {$\CC_\w^{\deps}$};
    \end{tikzpicture}
	\end{adjustbox}
    \caption{Left: Steepest descent/ascent contours
    $\{\z\colon\Im \Sfin(\z)=
    \Im \Sfin(\zfin)\}$ for the function $\Re \Sfin(\z)$
    in the upper half plane.
    Regions where $\Re \Sfin(\z)<\Re \Sfin(\zfin)$ are shaded. The parameters $\mathfrak{A}$
    and $T$ are as in \Cref{fig:roots}, and $\be=0.4$.
		Right: Modification of the contours
		in an $\deps$-neighborhood of the real line. }
    \label{fig:steepest_descent}
\end{figure}

Next, by $\CC_\w^{\deps}=\CC_\w^{\deps}(N)$ denote the
closed positively oriented contour
which 
inside the rectangle $\mathcal{R}$
coincides
with the union of the steepest ascent (for $\Re\Sfin$)
contours
$\{\z\colon\Im\Sfin(\z)=\Im\Sfin(\zfin)\}$
escaping through the points $z_1$ and $z_3$.
Outside $\mathcal{R}$ we 
modify the steepest ascent contour so that it encircles
$\{-1+T^{-1}(\D x-\D t),\ldots,T^{-1}(\D x-1),T^{-1}\D x\}$,
and crosses $\R$ at two points of the form $(k+1/2)/T$, $k\in \mathbb{Z}$, 
close to $z_1$ and $z_3$ within $C\cdot(\eps\log^2\eps+T^{-1})$. This is achieved by adding 
horizontal and vertical segments
similarly to $\CC_\z^{\deps}$. 
See \Cref{fig:steepest_descent} for an illustration
of the new contours 
$\CC_\z^{\deps}$ and $\CC_\w^{\deps}$.

\begin{proposition}\label{prop:K_Beta_paths_form}
    With the above definitions and conventions,
    for any $N>N_0$
    the kernel \eqref{K_Bernoulli_general} can be written as
    \begin{multline}\label{K_Beta_paths_form}
        K^{\textnormal{Bernoulli}}_{\vec a; \be}(T+\D t,\D x;T,0)
        =
        K_{\zfin/(\zfin+1)}(T+\D t,\D x;T,0)
        +
        \frac1{(2\pi\i)^{2}}
        \int_{\CC_\z^{\deps}}
        d\z
        \oint_{\CC_\w^{\deps}}d\w\,
        \frac{1}{\w-\z}\\\times
        \frac{(T+\D t)!\cdot T}{(T-1)!}
        \frac{(T\z+1)_{T-1}}{(T\w-\D x)_{T+\D t+1}}
        \frac{\sin(\pi T\w)}{\sin(\pi T\z)}
        \left(\frac{1-\be}{\be}\right)^{T(\w-\z)}
        \prod_{r=1}^{N}\frac{T\z-a_r}{T\w-a_r}+
        O(T^{-1}).
    \end{multline}
\end{proposition}
\begin{proof}
		All poles of the integrand in \eqref{K_Bernoulli_general} are on the real
		line. This integrand has no poles at
		$\z\in\{b_{-1}^{\ell}/T,\ldots,b_{0}^{r}/T\}$ (recall the notation
		\eqref{eq:notation_b_left_right}), and thus we can drag the
		point of intersection of the $\z$ contour with the real line to the desired
		location dictated by the contour $\CC_\z^{\deps}$. Hence we can deform the
		whole $\z$ contour to coincide with $\CC_\z^{\deps}$ without crossing any
		poles. Next, let us unite the two circles comprising the $\w$ contour in
		\eqref{K_Bernoulli_general} into the contour $\CC_\w^{\deps}$ intersecting
		with $\CC_\z^{\deps}$ at the critical points $\zfin$ and $\zbfin$. This
		leads to an additional integral of the residue at $\w=\z$ over the arc of
		$\CC_\z^{\deps}$ from $\zbfin$ to $\zfin$ crossing the real line between
		$-1$ and $0$, see \Cref{fig:dragging}. This deformation of the $\w$ contour
		does not cross any other $\w$ poles of the integrand.\footnote{The desire
        that these deformations do not cross any real poles is the reason 
        why the contours $\CC_\z^{\deps}$ and $\CC_\w^{\deps}$ should differ from the
        steepest descent/ascent ones close to the real line.}

\begin{figure}[htbp]
		\scalebox{.8}{\begin{tikzpicture}
				[scale=.9, ultra thick]
				\draw[densely dotted, postaction={decorate,decoration={markings,
						mark=between positions .1 and 1 step 0.25 with {\draw [->] (0,0) -- (.1,0);}}}]
				(-1.5,1.6) to [out=180,in=90]
				(-3,0) to [out=-90,in=180]
				(-1.5,-1.6) to [out=0,in=-90]
				(-0.03,0) to [out=90,in=0]
				(-1.5,1.6);
				\draw[densely dotted, postaction={decorate,decoration={markings,
						mark=between positions 0.2 and 1 step 0.25 with {\draw [->] (0,0) -- (.1,0);}}}]
				(-4,1.2) to [out=180,in=90]
				(-5.86,0) to [out=-90,in=180]
				(-4,-1.2) to [out=0,in=-90]
				(-3.4,0)
				to [out=90,in=0]
				(-4,1.2);
				\node at (-5.2,1.3) {$\w$};
				\node at (-1,1.8) {$\w$};
				\draw[postaction={decorate,decoration={markings,
						mark=between positions 0.1 and 0.9 step 0.4 with {\draw [->] (0,0) -- (.1,0);}}}]
				(-2.25,-2.5)
				to [out=135,in=-80] (-3.1,-1)
				to [out=100,in=-100] (-3.1,1)
				to [out=80,in=-135] (-2.25,2.5) node[right] {$\z$};
				\begin{scope}[shift={(-3.1,1)},scale=.6]
						\draw[line width=2.5] (-.2,-.2)--++(.4,.4);
						\draw[line width=2.5] (-.2,.2)--++(.4,-.4);
						\node at (-.48,.6) {$\zfin$};
				\end{scope}
				\begin{scope}[shift={(-3.1,-1)},scale=.6]
						\draw[line width=2.5] (-.2,-.2)--++(.4,.4);
						\draw[line width=2.5] (-.2,.2)--++(.4,-.4);
						\node at (-.4,-.4) {$\zbfin$};
				\end{scope}
				\node[anchor=west] at (-6.5,-2.9) {\large{}Double contour integral in \eqref{K_Bernoulli_general}};
		\end{tikzpicture}}
		\scalebox{.8}{\begin{tikzpicture}
				[scale=.9, ultra thick]
				\draw[densely dotted, postaction={decorate,decoration={markings,
						mark=between positions .1 and 1 step 0.25 with {\draw [->] (0,0) -- (.1,0);}}}]
				(-0.03,0)
				to [out=90, in=0] (-3.1,1)
				to [out=180,in=90] (-5.86,0)
				to [out=-90, in=180] (-3.1,-1)
				to [out=0,in=-90] (-0.03,0);
				\node at (-5.2,1.2) {$\w$};
				\draw[postaction={decorate,decoration={markings,
						mark=between positions 0.1 and 0.9 step 0.4 with {\draw [->] (0,0) -- (.1,0);}}}]
				(-2.25,-2.5)
				to [out=135,in=-80] (-3.1,-1)
				to [out=100,in=-100] (-3.1,1)
				to [out=80,in=-135] (-2.25,2.5) node[right] {$\z$};
				\begin{scope}[shift={(-3.1,1)},scale=.6]
						\draw[line width=2.5] (-.2,-.2)--++(.4,.4);
						\draw[line width=2.5] (-.2,.2)--++(.4,-.4);
						\node at (-.48,.6) {$\zfin$};
				\end{scope}
				\begin{scope}[shift={(-3.1,-1)},scale=.6]
						\draw[line width=2.5] (-.2,-.2)--++(.4,.4);
						\draw[line width=2.5] (-.2,.2)--++(.4,-.4);
						\node at (-.4,-.5) {$\zbfin$};
				\end{scope}
				\node at (-7,0) {\huge$=$};
				\node at (-7.9,0) {};
				\node at (1.2,0) {\huge$-$};
				\node at (1.8,0) {};
				\node[anchor=west] at (-5,-2.9) {\large{}\phantom{($*$)}};
		\end{tikzpicture}}
		\scalebox{.8}{\begin{tikzpicture}
				[scale=.9, ultra thick]
				\node at (-3.48,.2) {$\z$};
				\node at (-2.6,-.2) {$\mathop{\mathrm{Res}}
				\limits_{\textnormal{$\w=\z$}}$};
				\phantom{\draw[postaction={decorate,decoration={markings,
						mark=between positions 0.1 and 0.9 step 0.4 with {\draw [->] (0,0) -- (.1,0);}}}]
				(-2.25,-2.5)
				to [out=135,in=-80] (-3.1,-1)
				to [out=100,in=-100] (-3.1,1)
				to [out=80,in=-135] (-2.25,2.5) node[right] {$\z$};}
				\draw[postaction={decorate,decoration={markings,
						mark=at position 0.5 with {\draw [->] (0,0) -- (.1,0);}}}]
				(-3.1,-1) to [out=100,in=-100] (-3.1,1);
				\begin{scope}[shift={(-3.1,1)},scale=.6]
						\draw[line width=2.5] (-.2,-.2)--++(.4,.4);
						\draw[line width=2.5] (-.2,.2)--++(.4,-.4);
						\node at (-.48,.6) {$\zfin$};
				\end{scope}
				\begin{scope}[shift={(-3.1,-1)},scale=.6]
						\draw[line width=2.5] (-.2,-.2)--++(.4,.4);
						\draw[line width=2.5] (-.2,.2)--++(.4,-.4);
						\node at (-.4,-.4) {$\zbfin$};
				\end{scope}
				\node[anchor=west] at (-3.2,-2.9) {\large{}\phantom{($**$)}};
		\end{tikzpicture}}
		\caption{Unifying two circles into a single $\w$ contour.}
		\label{fig:dragging}
\end{figure}

    The expression coming from the residue of the integrand at $\w=\z$ behaves as
    \begin{multline}\label{discrete_sine_asympt}
        -\frac{1}{2\pi\i}\frac{(T+\D t)!\cdot T}{(T-1)!}
        \frac{(T\w+1)_{T-1}}{(T\w-\D x)_{T+\D t+1}}\\=
        -\frac{1}{2\pi\i}
        \big(1+O(T^{-1})\big)T^{\D t+2}
        \frac{\G(T \w+T)\G(T\w-\D x)}{\G(T\w+1)\G(T\w-\D x+T+\D t+1)}\\=
        -\frac{1}{2\pi\i}
        \big(1+O(T^{-1})\big)
        \w^{-\D x-1}
        (1+\w)^{-\D t+\D x-1},
    \end{multline}
    where we used
    \eqref{Stirling_Gamma}, and the
    asymptotic expression is valid for all $\w\in\CC_\w^{\deps}$
    (for real $\w<-1$ one should apply \eqref{Gamma_function_flip}
    to all four gamma functions before
    using \eqref{Stirling_Gamma}).

    The right-hand side of \eqref{discrete_sine_asympt} above
    has singularities at $\w=-1$ and $\w=0$, and the arc
    of the contour $\CC_\z^{\deps}$ between the critical points $\zbfin$ and $\zfin$
    crosses $(-1,0)$.
    The integral of the error $O(T^{-1})$
    in \eqref{discrete_sine_asympt} over the arc
    from $\zbfin$ to $\zfin$
    is
    bounded by $O(|\zfin|/T)$, which is $O(1/T)$ because $\zfin$ belongs to a
    compact set $\mathcal Z$.

    Let us now identify the extended sine kernel
    \eqref{incomplete_beta} in the remaining terms outside the double contour integral over
    $\CC_\z^{\deps}$ and $\CC_\w^{\deps}$.
    Observe that for $\D t\ge0$ we have
    \begin{equation*}
        \mathop{\mathrm{Res}}_{\w=0}
        \Big(\w^{-\D x-1}
        (1+\w)^{-\D t+\D x-1}\Big)=-(-1)^{\D x+1}\binom{\D t}{\D x}\mathbf{1}_{\D x\ge0}.
    \end{equation*}
    By dragging the integration arc through $0$ for $\D t>0$ we obtain
    \begin{multline*}
        \mathbf{1}_{\D x\ge0}\mathbf{1}_{\D t>0}
        (-1)^{\D x+1}
        \binom{\D t}{\D x}
        -\frac{1}{2\pi\i}\int_{\zfin}^{\zbfin}
        \w^{-\D x-1}
        (1+\w)^{-\D t+\D x-1}d\w
        \\=
        -\frac{1}{2\pi\i}\int_{\zfin}^{\zbfin}
        \w^{-\D x-1}
        (1+\w)^{-\D t+\D x-1}d\w,
    \end{multline*}
    where in the right-hand side the arc
    crosses $(-1,0)$ for $\D t\le 0$ and $(0,+\infty)$ for $\D t>0$.
    Changing the variables in the right-hand side as $\w=\frac{z}{1-z}$
    (so $z=\frac{\w}{1+\w}$)
    turns the above integral into
    $K_{\zfin/(\zfin+1)}(T+\D t,\D x;T,0)$ \eqref{incomplete_beta},
    as desired.
\end{proof}

To complete the proof of \Cref{Theorem_main} it remains to show that the double contour integral in \eqref{K_Beta_paths_form} is negligible as $N\to+\infty$. It has the form (cf. \eqref{integrand_asymptotics}, \eqref{S_T_action})
\begin{equation}\label{negligible_double_integral}
    \frac1{(2\pi\i)^{2}}
    \int_{\CC_\z^{\deps}}
    d\z
    \oint_{\CC_\w^{\deps}}d\w\,\frac{1}{\w-\z}\exp\Big\{T\big(\Sfin(\z)-\Sfin(\w)\big)\Big\}
    \frac{(T+\D t)!\cdot T}{(T-1)!}
    \frac{(T\w+1)_{T-1}}{(T\w-\D x)_{T+\D t+1}}.
\end{equation}
We need the following statement which we prove later
in \Cref{sub:convergent_initial_data}:

\begin{lemma}\label{lemma:length_bounded}
    Under Assumptions \ref{ass:density} and \ref{ass:intermediate_behavior}
    the length of 
	$\CC_\w^{\deps}$ is bounded uniformly in $N$.
\end{lemma}

This fact together with 
\eqref{discrete_sine_asympt} 
implies that 
that the parts in \eqref{negligible_double_integral} outside the exponent are 
bounded by a constant depending on $\D x,\D t$.

For $\z$ and $\w$ in a fixed small neighborhood of the critical point
$\zfin=\zfin(N)$ which is bounded away from $\R$ we can Taylor expand the
function $\Sfin(\z)$. 
Because
the second derivative of $\Sfin$ is nonzero by 
\Cref{lemma:second_derivative},
this leads to a convergent integral times $T^{-\frac12}$
which goes to zero. This is a standard part of the steepest descent analysis,
and we refer to, e.g., \cite[Section 3]{Okounkov2002} for details.

Consider the situation when $\z$ and $\w$ are outside of this neighborhood of $\zfin$. 
On the parts of the contours $\CC_\z^{\deps}$ and $\CC_\w^{\deps}$ 
inside the rectangle $\mathcal{R}$ we have
the steepest descent/ascent properties.
Together with \Cref{lemma:second_derivative} 
they imply that 
outside a sufficiently small neighborhood of 
$\zfin$ and 
for a sufficiently small fixed $\delta>0$
(both depend only on the constants in Assumptions
\ref{ass:density} and \ref{ass:intermediate_behavior}):
\begin{equation*}
	\begin{split}
		\Re\Sfin(\z)-\Re\Sfin(\zfin)&<-\delta;
		\\
		\Re\Sfin(\w)-\Re\Sfin(\zfin)&>\delta.
	\end{split}
\end{equation*}
Along the part of the of the $\z$ contour escaping to infinity
$\Re\Sfin(\z)$ cannot increase 
due to the second claim of \Cref{lemma:estimate_real}.

Let us consider the possible change of 
$\Re\Sfin$ along $\CC_\z^{\deps}$
and
$\CC_\w^{\deps}$
close to the real line.
The vertical segments crossing the real line at points
$(k+1/2)/T$, $k\in \mathbb{Z}$, 
have length $2\deps$, and due to the 
first claim of \Cref{lemma:estimate_real}
we see that the change of 
$\Re\Sfin$ is of order $\deps$.
The horizontal segments have length $C\cdot (\deps\log^2\deps+T^{-1})$
for some $C>0$ independent of $N$ or~$\deps$ (but $C$ might depend on $\D x,\D t$).
Using the third claim of \Cref{lemma:estimate_real} we can 
upper bound the absolute value of the change 
of $\Re \Sfin$ along the
horizontal parts of the contours by a constant times
\begin{equation*}
	\left( -\log\deps+\frac{1}{T^2\deps^2} \right)\left( \deps\log^2\deps+\frac{1}{T} \right)=
	\frac{1}{\varepsilon^2 T^3}+\frac{\log^2\varepsilon}{T^2\varepsilon}-
	\frac{\log\varepsilon}{T}
	-\deps \log^3\deps.
\end{equation*}
This can be made much smaller than $\delta$:
first choose $\eps$ that the forth term is small, and
then choose $N$ (thus $T(N)$) large enough so that the first three
terms are also small.
Therefore, 
the whole double contour integral \eqref{negligible_double_integral} is
negligible in the limit.
This completes the proof
of \Cref{Theorem_main}.


\subsection{Convergent initial data and proofs of \Cref{Theorem_main_convergence,Theorem_main_convergence_glob,Theorem_main_convergence_Lebesgue}} 
\label{sub:convergent_initial_data}

In this subsection we present 
proofs of \Cref{Theorem_main_convergence,Theorem_main_convergence_glob,Theorem_main_convergence_Lebesgue}
describing the convergence of the point processes to the extended sine process 
under suitable additional assumptions.
Moreover, using similar arguments
we prove \Cref{lemma:second_derivative,lemma:length_bounded}
which were formulated in the previous two subsections.
These lemmas are not directly involved in the proofs
of 
\Cref{Theorem_main_convergence,Theorem_main_convergence_glob,Theorem_main_convergence_Lebesgue}.

In addition to Assumptions \ref{ass:density} and \ref{ass:intermediate_behavior}, let
\begin{equation}\label{local_measure_convergence_S6}
    \lim_{N\to+\infty}\frac{1}{T(N)}\sum_{i=1}^N \delta_{{a_i(N)}/{T(N)}}=\mu_{loc},
\end{equation}
where $\mu_{loc}$ is a $\sigma$-finite measure, and the limit is understood according to \Cref{def:vague}. By Assumption \ref{ass:density}, $\mu_{loc}$ has a density (with respect to the Lebesgue measure) which is between $\lod$ and $\upd$.

Let us also assume that the quantities \eqref{drift_N} have a limit $\CCC(\RRR)=\lim_{N\to+\infty}\CCC_N(\RRR)$. Then the meromorphic function $\Sfin'(\z)$ \eqref{S_T_prime} has a limit as $N\to+\infty$:
\begin{lemma}\label{lemma:S_prime_limit}
    Under the above assumptions and notation, we have $\lim_{N\to+\infty}\Sfin'(\z) =\Sfin_*'(\z)$ for all $\z\in\HH$, where
    \begin{multline}\label{S_star_function}
        \Sfin_*'(\z)=
        \int_{-\infty}^{\infty}
        \left(\frac{1}{\z-v}
        +\frac{\mathbf{1}_{|v|>\RRR}}{v}
        \right)\mu_{loc}(dv)-\CCC(\RRR)
        \\+\log (\z+1)-\log \z+\i\pi
        -\log(\be^{-1}-1),
    \end{multline}
    and $\RRR>0$ is arbitrary (the limit does not depend on $\RRR$). The convergence is uniform in $\z$ belonging to compact subsets of $\HH$.
\end{lemma}
\begin{proof}
    Fix $\z\in\HH$. Let us use formula \eqref{Sfin_prime_initial_formula} for $\Sfin'(\z)$. First, we have
    \begin{equation*}
        -\pi\cot(\pi T\z)=
        -\i\pi
        \frac{e^{\i\pi T\z}+e^{-\i\pi T\z}}
        {e^{\i\pi T\z}-e^{-\i\pi T\z}}\to \i\pi,
    \end{equation*}
    because $e^{-\i\pi T\z}$ dominates for
    $\z\in\HH$.

    Next, the sum over $i=1,\ldots,T-1$ approximates the corresponding Riemann integral:
    \begin{equation*}
        \sum_{i=1}^{T-1}\frac{1}{T\z+i}
        =
        \frac{1}{T}\sum_{i=1}^{T-1}\frac{1}{\z+i/T}
        \to
        \int_{0}^{1}\frac{dv}{\z+v}=\log(\z+1)-\log\z,
    \end{equation*}
    and the convergence is uniform over $\z$ in compact subsets of $\HH$.

    Finally, recall the definition of the atomic measure \eqref{atomic_measure}, and note that \eqref{local_measure_convergence_S6} means that the measures $\m_{T}[\mathfrak{a}]$ vaguely converge to $\mu_{loc}$. The remaining $N$-dependent part of \eqref{Sfin_prime_initial_formula} can be written as
    \begin{equation*}
        \frac{1}{T}\sum_{r=1}^{N}\frac{1}{\z-a_r/T}
        =\int_{-\infty}^{+\infty}
        \frac{\m_{T}[\mathfrak{a}](dv)}{\z-v}.
    \end{equation*}
    Since the function $1/(\z-v)$ in $v$ does not have compact support, and its integral with respect to the Lebesgue measure diverges at infinity, one cannot directly apply \eqref{local_measure_convergence_S6} to the integral above. Here we need a regularization afforded by the convergence of the constants $\CCC_N(\RRR)$ \eqref{drift_N}. Namely, take any $\RRR>0$ and write
    \begin{equation}\label{S_prime_limit_proof}
        \int_{-\infty}^{+\infty}
        \frac{\m_{T}[\mathfrak{a}](dv)}{\z-v}=
        \int_{-\infty}^{+\infty}
        \left(\frac{1}{\z-v}
        +\frac{\mathbf{1}_{|v|>\RRR}}{v}
        \right)\m_{T}[\mathfrak{a}](dv)-
        \CCC_N(\RRR)
    \end{equation}
    (this expression does not depend on $\RRR$). Now the function under the integral decays as $v^{-2}$ at infinity, and so is Lebesgue integrable. Since the density of $\mu_{loc}$ is bounded, by restricting the integration to $[-M,M]$ for large $M$ and applying \eqref{local_measure_convergence_S6}, we conclude that the integral converges to the corresponding integral with respect to $\mu_{loc}$. The uniformity of each convergence above is evident, so this completes the proof.
\end{proof}

We are now in a good position to prove
the helpful 
\Cref{lemma:second_derivative,lemma:length_bounded}
formulated
previously: their proofs are similar
to each other and to that of \Cref{lemma:S_prime_limit}.

\begin{proof}[Proof of \Cref{lemma:second_derivative,lemma:length_bounded}]
		We will prove both statements 
		simultaneously. 
		If the contrary
		to one of the lemmas holds, 
		then
		there exists a subsequence $\{N_k\}$ along which
		the local measures converge in the sense of
		\eqref{local_measure_convergence_S6} (consider measures of segments with
		rational endpoints and choose a diagonal subsequence), the constants
		$\CCC_{N_k}(\RRR)$ converge to $\CCC(\RRR)$, 
		but 
		\begin{itemize}
			\item 
				(\Cref{lemma:second_derivative}) 
				The derivatives $\Sfin''(z_{N_k})$ converge to zero 
				or infinity along 
				a subsequence $z_{N_k}$ belonging to a compact 
				subset of the upper half plane.
				Further choosing a subsequence of 
				$\{ N_k \}$
				we may assume that along this subsequence 
				the $z_{N_k}$'s converge to some point
				$\tilde z$ is in the upper half plane.
			\item 
				(\Cref{lemma:length_bounded})
				The length of the
				contour $\CC_\w^{\deps}(N_k)$ grows to infinity.
		\end{itemize}
		To simplify notation, let us
		use the sequence $\{N\}$ instead of $\{N_k\}$ in the rest of the proof.

\smallskip

		Let us now show that there are constants $c_{N}$ such that there exists
		$\lim_{N\to+\infty}(\Sfin(\z)-c_N)$, call it $\Sfin_*(\z)$, uniformly in
		$\z$ belonging to bounded subsets of $\HH$. Moreover, $\Sfin_*$ is holomorphic in
		$\HH$, and its derivative is given by \eqref{S_star_function}.

    Fix $\z\in\HH$. Recall that $\Sfin(\z)$ is given by \eqref{S_T_action}, and let us consider the five summands in that formula separately. First, observe that
    \begin{equation*}
        -\frac 1T\log_{\HH}(\sin(\pi T \z))
        +
        \frac {2\pi\i}T
        \left\lfloor
        \tfrac12\Re(T\z)+\tfrac12
        \right\rfloor
        -\z\log(\be^{-1}-1)\to \i\pi\z-\z\log(\be^{-1}-1)
    \end{equation*}
    due to \eqref{log_branch2}, \eqref{log_branch_large_im}. Next, the second sum in \eqref{S_T_action} approximates a convergent Riemann integral:
    \begin{equation*}
        \frac 1T\sum_{i=1}^{T-1}\log_{\HH}\Big(\z+\frac iT\Big)\to \int_{0}^{1}\log(\z+v)dv=
        (\z+1) \log (\z+1)-\z \log \z-1
    \end{equation*}
    (in $\HH$ the branches $\log$ and $\log_\HH$ coincide). The first sum in \eqref{S_T_action} can be rewritten as
    \begin{multline*}
        \frac 1T\sum_{r=1}^{N}\log_{\HH}
        \Big(\z-\frac{a_r}T\Big)
        =
        \int_{-\infty}^{+\infty}
        \log(\z-v)\m_{T}[\mathfrak{a}](dv)
        \\=
        \int_{-\infty}^{+\infty}
        \biggl(\log(\z-v)-\log(\i-v)+
        (\z-\i)\frac{\mathbf{1}_{|v|>\RRR}}{v}\biggr)
        \m_{T}[\mathfrak{a}](dv)
        \\+
        \int_{-\infty}^{+\infty}\log(\i-v)
        \m_{T}[\mathfrak{a}](dv)
        -(\z-\i)\CCC_N(\RRR).
    \end{multline*}
    The integrand $\log(\z-v)-\log(\i-v)+(\z-\i)\frac{\mathbf{1}_{|v|>\RRR}}{v}$ decays as $v^{-2}$ at infinity, so the first integral in the right-hand side converges as $N\to+\infty$ to the same integral over $\mu_{loc}$. The second integral in the right-hand side does not depend on $\z$, call it $c_N$. The third summand converges to $-(\z-\i)\CCC(\RRR)$. Thus, we have the convergence $\Sfin(\z)-c_N\to\Sfin_*(\z)$ to a holomorphic function, uniformly in $\z$ in compact subsets in $\HH$. One can check that the derivative of $\Sfin_*$ is \eqref{S_star_function}.

\smallskip

	Once the existence of 
	the uniform limit
	$\Sfin_*(\z)=\lim_{N\to+\infty}(\Sfin(\z)-c_N)$
	is established, 
	we continue with separate arguments:
	\begin{itemize}
		\item 
			(\Cref{lemma:second_derivative})
			We have $\Sfin_*''(\tilde z)=0$ or $\infty$  and $\tilde z\in \HH$.
			The second case is not possible since $\Sfin_*$ is holomorphic
			in $\HH$. 
			If $\Sfin_*''(\tilde z)=0$
			then by Hurwitz's theorem 
			for all sufficiently large $N$
			there exist two complex critical points of 
			$\Sfin-c_N$ (equivalently, of $\Sfin$)
			in the upper half plane. 
			This is impossible by \Cref{lemma:S_T_critical_points_equation}.
			So in either case we get a contradiction.

		\item (\Cref{lemma:length_bounded})
			Observe that the length of the part of the contour $\CC_\w^{\deps}$ close
			to the real line is bounded. From the convergence of $\Sfin(\z)-c_N$ it
			follows that away from the real line the contour $\CC_\w^{\deps}$
			approximates the corresponding contour for $\Sfin_*(\z)$, and the latter
			has finite length. We get a contradiction, too.
	\end{itemize}
	This proves both desired statements.
\end{proof}

\begin{remark}\label{rmk:passing_to_subsequences}
		The previous argument
		shows that the additional hypotheses of \Cref{Theorem_main_convergence},
		namely, convergence of measures \eqref{local_measure_convergence_S6} and
		convergence $\CCC(\RRR)=\lim_{N\to+\infty}\CCC_N(\RRR)$, are not too
		restrictive.
\end{remark}

\begin{proof}[Proof of \Cref{Theorem_main_convergence}]
    By Hurwitz's theorem, the critical points $\zfin(N)\in\HH$
    of $\Sfin'$ converge, as $N\to+\infty$,
    to a critical point $\zlim\in\HH$, which belongs
    to the compact set $\mathcal Z$
    from \Cref{Theorem_main}.
    Applying the latter, we get the desired convergence.
    Equation \eqref{Theorem_main_convergence_equation}
    is simply $\Sfin'_*(\z)=0$ under a change of variables
    $\z=\uvar/(1-\uvar)$, which maps the
    upper half plane onto itself.
\end{proof}

As \Cref{prop:nonLeb_IC} is a particular case of \Cref{Theorem_main_convergence}, let us give its proof here:
\begin{proof}[Proof of \Cref{prop:nonLeb_IC}]
    Under the hypotheses of \Cref{prop:nonLeb_IC}, Assumptions \ref{ass:density} and \ref{ass:intermediate_behavior} clearly hold, and, moreover, $\CCC_N(\RRR)$ is close to zero for $\RRR>3$. Therefore, the function $\Sfin_*'(\z)$ \eqref{S_star_function} looks as
    \begin{align*}
        \Sfin_*'(\z)&=
        \frac12\int_{-3h}^{0}
        \frac{dv}{\z-v}+
        \frac13\int_{0}^{3h}
        \frac{dv}{\z-v}
        +\frac12\,\pv\int_{|v|>3h}\frac{dv}{\z-v}
        \\&\hspace{180pt}+\log (\z+1)-\log \z+\i\pi
        -\log(\be^{-1}-1)\\
        &=
        \frac{\i\pi}2-\frac16\log\z+\frac16\log(\z-3h)
        +\log (\z+1)-\log \z
        -\log(\be^{-1}-1)
        .
    \end{align*}
    After the substitution $\z=\uvar/(1-\uvar)$ this leads to the equation \eqref{eq_non_leb_example} for the complex slope.
\end{proof}

\begin{proof}[Proof of \Cref{Theorem_main_convergence_glob}]
    Fix an arbitrary $\eps>0$. Choose and fix $\updelta>0$ so small that $\A_{\RRR,\updelta}<\eps/3$ and that
    \begin{equation*}
        \left|\int_{|v|>\updelta}
        \frac{\mu_{glob}(dv)}{v}
        -
        \pv\int_{-\infty}^{\infty}
        \frac{\mu_{glob}(dv)}{v}
        \right|<\frac{\eps}{3}.
    \end{equation*}
    This approximation is possible because $\mu_{glob}$ is a probability measure, so $1/v$ is integrable at infinity, and thus the only singularity is at zero. Next, let $N$ be so large that
    \begin{equation*}
        \left|
        \sum_{i\colon
        |a_i(N)|>\updelta N}\frac{1}{a_i(N)}
        -\int_{|u|>\updelta}
        \frac{\mu_{glob}(dv)}{v}\right|<\frac{\eps}3.
    \end{equation*}
    This is possible because the sum above is the same as the integral of $\mathbf{1}_{|v|>\updelta}/v$ with respect to the atomic measure $\mu_{glob}^{N}=\frac{1}{N}\sum_{i=1}^N \delta_{a_i(N)/N}$ converging vaguely to $\mu_{glob}$. Here the vague convergence implies the convergence on the function $\mathbf{1}_{|v|>\updelta}/v$ because we are dealing with probability measures and so can cut away the tails at infinity. These estimates imply that
    \begin{equation*}
        \CCC_N(\RRR)=
        \sum_{i\colon \RRR T\le|a_i(N)|\le \updelta N}
        \frac{1}{a_i(N)}
        +
        \int_{|v|>\updelta}
        \frac{\mu_{glob}^{N}(dv)}{v}
    \end{equation*}
    is close to $\pv\int_{-\infty}^{\infty}v^{-1}\mu_{glob}(dv)$ within $\eps$, and so an application of \Cref{Theorem_main_convergence} gives the result.
\end{proof}

\begin{proof}[Proof of \Cref{Theorem_main_convergence_Lebesgue}]
    When $\mu_{loc}$ is a multiple of the Lebesgue measure, the integral in \eqref{S_star_function} can be explicitly computed:
    \begin{multline*}
        \int_{-\infty}^{\infty}
        \left(\frac{1}{\z-v}
        +\frac{\mathbf{1}_{|v|>\RRR}}{v}
        \right)dv=\int_{-\RRR}^{\RRR}\frac{dv}{\z-v}+
        \lim_{M\to+\infty}\int\limits_{\RRR<|v|<M}
        \left(\frac{1}{\z-v}
        +\frac{1}{v}
        \right)dv\\
        =
        \log (\z+\RRR)-\log (\z-\RRR)
        +\lim_{M\to+\infty}
        \Bigl(\log (\z+M)-\log (\z-M)+
        \log (\z-\RRR)-\log (\z+\RRR)\Bigr)
        \\=
        \lim_{M\to+\infty}
        \Bigl(\log (\z+M)-\log (\z-M)\Bigr),
    \end{multline*}
    where the branches of all the logarithms above are standard. The last limit is equal to $-\i\pi$ because $\arg(\z+M)\to0$ while $\arg(\z-M)\to\pi$. This immediately leads to the desired formula for $\ulim$ in \Cref{Theorem_main_convergence_Lebesgue}.
\end{proof}
\begin{remark}
    When $\mu_{loc}$ is a multiple of the Lebesgue measure, the above computation shows that the integral in \eqref{S_star_function} is independent of $\RRR$, and hence $\CCC(\RRR)$ is, too. This agrees with the fact that the difference $\CCC_N(\RRR)-\CCC_N(\RRR')$ (where, say, $\RRR>\RRR'$) is equal to the integral of ${\m_{T}[\mathfrak{a}](dv)}/{v}$ over $\RRR'<|v|<\RRR$, and thus vanishes as $N\to+\infty$.
\end{remark}



\section{Applications: proofs of
\Cref{Theorem_f_IC,Theorem_Bernoulli_IC,Proposition_sine_IC}} 
\label{sec:applications}

\subsection{Discretization of a continuous profile: proof of \Cref{Theorem_f_IC}} 
\label{sub:discretization_of_limiting_profile_proof_of_theorem_f_ic}

Let us show that the initial data \eqref{f_IC} defined using a
twice continuously differentiable function $f$ satisfies Assumptions
\ref{ass:density} and \ref{ass:intermediate_behavior}, as well as additional
hypotheses of \Cref{Theorem_main_convergence}.

Since $1<f'(x)<+\infty$ for all $x\in[-\frac12,\frac12]$,
Assumption \ref{ass:density} holds with
\begin{equation*}
    \lod=\frac12\inf_{x\in[-\frac12,\frac12]}\frac{1}{f'(x)},\qquad
    \upd=\frac12+\frac12\sup_{x\in[-\frac12,\frac12]}\frac{1}{f'(x)}
\end{equation*}
on scales, say, $\dd(N)=\fl{N^{\TNPower/2}}$ and $\QQ(N)=\fl{N^{(1+\TNPower)/2}}$.

Next, the local and global measures
as in \eqref{mu_loc_mu_glob_convergence} exist,
$\mu_{loc}$ is the Lebesgue measure on $\R$
times $q=1/f'(\chi)$ (recall that $f(\chi)=0$),
and $\mu_{glob}$
has the density
\begin{equation*}
    \mu_{glob}(dv)=\frac{dv}{f'(f^{-1}(v))}.
\end{equation*}
Because $f'>1$,
the principal value integral
\begin{equation}\label{Theorem_f_IC_proof}
    \pv\int_{-\infty}^{+\infty}
    \frac{dv}{vf'(f^{-1}(v))}=
    \pv\int_{-\frac12}^{\frac12}\frac{dx}{f(x)}
\end{equation}
also exists.

Let us now consider the quantities $\CCC_N(\RRR)$ \eqref{drift_N}.
Replace the condition
$|a_i|\ge\RRR T$
on $i$ in \eqref{drift_N} by
$|i-N\chi|>\RRR T/f'(\chi)$.
The difference between the two sums can be estimated by
a part of the harmonic series between $\RRR T$ and
$\RRR T+C T^2/N$, which is negligible.
Thus,
\begin{equation}\label{Theorem_f_IC_proof_2}
    \lim_{N\to+\infty}\CCC_N(\RRR f'(\chi))=
    \lim_{N\to+\infty}
    \sum_{i\colon |i-N\chi|>\RRR T}
    \frac{1}{\fl{N f(i/N)}}=
    \lim_{N\to+\infty}
    \frac{1}{N}\sum_{i\colon |i/N-\chi|>\RRR T/N}
    \frac{1}{f(i/N)}.
\end{equation}
Taylor expand $f(i/N)=(i/N-\chi)f'(\chi)+(i/N-\chi)^{2}r$,
where $|r|$ is uniformly bounded (here we use that
$f$ is twice continuously differentiable). Observe that
the sum of
\begin{equation*}
    \frac{1}{Nf(i/N)}-\frac{1}{(i-N\chi)f'(\chi)}=
    -\frac{r}{N}\cdot\frac{1}{f'(\chi)^{2}+r f'(\chi )(i/N-\chi)}
\end{equation*}
over $i$ such that $\RRR T/N<i/N-\chi<\updelta$ (i.e., the one-sided sum)
is bounded for sufficiently small $\updelta>0$, and goes
to zero as $\updelta\to0$,
and similarly for
$-\updelta<i/N-\chi<-\RRR T/N$.
The sum
of $1/\big((i-N\chi)f'(\chi)\big)$ over $|i-N \chi|>\RRR T$
(i.e., the symmetric sum) is negligible for large $N$.
Thus, the last sum in
\eqref{Theorem_f_IC_proof_2}
over $|i/N-\chi|>\RRR T/N$
is close to the same sum over
$|i/N-\chi|>\updelta$ for small $\updelta$,
and the latter approximates the principal value integral \eqref{Theorem_f_IC_proof}.
This
implies Assumption \ref{ass:intermediate_behavior}
and the property that the $\CCC_N(\RRR)$'s
converge as $N\to+\infty$.
This completes the proof of
\Cref{Theorem_f_IC}.


\subsection{Random initial configuration} 
\label{sub:random_initial_configuration}

Let us now consider
the noncolliding Bernoulli random walk
started from a
random initial configuration.
We assume that this random configuration
belongs to
$\{-M,-M+1,\ldots,M\}$,
where
$M\to+\infty$ is our main large parameter.
Denote by
\begin{equation}\label{W(M)}
    \mathbb{W}(M)=\bigcup_{k=0}^{2M+1}
    \big\{
    \vec x\in\W k\colon {-M}\le x_1<\ldots<x_k\le M
    \big\}
\end{equation}
the space of possible initial configurations (cf. \eqref{Weyl_chamber}). The law of the initial configuration will be denoted by $\mathbf{P}^{M}$, and the configuration itself by $\boldsymbol{\mathfrak{A}}(\NRandom)$ (here $\NRandom$ can be random). Let $\vec{\mathbf{X}}(t)$ stand for the noncolliding Bernoulli random walk started from $\boldsymbol{\mathfrak{A}}(\NRandom)$.

When does $\vec{\mathbf{X}}(t)$ satisfy an annealed\footnote{That is, with respect to the combined randomness coming from the initial configuration and from the random walk itself.} bulk limit theorem similar to the ones formulated in \Cref{sub:main_result_bulk_limit_theorems}? Informally speaking, this happens when the walk started from a fixed ``typical''configuration (with respect to $\mathbf{P}^{M}$) satisfies a bulk limit theorem with a constant complex slope $\ulim$ (i.e., independent of the randomness coming from the random initial data). Let us formalize this understanding, cf. \cite{duse2015cusp}, \cite{Gorin2016} for recent annealed limit theorems for uniformly random tilings.

\begin{proposition}\label{prop:annealed}
    Choose and fix a time scale $T=T(M)\ll M$, $T(M)\to+\infty$. Suppose that there exist subsets $\mathbb{W}_{reg}(M)\subset \mathbb{W}(M)$, $M=1,2,\ldots$, such that
    \begin{enumerate}[\bf{}1.]
        \item $\lim_{M\to+\infty}\mathbf{P}^{M} (\mathbb{W}_{reg}(M))=1$;
        \item For any fixed sequence of (nonrandom) initial configurations $\mathfrak{A}(N_{M})\in\mathbb{W}_{reg}(M)$ the bulk limit theorem near $x=0$ (i.e., the conclusion of \Cref{Theorem_main_convergence}) holds for a complex slope $\ulim$ independent of this sequence $\mathfrak{A}(N_{M})$.
    \end{enumerate}
    Then as $M\to+\infty$ the point process describing $\{\vec{\mathbf{X}}(T(M)+t)\}_{t}$ near $x=0$ converges in distribution to the extended sine process of the complex slope $\ulim$.
\end{proposition}
In particular, hypotheses of \Cref{prop:annealed} imply that the random number $\NRandom$ of particles in the initial configuration goes to infinity in probability (with respect to the $\mathbf{P}^{M}$'s).
\begin{proof}[Proof of \Cref{prop:annealed}]
    Fix an event of the form
    \begin{equation*}
        F=\Big\{
        \textnormal{the configuration
        $\vec{\mathbf{X}}(T+t_i)$ on $\Z$
        contains the point $y_i$ for all $i=1,\ldots,k$
        }\Big\},
    \end{equation*}
    where $k=1,2,\ldots$, and $t_i,y_i\in\Z$
    (cf. \eqref{determinantal_kernel}).
    Such events generate
    the
    $\sigma$--algebra
    describing the
    configuration
    $\{\vec{\mathbf{X}}(T+t)\}_{t}$
    near $x=0$.

    Let $\PP_{\mathfrak{A}(N)}$ stand for the
    law of the noncolliding Bernoulli random walk
    started from the
    initial configuration $\mathfrak{A}(N)$.
    We need to show that
    \begin{equation}\label{annealed_convergence}
        \lim_{M\to+\infty}
        \mathbf{E}_{\mathbf{P}^{M}}
        \big[
        \PP_{\boldsymbol{\mathfrak{A}}(\NRandom)}(F)\big]=
        \det\big[K_{\ulim}(t_\aind,y_\aind;t_\bind,y_\bind)\big]
        _{\aind,\bind=1}^{k},
    \end{equation}
    where $K_{\ulim}$ is the extended sine kernel
    \eqref{incomplete_beta},
    and
    $\mathbf{E}_{\mathbf{P}^{M}}$ denotes the expectation
    with respect to
    $\mathbf{P}^{M}$.
    We have
    \begin{equation*}
        \mathbf{E}_{\mathbf{P}^{M}}
        \big[
        \PP_{\boldsymbol{\mathfrak{A}}(\NRandom)}(F)\big]
        =
        \mathbf{E}_{\mathbf{P}^{M}}
        \big[
        \PP_{\boldsymbol{\mathfrak{A}}(\NRandom)}(F)
        \mathbf{1}_{\mathbb{W}_{reg}(M)}
        \big]
        +
        \mathbf{E}_{\mathbf{P}^{M}}
        \big[
        \PP_{\boldsymbol{\mathfrak{A}}(\NRandom)}(F)
        \mathbf{1}_{\mathbb{W}(M)
        \setminus\mathbb{W}_{reg}(M)}
        \big].
    \end{equation*}
    The second summand goes to zero by hypothesis 1,
    and the first
    summand can be
    estimated as
    \begin{equation}\label{annealed_convergence_2}
        \PP_{\mathfrak{A}_{\min}(N_{\min})}(F)
        \le\mathbf{E}_{\mathbf{P}^{M}}
        \big[
        \PP_{\boldsymbol{\mathfrak{A}}(\NRandom)}(F)
        \mathbf{1}_{\mathbb{W}_{reg}(M)}
        \big]\le
        \PP_{\mathfrak{A}_{\max}(N_{\max})}(F),
    \end{equation}
    where $\mathfrak{A}_{\max}(N_{\max})$
    is the configuration
    which maximizes $\PP_{\mathfrak{A}}(F)$
    over all $\mathfrak{A}\in\mathbb{W}_{reg}(M)$
    (it exists because this is a finite set),
    and similarly for
    $\mathfrak{A}_{\min}(N_{\min})$.
    Because both minimizing and maximizing configurations
    belong to $\mathbb{W}_{reg}(M)$,
    both bounds in
    \eqref{annealed_convergence_2}
    converge to the right-hand side of \eqref{annealed_convergence}
    by hypothesis 2,
    and so the desired convergence holds.
\end{proof}


\subsection{Bernoulli initial data: proof of \Cref{Theorem_Bernoulli_IC}} 
\label{sub:bernoulli_initial_data_proof_of_theorem_bernoulli_ic}

Let the parameters $p$ and $\al$ be as in \Cref{Theorem_Bernoulli_IC},
the time scale be $T(M)=\fl{M^{\TNPower}}$, and the initial
particle configuration on
$\{-\fl{M(1-\al)},-\fl{M(1-\al)}+1,\dots,
\fl{M\al}-1,\fl{M\al}\}$ be obtained by
putting a particle at each location with
probability $p$ independently of all others.

We will construct a subset $\mathbb{W}_{reg}(M)\subset\mathbb{W}(M)$ \eqref{W(M)} satisfying \Cref{prop:annealed} by imposing two conditions (of asymptotic $\mathbf{P}^M$-probability $1$) on the configuration. First, fix $0<\delta<\min (1-\TNPower,\TNPower/7)$, and take scales $\dd(M)=\fl{M^{\TNPower- \delta}}$, $\QQ(M)=\fl{M^{\TNPower+\delta}}$. Denote
\begin{equation*}
    \mathbb{W}_{reg,1}(M)=
    \left\{
    \parbox{.62\textwidth}{
    $\vec x\in\mathbb{W}(M)$ such that
    in every segment of length
    $\dd(M)$ inside
    $[-\QQ(M),\QQ(M)]$
    the number of points in the configuration
    $\vec x$ is between $p\dd-\dd^{2/3}$
    and $p\dd+\dd^{2/3}$}
    \right\}.
\end{equation*}
Since the expected number of points in one of the segments of length $\dd$ is $p\dd$ and the variance is of order $\dd$, by the Chebyshev inequality the probability that in one of such segments the (random) number of points is not between $p\dd-\dd^{2/3}$ and $p\dd+\dd^{2/3}$ can be bounded from above by a constant times $\dd^{-1/3}$. The number of segments of length $\dd$ inside $[-\QQ,\QQ]$ is of order $\QQ/\dd$, and so
\begin{equation*}
    1-\mathbf{P}^M(\mathbb{W}_{reg,1}(M))\le C\cdot
    \dd^{-\frac13}\frac{\QQ}{\dd}=C\cdot M^{-\frac43(\TNPower- \delta)+
    \TNPower+\delta}=C\cdot M^{\frac13(7 \delta-\TNPower)}\to0.
\end{equation*}
Clearly, configurations in $\mathbb{W}_{reg,1}(M)$ satisfy Assumption \ref{ass:density}. Moreover, for these configurations the local density of particles at $0$ vaguely converges as $M\to+\infty$ to $p$ times the Lebesgue measure on $\R$.

Second, recall $\CCC_M(\RRR)$ defined by \eqref{drift_N} as a sum of $1/a_i$ over $|a_i|\ge\RRR T(M)$, and interpret it as a sum of independent random variables $\boldsymbol\delta_i/i$ over all $i\in\{-\fl{M(1-\al)},\ldots,\fl{M\al}\}$, where $\boldsymbol\delta_i$ is the indicator of the event that there is a point of the configuration at the location $i$. We have
\begin{equation*}
    \mathbf{E}_{\mathbf{P}^{M}}(\CCC_M(\RRR))=
    p\sum_{\substack{j\colon|j|\ge\RRR T(M)\\
    -\fl{M(1-\al)}\le j\le\fl{M\al}}}\frac{1}{j},\qquad
    \mathbf{Var}_{\mathbf{P}^{M}}(\CCC_M(\RRR))=
    p(1-p)\sum_{\substack{j\colon|j|\ge\RRR T(M)\\
    -\fl{M(1-\al)}\le j\le\fl{M\al}}}
    \frac{1}{j^{2}}.
\end{equation*}
We see that the expectation approximates an integral
\begin{equation*}
    \mathbf{E}_{\mathbf{P}^{M}}(\CCC_M(\RRR))=\pv\int_{-(1-\al)}^{\al}\frac{p}{v}\,dv+O(M^{-1})=
    p\log\bigg(\frac{\al}{1-\al}\bigg)+O(M^{-1}),
\end{equation*}
and $\mathbf{Var}_{\mathbf{P}^{M}}(\CCC_M(\RRR))=O(T(M)^{-1})$, so the random variable $\CCC_M(\RRR)$ converges as $M\to+\infty$ to the constant $\CCC=p\log(\frac{\al}{1-\al})$. Thus, if we define for some fixed $\RRR>0$:
\begin{equation*}
    \mathbb{W}_{reg,2}(M)=
    \bigg\{\vec x\in\mathbb{W}(M)\colon
    \bigg|p\log\bigg(\frac{\al}{1-\al}\bigg)-
    \sum_{i\colon|x_i|\ge\RRR T(M)}\frac{1}{x_i}\bigg|
    <M^{-\TNPower/3}
    \bigg\},
\end{equation*}
then by the Chebyshev inequality we have $\mathbf{P}^M(\mathbb{W}_{reg,2}(M))\to1$. Configurations from the sets $\mathbb{W}_{reg,2}(M)$ satisfy Assumption \ref{ass:intermediate_behavior}, and have $\CCC_M(\RRR)\to\CCC$.

Defining $\mathbb{W}_{reg}(M)=\mathbb{W}_{reg,1}(M)\cap \mathbb{W}_{reg,2}(M)$, we see that the $\mathbf{P}^{M}$-probabilities of these sets go to $1$, while any sequence of configurations from $\mathbb{W}_{reg}(M)$ satisfies the hypotheses of \Cref{Theorem_main_convergence} and hence the bulk limit theorem near $x=0$. Thus, applying \Cref{prop:annealed} we see that \Cref{Theorem_Bernoulli_IC} is established.


\subsection{Sine process initial data: proof of \Cref{Proposition_sine_IC}} 
\label{sub:sine_initial_data_proof_of_theorem_sine_ic}

Let the parameters $\phi$ and $\al$ be as in \Cref{Proposition_sine_IC}, the time scale be $T(M)=\fl{M^{\TNPower}}$, and the initial particle configuration be obtained by restricting the configuration of the discrete sine process of density $\phi/\pi$ to $\{-\fl{M(1-\al)},-\fl{M(1-\al)}+1,\dots, \fl{M\al}-1,\fl{M\al}\}$. We will use the same scales $\dd(M),\QQ(M)$ and sets $\mathbb{W}_{reg,1}(M), \mathbb{W}_{reg,2}(M)$ as in the Bernoulli case in \Cref{sub:bernoulli_initial_data_proof_of_theorem_bernoulli_ic} with $p$ replaced by $\phi/\pi$. This ensures that the second hypothesis of \Cref{prop:annealed} holds for $\mathbb{W}_{reg}(M)=\mathbb{W}_{reg,1}(M)\cap \mathbb{W}_{reg,2}(M)$. To establish \Cref{Proposition_sine_IC} it remains to show that
\begin{equation*}
    \lim_{M\to+\infty}
    \mathbf{P}^{M}(\mathbb{W}_{reg,1}(M))=
    \lim_{M\to+\infty}
    \mathbf{P}^{M}(\mathbb{W}_{reg,2}(M))=1,
\end{equation*}
where now
$\mathbf{P}^{M}$ stands for the law of
the initial configuration under the restriction of the
discrete sine process.
For $\mathbb{W}_{reg,1}(M)$, observe that
the variance under $\mathbf{P}^{M}$
of the number
of points in a segment of length $\dd$
(say, $\{1,\ldots,\dd\}$, since the
sine process is translation invariant)
can be
estimated as
\begin{multline}\label{variance_estimate}
    \mathbf{Var}_{\mathbf{P}^{M}}\bigg(\sum_{i=1}^{\dd}
    \boldsymbol{\delta}_i\bigg)=
    \frac{\phi}{\pi}\dd-
    \Bigl(\frac{\phi}{\pi}\dd\Bigr)^{2}+2\sum_{1\le i<j\le\dd}
    \mathbf{E}_{\mathbf{P}^{M}}
    (\boldsymbol{\delta}_i
    \boldsymbol{\delta}_j)\\=
    \dd\frac{\phi}{\pi}\Bigl(1-\frac{\phi}{\pi}\Bigr)-
    2\sum_{1\le i<j\le\dd}
    \biggl(\frac{\sin(\phi(i-j))}{\pi(i-j)}\biggr)^{2}
    \le \dd\frac{\phi}{\pi}\Bigl(1-\frac{\phi}{\pi}\Bigr),
\end{multline}
where $\boldsymbol{\delta}_i$ is as in
\Cref{sub:bernoulli_initial_data_proof_of_theorem_bernoulli_ic}.
This variance does not exceed
the one in the Bernoulli case,
and so
using the argument from
\Cref{sub:bernoulli_initial_data_proof_of_theorem_bernoulli_ic}
we conclude that
$\mathbf{P}^{M}(\mathbb{W}_{reg,1}(M))\to1$.

\begin{remark}
    In fact,
    the variance in the left-hand side of \eqref{variance_estimate}
    grows as $O(\log\dd)$,
    see \cite[Lemma 4.6]{BorFerr2008DF},
    \cite[Section 4.2]{Bufetov_Sr_entropy},
    but we do not need this for our proof.
\end{remark}

For $\mathbb{W}_{reg,2}(M)$
consider the random variable
$\CCC_M(\RRR)$ \eqref{drift_N}.
Arguing as in \Cref{sub:bernoulli_initial_data_proof_of_theorem_bernoulli_ic},
we see that
its expectation converges to
$\frac{\phi}{\pi}\log(\frac{\al}{1-\al})$. Let us
estimate its variance.
Observe that for any subset $B\subset\Z\setminus\{0\}$ one has
\begin{equation}\label{variance_estimate2}
    \mathbf{Var}_{\mathbf{P}^{M}}
    \bigg(
    \sum_{i\in B}\frac{\boldsymbol{\delta}_i}{i}
    \bigg)
    =
    \frac{\phi}{\pi}\Bigl(1-\frac{\phi}{\pi}\Bigr)
    \sum_{i\in B}\frac{1}{i^2}-
    2\sum_{i,j\in B\colon i<j}\frac{1}{ij}\biggl(\frac{\sin(\phi(i-j))}{\pi(i-j)}\biggr)^{2}.
\end{equation}
Apply this with
\begin{equation*}
    B=\{-\fl{M(1-\al)},\dots,
    -\RRR T-1,-\RRR T,\ldots,\RRR T,\RRR T+1,\ldots,
    \fl{M\al}-1,\fl{M\al}\}
\end{equation*}
for some fixed $\RRR\in\Z_{\ge1}$.
We see that as $M\to+\infty$,
the first sum in \eqref{variance_estimate2}
decays as $O(T^{-1})$. Throwing away the pairs
$(i,j)$ of the same sign from the second sum, we
can bound the second sum in \eqref{variance_estimate2} by
a constant times
\begin{equation}\label{variance_estimate3}
    \sum_{k,j\ge\RRR T}
    \frac{1}{kj(k+j)^{2}}\le
    \sum_{j\ge \RRR T}
    \biggl(\sum_{k\ge 1}
    \frac{1}{kj(k+j)^{2}}\biggr)\le
    \sum_{j\ge \RRR T}\frac{C_1+C_2\log j}{j^{2}}=
    O(T^{-1+\eps})
\end{equation}
for some $C_{1,2}>0$ and an arbitrary small $\epsilon>0$. Thus, the variance of $\CCC_M(\RRR)$ decays as $O(T^{-1+\eps})$, and so by the Chebyshev inequality we have $\mathbf{P}^M(\mathbb{W}_{reg,2}(M))\to1$. Applying \Cref{prop:annealed}, we see that \Cref{Proposition_sine_IC} holds.

\begin{remark}
    One can say that the constant $\CCC=\frac{\phi}{\pi}\log(\frac{\al}{1-\al})$
    in \Cref{Proposition_sine_IC}
    corresponds via \eqref{main_convergence_glob_drift_term}
    to the global probability measure
    $\mu_{glob}$
    which is the uniform measure on the segment
    $[-(1-\al)\frac{\pi}{\phi},\al \frac{\pi}{\phi}]$.
    Indeed, this $\mu_{glob}$ is a limit
    as in \eqref{mu_loc_mu_glob_convergence}
    of random atomic measures
    corresponding to the sine process initial data,
    where as $N$ one should take the random number of particles
    $\NRandom$ (it is concentrated around $\frac{\phi}{\pi}M$).
    Similarly,
    the constant
    $p\log(\frac{\al}{1-\al})$
    in \Cref{Theorem_Bernoulli_IC} corresponds to $\mu_{glob}$
    being the uniform measure on
    $[-(1-\al)p^{-1},\al p^{-1}]$.
\end{remark}

\begin{remark}
    The proof of \Cref{Proposition_sine_IC}
    carries over from the Bernoulli case
    modulo two
    estimates of the
    variance \eqref{variance_estimate} and
    \eqref{variance_estimate2}--\eqref{variance_estimate3},
    which are rather straightforward
    for the sine process.
    Thus, the bulk limit theorem should hold
    for rather general random initial data,
    but we will not formulate any other
    results in this direction.
\end{remark}



\appendix

\section{Determinantal kernels for other noncolliding processes} 
\label{sec:determinantal_kernels_for_other_noncolliding_dynamics}

\subsection{Noncolliding Poisson random walk} 
\label{sub:poisson_case}

Taking the limit as $\be\to0$ and scaling to the continuous time as $t=\fl{\be^{-1}\tau}$, $\tau\in\R_{\ge0}$, turns the noncolliding Bernoulli random walk into the noncolliding Poisson random walk --- the continuous time dynamics of $N$ independent speed $1$ Poisson particles conditioned to never collide \cite{konig2002non}. This Markov chain $\vec X(\tau)$ on $\W N$ has jump rates (cf. \eqref{noncolliding_Bernoulli_walks_transitions})
\begin{equation*}
    \PP(\vec X(\tau+d\tau)=\vec x'\mid\vec X(\tau)=\vec x)=
    \begin{cases}
        \dfrac{\V(\vec x')}{\V(\vec x)}d\tau +O(d\tau^2),&
        \textnormal{\parbox{.22\textwidth}{$x_i'=x_i+1$
        for some $i$, and $x_j'=x_j$ for $j\ne i$;}}\\
        \rule{0pt}{14pt}1-N d\tau+O(d\tau^{2}),& \vec x'=\vec x;\\
        \rule{0pt}{11pt}0,&\textnormal{otherwise},
    \end{cases}
\end{equation*}
where $\vec x,\vec x'\in\W N$ are arbitrary.\footnote{In particular, this implies $\sum_{i=1}^{N}\V(\vec x+e_i)=N \V(\vec x)$, where $e_i$ is the $i$-th basis vector $(0,\ldots,0,1,0,\ldots,0)$.} The noncolliding Poisson random walk is also sometimes referred to as the Charlier process \cite{BorFerr2008DF} due to the fact that if it starts from the densely packed initial configuration $(0,1,\ldots,N-1)$, then its (fixed time) distribution is the Charlier orthogonal polynomial ensemble (cf. \Cref{rmk:Krawtchouk}).

\begin{theorem}\label{thm:Poisson_kernel}
    The noncolliding Poisson random walk $\vec X(\tau)$ started from an arbitrary initial configuration $\vec a\in\W N$ is determinantal in the sense of \eqref{determinantal_kernel}, with the kernel
    \begin{multline}\label{K_Poisson}
        K^{\textnormal{Poisson}}_{\vec a}(\tau_1,x_1;\tau_2,x_2)=
        -
        \mathbf{1}_{x_1\ge x_2}\mathbf1_{\tau_1>\tau_2}
        \frac{(\tau_1- \tau_2)^{x_1-x_2}}{(x_1-x_2)!}
        \\-
        \frac1{(2\pi\i)^{2}}
        \int\limits_{x_2-\frac12-\i\infty}^{x_2-\frac12+\i\infty}dz
        \oint\limits_{\textnormal{all $w$ poles}}dw
        \,\frac{1}{w-z}
        \frac{\G(x_2-z)}{\G(x_1-w+1)}
        \tau_1^{x_1-w}\tau_2^{z-x_2}
        \prod_{r=1}^{N}\frac{z-a_r}{w-a_r},
    \end{multline}
    where $x_{1,2}\in\Z$, $\tau_{1,2}>0$, the $z$ integration contour is a vertical line $\Re z=x_2-\frac12$ traversed upwards, and the $w$ contour is a positively oriented circle or a union of two circles encircling all the $w$ poles $\{\ldots,x_1-1,x_1\}\cap\{a_1,\ldots,a_N\}$ of the integrand except $w=z$.
\end{theorem}
\begin{proof}
    We will obtain $K^{\textnormal{Poisson}}_{\vec a}$ from $K^{\textnormal{Bernoulli}}_{\vec a; \be}$ \eqref{K_Bernoulli} via the $\be\to0$ limit described above.
    Employing \Cref{rmk:moving_int_contours}, write $K^{\textnormal{Bernoulli}}_{\vec a; \be}$ as
    \begin{multline}
        K^{\textnormal{Bernoulli}}_{\vec a; \be}(t_1,x_1;t_2,x_2)
        =
        \mathbf{1}_{x_1\ge x_2}\mathbf{1}_{t_1>t_2}
        (-1)^{x_1-x_2+1}
        \binom{t_1-t_2}{x_1-x_2}
        +
        \frac{t_1!}{(t_2-1)!}\frac1{(2\pi\i)^{2}}\\\times
        \int\limits_{x_2-\frac12-\i\infty}^{x_2-\frac12+\i\infty}dz
        \oint\limits_{\textnormal{all $w$ poles}}dw
        \frac{(z-x_2+1)_{t_2-1}}{(w-x_1)_{t_1+1}}
        \frac{1}{w-z}
        \frac{\sin(\pi w)}{\sin(\pi z)}
        \left(\frac{1-\be}{\be}\right)^{w-z}
        \prod_{r=1}^{N}\frac{z-a_r}{w-a_r}.
        \label{K_Bernoulli_in_Poisson_proof}
    \end{multline}
    Here the $w$ contour (a circle or a union of two circles) can be
    taken to encircle the points $\{a_1,a_2,\ldots,a_N\}$, which
    contain all the $w$ poles except $w=z$. (Indeed, for $w=a_i$ to
    be a pole, it must additionally satisfy $(a_i-x_1)_{t_1+1}=0$ to
    not cancel with the zero coming from $\sin(\pi w)$.) Therefore,
    the integration contours do not depend on $t_{1,2}$, and we can
    take the Poisson rescaling of
    \eqref{K_Bernoulli_in_Poisson_proof}, that is, $\be\to0$ and
    $t_{1,2}=\fl{\be^{-1}\tau_{1,2}}$ with $\tau_{1}\ge0$, $\tau_2>0$.

    The $t$-dependent part of the first summand
    in \eqref{K_Bernoulli_in_Poisson_proof} scales as
    \begin{multline*}
        \mathbf1_{\tau_1>\tau_2}
        \binom{\fl{\be^{-1}\tau_{1}}-\fl{\be^{-1}\tau_{2}}}{x_1-x_2}
        =\mathbf1_{\tau_1>\tau_2}
        \frac{\G(\fl{\be^{-1}\tau_{1}}-\fl{\be^{-1}\tau_{2}}+1)}{
        \G(\fl{\be^{-1}\tau_{1}}-\fl{\be^{-1}\tau_{2}}-x_1+x_2+1)
        (x_1-x_2)!
        }
        \\\sim
        \mathbf1_{\tau_1>\tau_2}
        \frac{(\tau_1- \tau_2)^{x_1-x_2}}{(x_1-x_2)!}\,\be^{-(x_1-x_2)},
    \end{multline*}
    where we used \eqref{gamma_asymptotics}. Similarly,
    for the part of the integrand depending
    on $\be$ and $t_{1,2}$ we have
    \begin{multline*}
        \frac{t_1!}{(t_2-1)!}
        \frac{(z-x_2+1)_{t_2-1}}{(w-x_1)_{t_1+1}}
        \left(\frac{1-\be}{\be}\right)^{w-z}
        \frac{\sin(\pi w)}{\sin(\pi z)}
        \\=
        \frac{\G(t_1+1)}{\G(w-x_1+t_1+1)}\frac{\G(z-x_2+t_2)}{\G(t_2)}
        \frac{\G(w-x_1)}{\G(z-x_2+1)}
        \left(\frac{1-\be}{\be}\right)^{w-z}
        \frac{\sin(\pi w)}{\sin(\pi z)}
        \\\sim
        \frac{\G(x_2-z)}{\G(x_1-w+1)}(-1)^{x_1-x_2+1}
        {\tau_1}^{x_1-w}
        {\tau_2}^{z-x_2}\be^{-(x_1-x_2)}.
    \end{multline*}
    Note that $\left(1-\be\right)^{w-z}\to1$ as $\be\to0$,
    and in the last step we used \eqref{Gamma_function_flip}.
    This implies
    \begin{equation*}
        \lim_{\be\to0}
        \Big((-\be)^{x_1-x_2}
        K^{\textnormal{Bernoulli}}_{\vec a; \be}(\fl{\be^{-1}\tau_1},x_1
        ;\fl{\be^{-1}\tau_1},x_2)\Big)
        =K^{\textnormal{Poisson}}_{\vec a}(\tau_1,x_1;\tau_2,x_2),
    \end{equation*}
    where $K^{\textnormal{Poisson}}_{\vec a}$
    is given by \eqref{K_Poisson}.
    Because the multiplication by $(-\be)^{x_1-x_2}$ does not change the
    correlation functions
    (cf. footnote${}^{\textnormal{\ref{gaugefootnote}}}$),
    this completes the proof.
\end{proof}

The correlation kernel of \Cref{thm:Poisson_kernel} appears to be new. By analogy with the results in \Cref{sub:main_result_bulk_limit_theorems}, we believe that the local statistics of the noncolliding Poisson random walk are universally described by an extension of the discrete sine kernel with the continuous time parameter. This extension first appeared in \cite{borodin2006stochastic}, see also \cite{borodin2007periodic}, \cite{borodin2010gibbs} for a general discussion of extensions of the discrete sine kernel. We will not pursue this in the present paper.


\subsection{Dyson Brownian Motion} 
\label{sub:dyson_brownian_motion}

A diffusion scaling brings the kernel
$K^{\textnormal{Bernoulli}}_{\vec a; \be}$
\eqref{K_Bernoulli} to the kernel of the Dyson Brownian Motion.
We will use the following scaling (where $M\to+\infty$):
\begin{equation*}
    t_{1,2}=\fl{M\tau_{1,2}},\qquad
    x_{1,2}=\fl{\be M\tau_{1,2}+\xi_{1,2}{\sqrt{\be(1-\be)}}\sqrt{M}},\qquad
    a_i=\fl{\al_{i}\sqrt{\be(1-\be)}\sqrt{M}},
\end{equation*}
where $(\al_1\le\ldots\le\al_N)\in\R^{N}$ are the rescaled starting points (by agreement,
when some of the $\al_i$'s
coincide, the corresponding $a_i$'s differ by $1$,
so that the discrete
noncolliding Bernoulli random walk
is well-defined).

\begin{theorem}\label{thm:Dyson_convergence}
    Under the above scaling and up to a gauge transformation as in footnote${}^{\textnormal{\ref{gaugefootnote}}}$, the kernel $(M\be(1-\be))^{\frac12}K^{\textnormal{Bernoulli}}_{\vec a; \be}$ converges to the following kernel:
    \begin{multline}\label{Dyson_kernel}
        K^{\textnormal{DBM}}_{\vec a; \be}(\tau_1,\xi_1;\tau_2,\xi_2)
        =
        -
        \frac{\mathbf{1}_{\tau_1>\tau_2}}{\sqrt{2\pi\D \tau}}
        \exp\bigg\{
        -\frac{(\D\xi)^{2}}{2\D \tau}
        \bigg\}
        \\-
        \frac{1}{(2\pi\i)^{2}\sqrt{\tau_1\tau_2}}
        \int\limits_{c-\i\infty}^{c+\i\infty}dz
        \oint\limits_{\textnormal{all $w$ poles}}dw\,\frac{1}{w-z}
        \exp\bigg\{
        \frac{\tau_1(z-\xi_2)^{2}-\tau_2(w-\xi_1)^{2}}{2\tau_1\tau_2}\bigg\}
        \prod_{r=1}^{N}\frac{z-\al_r}{w-\al_r},
    \end{multline}
    where $\xi_{1,2}\in\R$, $\tau_{1,2}>0$, and we use the notation $\D\xi=\xi_1-\xi_2$, $\D\tau=\tau_1-\tau_2$. The $z$ contour is a vertical line which lies to the left of all the $\al_r$'s (i.e., $c<\al_1$), and the $w$ contour is a positively oriented circle encircling all the $\al_r$'s.
\end{theorem}
The multiplication by $(M\be(1-\be))^{\frac12}$ corresponds to the rescaling of the space from discrete to continuous (the correlation kernel should be viewed as a kernel of an integral operator). Note also that the kernel \eqref{Dyson_kernel} can similarly be obtained as a diffusion limit of the Poisson kernel \eqref{K_Poisson}, but we will not perform this computation.
\begin{proof}[Proof of \Cref{thm:Dyson_convergence}]
    Let us denote $\sigma=\sqrt{\be(1-\be)}$ to shorten the notation. Changing the variables as $z=\tilde z \sigma\mm$, $w=\tilde w \sigma\mm$, and renaming back to $z,w$, we can rewrite \eqref{K_Bernoulli} as
    \begin{multline}
        K^{\textnormal{Bernoulli}}_{\vec a; \be}(t_1,x_1;t_2,x_2)
        =
        \mathbf{1}_{x_1\ge x_2}\mathbf{1}_{t_1>t_2}
        (-1)^{x_1-x_2+1}
        \binom{t_1-t_2}{x_1-x_2}
        \\+
        \sigma\mm\frac{t_1!}{(t_2-1)!}\frac1{(2\pi\i)^{2}}
        \int\limits_{(x_2-t_2+\frac12)/(\sigma\mm)-\i\infty}
        ^{(x_2-t_2+\frac12)/(\sigma\mm)+\i\infty}dz
        \oint\limits_{\textnormal{all $w$ poles}}dw\,
        \frac{(z \sigma\mm-x_2+1)_{t_2-1}}{(w \sigma\mm-x_1)_{t_1+1}}
        \\\times\frac{1}{w-z}
        \frac{\sin(\pi w \sigma\mm)}{\sin(\pi z \sigma\mm)}
        \left(\frac{1-\be}{\be}\right)^{(w-z)\sigma\mm}
        \prod_{r=1}^{N}\frac{z-\fl{\al_r \sigma\mm}/(\sigma\mm)}{w-\fl{\al_r \sigma\mm}/(\sigma\mm)}.
        \label{K_Bernoulli_to_Dyson_limit}
    \end{multline}
    The $w$ contour encircles all the $w$ poles of the integrand except $w=z$, which are close to the
    points $(\al_1\le \ldots\le \al_N)$. For large $M$ the $z$ contour will get shifted further to the left;
    the computation below shows that there will be no poles crossed while doing this.

    Denote
    \begin{equation*}
        c_{1,2}=t_{1,2}-M \tau_{1,2},\qquad
        d_{1,2}=x_{1,2}-
        \be M\tau_{1,2}-\xi_{1,2}\sigma\mm,
    \end{equation*}
    these are numbers between $-1$ and $0$.
    Let us first consider asymptotics of the
    non-integral summand. For large $M$, the two indicators reduce simply to
    $\mathbf{1}_{\tau_1>\tau_2}$, and the binomial coefficient has the asymptotics
    \begin{multline*}
        \binom{t_1-t_2}{x_1-x_2}=\frac{\G(t_1-t_2+1)}{\G(x_1-x_2+1)\G(t_1-t_2-(x_1-x_2)+1)}
        =\frac{1/\mm}{\sqrt{2\pi \sigma\D\tau}}
        \left(\frac{1-\be}{\be}\right)^{\sigma\mm\D\xi}
        \\\times
        e^{-M\D \tau(\be\log\be+(1-\be)\log(1-\be))}
        \be^{d_2-d_1} (1-\be)^{c_2-c_1+d_1-d_2}
        e^{-(\D \xi)^{2}/(2\D \tau)}
        \big(1+O(1/\mm)\big)
    \end{multline*}
    where we used \eqref{Stirling_Gamma}.

    Now consider the asymptotics of various parts of the integrand. We have
    \begin{equation*}
        \frac{t_1!}{(t_2-1)!}=M
        (M\tau_1)^{M \tau_1+c_1} (M \tau_2)^{-M \tau_2-c_2} e^{-M\D\tau}
        \sqrt{\tau_1 \tau_2}
        \big(1+O(1/M)\big).
    \end{equation*}
    We can also write
    \begin{multline*}
        \frac{(z \sigma\mm-x_2+1)_{t_2-1}}{(w \sigma\mm-x_1)_{t_1+1}}
        =
        \frac{\G(z \sigma\mm-x_2+t_2)}{\G(z \sigma\mm-x_2+1)}
        \frac{\G(w \sigma\mm-x_1)}{\G(w \sigma\mm-x_1+t_1+1)}
        \\=(-1)^{x_1-x_2+1}\frac{\G(z \sigma\mm-x_2+t_2)}{\G(w \sigma\mm-x_1+t_1+1)}
        \frac{\G(-z \sigma\mm+x_2)}{\G(1-w \sigma\mm+x_1)}
        \frac{\sin(\pi z \sigma\mm)}{\sin(\pi w \sigma\mm)},
    \end{multline*}
    where we used \eqref{Gamma_function_flip}. This cancels with the existing ratio of the sine functions.
    Continuing with the asymptotics, we obtain
    \begin{multline*}
        \frac{\G(z \sigma\mm-x_2+t_2)}{\G(w \sigma\mm-x_1+t_1+1)}
        \frac{\G(-z \sigma\mm+x_2)}{\G(1-w \sigma\mm+x_1)}=M^{-2}
        (M\tau_1)^{-M \tau_1-c_1} (M \tau_2)^{M \tau_2+c_2} e^{M\D\tau}
        \\\times
        e^{-M\D \tau(\be\log\be+(1-\be)\log(1-\be))}
        \be^{d_2-d_1} (1-\be)^{c_2-c_1+d_1-d_2}
        \left(\frac{1-\be}{\be}\right)^{\sigma\mm(\D\xi+z-w)}
        \\\times
        \frac{1}{\be(1-\be)\tau_1\tau_2}
        \exp\bigg\{\frac{\tau_1(z-\xi_2)^{2}-\tau_2(w-\xi_1)^{2}}{2\tau_1\tau_2}\bigg\}
        \big(1+O(1/\mm)\big).
    \end{multline*}

    We see that the factor
    \begin{equation*}
        (-1)^{x_1-x_2}e^{-M\D \tau(\be\log\be+(1-\be)\log(1-\be))}
        \be^{d_2-d_1} (1-\be)^{c_2-c_1+d_1-d_2}
        \left(\frac{1-\be}{\be}\right)^{\sigma\mm\D\xi}
    \end{equation*}
    appearing in both summands in \eqref{K_Bernoulli_to_Dyson_limit} is of the form
    $f(\tau_1,\xi_1)/f(\tau_2,\xi_2)$, and thus does not affect
    the correlation functions
    (cf. footnote${}^{\textnormal{\ref{gaugefootnote}}}$).
    The remaining parts of \eqref{K_Bernoulli_to_Dyson_limit}
    multiplied by $\sigma\mm$
    converge to \eqref{Dyson_kernel}, as desired.
\end{proof}

Since the noncolliding Bernoulli random walks converge under the diffusion scaling to the Dyson Brownian Motion \cite{eichelsbacher2008ordered}, the kernel \eqref{Dyson_kernel} is the correlation kernel for the latter process started from the arbitrary initial configuration $(\al_1\le\ldots\le\al_N)$. When $\tau_1=\tau_2$, the kernel \eqref{Dyson_kernel} turns into the one appeared in \cite{brezin1997extension}, \cite{Johansson2001Universality}, \cite{Shcherbina2008universality}. Utilizing \eqref{Dyson_kernel}, these papers show that the local statistics of the eigenvalues of the deformed GUE ensemble (equivalently, the distribution of the Dyson Brownian Motion started from $(\al_1\le\ldots\le\al_N)$ at a fixed time) are universally governed by the continuous sine kernel.



\section{Representation-theoretic interpretation of noncolliding walks} 
\label{sub:representation_theoretic_interpretation}

The random matrix analogue of noncolliding Bernoulli or Poisson random walks is the GUE Dyson Brownian Motion. The distribution of the Dyson Brownian Motion started from an arbitrary initial configuration $(\al_1\le\ldots\le \al_N)\in\R^{N}$ after time $t>0$ can also be interpreted as the eigenvalue distribution of the deformed GUE ensemble $A+\sqrt{t}G$, where $A=\mathop{\mathrm{diag}}(\al_1,\ldots,\al_N)$ is a fixed diagonal matrix, and $G$ is an $N\times N$ random matrix from the GUE.\footnote{There are different normalizations of the GUE random matrices and Dyson Brownian Motion in the literature (cf. \cite[(2.2.2)]{AndersonGuionnetZeitouniBook}, \cite[Example 2]{TaoVu2012Survey}), and here we assume that they agree.} Let us discuss a similar interpretation of the noncolliding Bernoulli or Poisson random walks in terms of representation theory of unitary groups.

Irreducible representations of the unitary group $U(N)$ can be parametrized by points of $\W N$. Let $\ch_{\vec x}(u_1,\ldots,u_N)$, $\vec x\in\W N$ denote the corresponding normalized irreducible characters. Here the $u_i$'s are eigenvalues of the matrix from $U(N)$, and the characters $\ch_{\vec x}$ are normalized in the sense that $\ch_{\vec x}(1,\ldots,1)=1$. These characters are the normalized Schur polynomials:
\begin{equation*}
    \ch_{\vec x}(u_1,\ldots,u_N)
    =\frac{s_\la(u_1,\ldots,u_N)}{\dim_N\vec x},
\end{equation*}
where $\la=(\la_1\ge \ldots \ge \la_N)$ with $\la_i=x_{N+1-i}+i-N$ is the highest weight of the representation, and
\begin{equation*}
    \dim_N \vec x=s_\la(1,\ldots,1)
    =\prod_{1\le i<j\le N}\frac{x_j-x_i}{j-i}=
    \frac{\V(\vec x)}{\V(0,1,\ldots,N-1)}
\end{equation*}
is the dimension of this representation. Details on representations of unitary groups can be found in, e.g., \cite{Weyl1946}.

An (abstract) normalized character $\ch$ of $U(N)$ is defined as a nonnegative definite continuous function on $U(N)$ which satisfies $\ch(e)=1$ and $\ch(ab)=\ch(ba)$ for any $a,b\in U(N)$ (that is, we speak about characters which do not necessarily correspond to actual representations). The set of such characters is convex, and the normalized irreducible characters are its extreme points. Thus, any abstract character can be decomposed into irreducibles as $\ch(u_1,\ldots,u_N)=\sum_{\vec x\in\W N}c_{\vec x}\,\ch_{\vec x}(u_1,\ldots,u_N)$, where the numbers $\{c_{\vec x}\}_{\vec x\in \W N}$ are nonnegative and sum to one, hence they define a probability distribution on $\W N$.

The product of two normalized characters $\ch^{(1)}$ and $\ch^{(2)}$ is also a normalized character. (If both $\ch^{(1)}$ and $\ch^{(2)}$ correspond to actual representations, then $\ch^{(1)}\ch^{(2)}$ is the normalized character of the tensor product of these representations.) The product $\ch^{(1)}\ch^{(2)}$ then can be decomposed into irreducibles, thus yielding a probability distribution on $\W N$.

Fix $\vec a\in\W N$ and take the irreducible normalized character $\ch_{\vec a}$ as $\ch^{(1)}$. Let $\ch^{(2)}$ be the restriction to $U(N)$ of a certain extreme character of the infinite-dimensional unitary group $U(\infty)$. We will consider two classes of such characters of $U(N)$ having the form
\begin{equation*}
    \ch_{\be;t}(u_1,\ldots,u_N)=\prod_{r=1}^{N}(1-\be+\be u_r)^{t}
    \qquad\textnormal{or}\qquad
    \ch_{\tau}(u_1,\ldots,u_N)=
    \prod_{r=1}^{N}e^{\tau(u_r-1)},
\end{equation*}
where $t\in\Z_{\ge0}$ and $\tau>0$. There is a number of papers discussing classification of extreme characters of $U(\infty)$, e.g., see \cite{Edrei53}, \cite{Voiculescu1976}, \cite{VK82CharactersU}, \cite{OkOl1998}, and other references in \cite{BorodinOlsh2011GT} and \cite{Petrov2012GT}.

\begin{proposition}\label{prop:rep_th}
    The probability weights $\{c_{\vec x}\}_{\vec x\in\W N}$ arising from the decomposition
    \begin{equation*}
        \ch_{\vec a}(u_1,\ldots,u_N)
        \ch_{\be;t}(u_1,\ldots,u_N)=
        \sum_{\vec x\in\W N}c_{\vec x}\,\ch_{\vec x}(u_1,\ldots,u_N)
    \end{equation*}
    describe the distribution of the noncolliding Bernoulli random walk with parameter $\be$ started from the initial configuration $\vec a$ after $t$ steps.

    Similarly, the decomposition of $\ch_{\vec a}(u_1,\ldots,u_N) \ch_{\tau}(u_1,\ldots,u_N)$ into irreducibles cor\-res\-ponds to the distribution of the noncolliding Poisson random walk started from the configuration $\vec a$ after time $\tau$.
\end{proposition}
The case $\vec a=(0,1,\ldots,N-1)$ leads to the trivial representation: $\ch_{\vec a}(u_1,\ldots,u_N)\equiv 1$. In this case the distribution of the noncolliding Bernoulli or Poisson random walks is related to the decomposition of extreme characters of $U(\infty)$ into irreducibles. Probabilistic properties of the corresponding measures were studied in, e.g., \cite{BorodinKuan2007U}, \cite{BorFerr2008DF}. These measures can be regarded as discrete analogues of the GUE eigenvalue distribution.
\begin{proof}[Proof of \Cref{prop:rep_th}]
    This fact is well known to specialists, but we include its proof for completeness.

    It suffices to consider only the Bernoulli $t=1$ case, because the general $t$ case follows by induction, and the Poisson statement follows by a simple limit transition. The result would follow if we interpret the coefficients $c_\mu$ in the decomposition
    \begin{equation*}
        \frac{s_\la(u_1,\ldots,u_N)}{s_\la(1,\ldots,1)}
        \prod_{r=1}^{N}(\be u_r+1-\be)
        =\sum_{\mu_1\ge \ldots\ge\mu_N}c_\mu
        \frac{s_\mu(u_1,\ldots,u_N)}{s_\mu(1,\ldots,1)},
        \qquad \la_i=a_{N+1-i}+i-N,
    \end{equation*}
    as one-step transition probabilities \eqref{noncolliding_Bernoulli_walks_transitions}. Multiply the above decomposition by $s_\la(1,\ldots,1)$ and the Vandermonde in the $u_i$'s, and expand the determinants in
    \begin{equation*}
        \det[u_i^{\la_j+N-j}]_{i,j=1}^{N}
        \prod_{r=1}^{N}(\be u_r+1-\be)
        =\sum_{\mu_1\ge \ldots\ge\mu_N}c_\mu
        \det[u_i^{\mu_j+N-j}]_{i,j=1}^{N}
        \frac{s_\la(1,\ldots,1)}{s_\mu(1,\ldots,1)}.
    \end{equation*}
    Because of the ordering in $\la$ and $\mu$, it suffices to consider the coefficient by $u_1^{\mu_1+N-1}\ldots u_N^{\mu_N}$ in $u_1^{\la_1+N-1}u_2^{\la_2+N-2}\ldots u_N^{\la_N}$ multiplied by $\prod_{r=1}^{N}(\be u_r+1-\be)$. Clearly, picking $\be u_r$ from each $r$th factor corresponds to the $r$th particle jumping by one to the right, while picking $(1-\be)$ means that this particle stays put. Particle collisions are not allowed because $\mu_1+N-1> \ldots>\mu_N$, and the factor ${s_\la(1,\ldots,1)}/{s_\mu(1,\ldots,1)}$ can be identified with the ratio of the Vandermondes in \eqref{noncolliding_Bernoulli_walks_transitions}.

    The Poisson case follows from the above argument in the limit as $\be\searrow0$.
\end{proof}

We see that tensor multiplication of representations (and, more generally, multiplication of normalized characters) is a discrete analogue of the matrix addition. Moreover, multiplying by a suitable extreme character of $U(\infty)$ corresponds to adding a multiple of the GUE matrix. This similarity can be continued further to include the operation of the free convolution --- its discrete analogue is the so-called quantized free convolution, see \cite{GorinBufetov2013free}.

Therefore, our main universality results in \Cref{sub:main_result_bulk_limit_theorems} can be reformulated as bulk universality for tensor products of two representations of $U(N)$, when one of the factors is arbitrary, and the other factor is the specific representation $\ch_{\be;t}$ (with $t$ large). We \emph{conjecture} that under mild technical conditions the same bulk universality should hold for tensor products of two arbitrary representations of $U(N)$ (cf. \cite{bao2015local} for a progress towards a similar random matrix result). A weaker version of the bulk universality for tensor products of two arbitrary representations is established in \cite{Gorin2016}.



\printbibliography

\end{document}